\documentclass[reqno]{amsart}

\usepackage{amsmath,amssymb,amsthm,mathtools}
\usepackage{bm}
\usepackage{enumitem}
\usepackage{graphicx}
\usepackage{booktabs}
\usepackage[hidelinks]{hyperref}
\usepackage[capitalize,noabbrev]{cleveref}
\graphicspath{{new_implementation/figures/}{../Implementation/data/}{Implementation/figures/}{Implementation/data/}}

\IfFileExists{algorithm.sty}{\usepackage{algorithm}}{%
  \makeatletter
  \newcommand{\ALGgobblecap}[1]{}
  \newcounter{algorithm}[section]
  \renewcommand{\thealgorithm}{\thesection.\arabic{algorithm}}
  \newenvironment{algorithm}[1][]{%
    \par\addvspace{\medskipamount}%
    \refstepcounter{algorithm}%
    \noindent\begin{minipage}{\linewidth}%
    \noindent\rule{\linewidth}{0.8pt}\par\smallskip
    \noindent\textbf{Algorithm~\thealgorithm.}\par
    \nobreak\smallskip
    \noindent\rule{\linewidth}{0.4pt}\par\nobreak\smallskip
    \let\caption\ALGgobblecap
  }{%
    \par\smallskip\noindent\rule{\linewidth}{0.8pt}%
    \end{minipage}\par\addvspace{\medskipamount}%
  }
  \makeatother
}
\IfFileExists{algpseudocode.sty}{\usepackage{algpseudocode}}{%
  \makeatletter
  \newcounter{ALG@line}
  \newenvironment{algorithmic}[1][0]{%
    \setcounter{ALG@line}{0}%
    \par\addvspace{0.3em}\noindent
    \list{\refstepcounter{ALG@line}\arabic{ALG@line}:}{%
      \setlength{\leftmargin}{2.5em}%
      \setlength{\itemsep}{3pt}\setlength{\parsep}{0pt}%
      \setlength{\topsep}{0pt}\setlength{\partopsep}{0pt}%
      \setlength{\labelwidth}{1.8em}\setlength{\labelsep}{0.5em}%
    }%
  }{\endlist\par\addvspace{0.3em}}
  \newcommand{\State}{\item}
  \newcommand{\Require}{\item[\textbf{Require:}]}
  \newcommand{\Ensure}{\item[\textbf{Ensure:}]}
  \newcommand{\Return}{\textbf{return}\space}
  \newenvironment{ALG@forblock}{%
    \list{\refstepcounter{ALG@line}\arabic{ALG@line}:}{%
      \setlength{\leftmargin}{1.5em}%
      \setlength{\itemsep}{3pt}\setlength{\parsep}{0pt}%
      \setlength{\topsep}{0pt}\setlength{\partopsep}{0pt}%
      \setlength{\labelwidth}{1.8em}\setlength{\labelsep}{0.5em}%
    }%
  }{\endlist}
  \newcommand{\For}[1]{\item \textbf{for}~##1~\textbf{do}\begin{ALG@forblock}}
  \newcommand{\EndFor}{\end{ALG@forblock}\item \textbf{end for}}
  \makeatother
}

\newcommand{\R}{\mathbb{R}}

\newcommand{\N}{\mathbb{N}}
\newcommand{\C}{\mathbb{C}}
\newcommand{\T}{\mathbb{T}}
\newcommand{\Sone}{S^{1}}
\newcommand{\SO}[1]{\mathrm{SO}(#1)}

\newcommand{\SU}[1]{\mathrm{SU}(#1)}
\newcommand{\Ad}{\mathrm{Ad}}

\newcommand{\Lie}{\mathfrak{g}}
\newcommand{\Ham}{\mathcal{H}}
\newcommand{\val}{V}
\newcommand{\dist}{\mathrm{dist}}
\newcommand{\eps}{\varepsilon}
\newcommand{\inner}[2]{\langle #1,\,#2\rangle}
\newcommand{\norm}[1]{\left\lVert #1\right\rVert}
\newcommand{\Lap}{\Delta}
\DeclareMathOperator{\Tr}{tr}
\DeclareMathOperator{\Lip}{Lip}

\theoremstyle{plain}
\newtheorem{theorem}{Theorem}[section]
\newtheorem{proposition}[theorem]{Proposition}
\newtheorem{lemma}[theorem]{Lemma}
\newtheorem{corollary}[theorem]{Corollary}

\theoremstyle{definition}

\newtheorem{assumption}[theorem]{Assumption}

\theoremstyle{remark}
\newtheorem{remark}[theorem]{Remark}

\crefname{assumption}{Assumption}{Assumptions}
\Crefname{assumption}{Assumption}{Assumptions}
\crefname{definition}{Definition}{Definitions}
\Crefname{definition}{Definition}{Definitions}
\crefname{algorithm}{Algorithm}{Algorithms}
\Crefname{algorithm}{Algorithm}{Algorithms}

\title[Structure-Preserving Spectral Dynamic Programming]
{Structure-Preserving Spectral Dynamic Programming on Compact Lie Groups}

\author{Shanqing Liu}
\address{Division of Applied Mathematics, Brown University}
\email{Shanqing\_Liu@brown.edu} 

\author{Yang Qi}
\address{Inria Saclay Ile-de-France and CMAP, Ecole polytechnique, IP Paris, CNRS}
\email{Yang.Qi@inria.fr}

\subjclass[2020]{49L25, 49L20, 49N35, 65M70, 65M12, 43A77}
\keywords{Hamilton-Jacobi-Bellman equation, optimal control, dynamic programming, Lax-Oleinik semigroup, compact Lie group, Peter-Weyl theorem, viscosity solution, spectral method, vanishing viscosity}

\date{\today}

\begin{document}

\begin{abstract}

We study spectral approximations of the dynamic programming semigroup of a finite-horizon optimal control problem on a connected compact Lie group $G$, and of the associated first-order Hamilton-Jacobi-Bellman equation. 
The Bellman operator is monotone and non-expansive in the supremum norm, while the Peter-Weyl decomposition of $L^{2}(G)$, on which every Fourier method on $G$ rests, is orthogonal, and the mismatch is quantitative. 
The natural sup-norm error recursion of the Galerkin iteration is amplified at every step by the Lebesgue constant of the spectral projection, which grows logarithmically on $S^{1}$ and polynomially on compact Lie groups of rank one, including $\mathrm{SO}(3)$, and in computation the iteration violates elementary bounds within a few steps. 
We restore the dynamic programming structure at the discrete level by replacing the orthogonal projection with spectral filters of Markov type. 
A first, auxiliary route combines a heat-kernel filter with a vanishing-viscosity regularization and yields qualitative convergence in the supremum norm for Lipschitz data. 
The main result is a Fej\'er-type filter on $G$, simultaneously finite-rank, positivity preserving and non-expansive, together with a convergence theorem at the rate $O(\sqrt{\delta}+\sqrt{\eps+(\delta N^{2})^{-1}}\,)$ for Lipschitz data, where $\delta$ is the time step, $N$ the spectral resolution and $\eps$ the viscosity. 
The viscosity may be set to zero, and the coupling $\delta=N^{-1}$ then gives the rate $N^{-1/2}$. 
The proof interprets the filter as a small random perturbation of the controlled dynamics, requires neither a priori regularity of the value function nor a consistency argument in the viscosity sense for the filtering step, and extends to a fully discrete realization based on positive cubature, with exact Wigner transport on $\mathrm{SO}(3)$. 
Numerical experiments are consistent with the predicted rates and with the predicted size of the filter bias, and quantify the frame dependence of the two chart-based baselines tested here.
\end{abstract}

\maketitle

\section{Introduction}

\subsection{Motivation and context}

This paper is concerned with the numerical solution of a deterministic optimal control problem with finite horizon, whose state evolves on a connected compact Lie group $G$. 
By the dynamic programming principle \cite{Bellman1957,FlemingSoner2006,BardiCapuzzoDolcetta1997}, the value function of such a problem is characterized as the unique viscosity solution of a first-order Hamilton-Jacobi-Bellman (HJB) equation on $G$. 
Problems of this type arise in the control of mechanical systems with rotational degrees of freedom, for instance in the attitude control of rigid bodies and spacecraft, where the state space is the rotation group $\SO{3}$~\cite{Jurdjevic1997,BulloLewis2005}, and in the control of closed quantum systems, where the evolution operator moves on the special unitary group $\SU{n}$~\cite{DAlessandro2021}. 
Computing the value function globally on the group, rather than a single optimal trajectory, yields an optimal feedback valid for every initial state, which is the relevant object when many queries must be answered or when a stabilizing control law is sought. 
Working directly on $G$ also avoids the artifacts of local coordinates, such as the singularities of Euler angle charts or the double covering of the quaternion parameterization.

A discretization of such a problem may aim at preserving two structures of a rather different nature. 
The first is geometric. The group is a curved space with a canonical family of symmetries, a problem posed on it has no preferred axes, and a scheme formulated in a chart introduces axes of its own. 
At the level of trajectories this requirement is by now classical, and the Lie-group integrators of geometric numerical integration provide one-step maps that remain on the group exactly \cite{IserlesMuntheKaasNorsettZanna2000,HairerLubichWanner2006,HallLeok2017}. 
The flow maps used in this paper are of that kind. 
The second structure is order-theoretic. 
Under the dynamic programming principle the value functions form a nonlinear semigroup, known for convex Hamiltonians as the Lax-Oleinik semigroup~\cite{Fathi2008}, and every operator of this semigroup is monotone and non-expansive in the supremum norm. 
These two properties are the discrete counterparts of the comparison principle, and they are precisely the hypotheses on which the convergence framework of Barles and Souganidis for approximation schemes of viscosity solutions rests \cite{BarlesSouganidis1991}. 
A discretization that preserves the geometry but destroys the order structure therefore loses the standard route to a convergence proof. 
We are not aware of a previous discretization of HJB equations on a noncommutative compact group that is intrinsic and comes with a stability and convergence theory in the norm in which viscosity solutions are posed and compared. 
Providing one is the purpose of this paper.

The numerical analysis of HJB equations has been developed mostly in the Euclidean setting. Available methods include monotone finite difference schemes~\cite{CrandallLions1984}, for which error bounds in the supremum norm are available in the stationary and parabolic settings~\cite{BarlesJakobsen2005,BarlesJakobsen2007}, semi-Lagrangian schemes based on the discrete dynamic programming principle~\cite{Falcone1987,FalconeFerretti1998,FalconeFerretti2014,DebrabantJakobsen2013}, high-order essentially non-oscillatory schemes~\cite{OsherShu1991}, finite element methods~\cite{jensen2013convergence}, and max-plus methods~\cite{FlemingMcEneaney2000,McEneaney2006,AkianGaubertLakhoua2008}.
A further classical route is the Markov chain approximation method of Kushner and Dupuis~\cite{KushnerDupuis2001}, which replaces the controlled dynamics by a locally consistent controlled Markov chain and thereby inherits the dynamic programming structure by construction.
Quantitative rates for monotone approximations of this type go back to Krylov's method of shaking the coefficients~\cite{Krylov2000} and to the error analysis of Barles and Jakobsen~\cite{BarlesJakobsen2005,BarlesJakobsen2007}.
In all of these works, the general convergence theory requires monotonicity. Monotone schemes, however, are confined to low order, and the constructions that go beyond first order, such as the essentially non-oscillatory schemes, give up the comparison structure together with the guarantees it carries. An exception is the discontinuous Galerkin approximation of uniformly elliptic HJB equations under a Cordes condition \cite{SmearsSuli2014,SmearsSuli2016}, which attains high order with convergence guarantees, a route specific to the second-order setting. 
The efficient numerical solution of HJB equations remains an active field of study. Adaptive and sparse grids mitigate the cost of dynamic programming in higher dimension~\cite{Grune1997,BokanowskiGarckeGriebelKlompmaker2013}, fast-marching and ordered-upwind methods exploit the causality of static first-order equations~\cite{SethianVladimirsky2003,Mirebeau2014,AkianGaubertLiu2024}, anti-dissipative schemes and state-constrained formulations sharpen front propagation and reachability computations~\cite{BokanowskiZidani2007,BokanowskiForcadelZidani2010,MitchellBayenTomlin2005}, and representation formulas of Lax-Oleinik/Hopf and max-plus type lead to methods that attenuate the curse of dimensionality~\cite{DarbonOsher2016,McEneaneyDower2015}. Intrinsic semi-Lagrangian schemes on embedded submanifolds were analyzed recently in \cite{HangelbroekRiegerWright2026}, with uniform convergence rates for linear transport equations with smooth solutions.

Most of these methods are formulated on grids or meshes of $\R^{n}$, and their transplantation to a curved state space proceeds through local charts. This is feasible, and two of the baselines in our experiments are of exactly this kind, but the chart obscures the global structure of the group and leaves a trace in the computed value, which acquires a dependence on coordinate choices that the problem itself does not have. The experiments of \cref{sec:numerics} quantify this effect. 
By contrast, harmonic analysis provides, on every compact group, a canonical global substitute for the trigonometric basis. By the Peter-Weyl theorem, the matrix coefficients of the irreducible unitary representations of $G$ form a complete orthonormal system of $L^{2}(G)$. On $\SO{3}$ this system is the classical Wigner $D$-basis, fast transforms are available~\cite{ShenXuZhang2018}, and Fourier methods built on it are standard tools in the engineering applications of noncommutative harmonic analysis~\cite{ChirikjianKyatkin2016}. It is therefore natural to perform the dynamic programming iteration in the Peter-Weyl coefficient domain, in the spirit of classical spectral methods~\cite{KarniadakisSherwin2005,HesthavenGottliebGottlieb2007,shen2011spectral}.

Galerkin approximations of HJB-type equations have been considered before. The successive Galerkin approximation~\cite{BeardSaridisWen1997} was applied to optimal attitude control by Lawton and Beard~\cite{LawtonBeard1999}, and a successive Wigner-Galerkin method, which represents the iterates of a policy iteration in the Wigner $D$-basis, was proposed for stochastic kinematic control on $\SO{3}$~\cite{WangWangSolo2024}. These works approximate the equation in the $L^{2}$ sense of Galerkin methods and do not address the stability of the dynamic programming iteration in the supremum norm. In the commutative case, spectral viscosity approximations of Hamilton-Jacobi equations on the torus were shown to converge to the viscosity solution by Lepsky~\cite{Lepsky2000}, in a method-of-lines setting. Max-plus basis methods~\cite{FlemingMcEneaney2000,McEneaney2006,AkianGaubertLakhoua2008} take the opposite route. They expand the value function in bases adapted to the min-plus linearity of the dynamic programming semigroup, so they preserve the order structure by construction. The method proposed here keeps both ingredients, the group basis for the transport and a Markov structure for the nonlinear iteration.

Our first result is negative. 
The Bellman operator inherits from the dynamic programming principle the two properties on which the convergence theory ultimately rests, namely monotonicity and non-expansiveness in $L^{\infty}$.
The orthogonal Peter-Weyl projection, by contrast, is non-expansive in $L^{2}$ but possesses neither property, and its $L^{\infty}$ operator norm, the Lebesgue constant of the truncation, is unbounded. 
The natural sup-norm error recursion for the iterated projection and Bellman steps therefore carries, at each step, a factor that grows like $\log N$ on $S^{1}$ and like $N^{(d-1)/2}$ on compact Lie groups of rank one and dimension $d\ge2$, a class that contains $\SO{3}$. 
We do not claim that no norm exists in which both operators are simultaneously well behaved, but the natural $L^{2}/L^{\infty}$ framework, in which each operator is non-expansive in its own norm, cannot be closed without this amplification.
The obstruction is not an artifact of the analysis. In the experiments of \cref{sec:numerics}, the Galerkin iteration violates elementary bounds satisfied by the exact discrete value function within the first few steps, at every resolution tested.

We remove this obstruction by restoring the dynamic programming structure at the discrete level. 
We replace the orthogonal projection by a spectral filter that is bounded on $L^{\infty}$ uniformly in the truncation, a classical stabilization device in spectral methods \cite{GottliebShu1997,HesthavenGottliebGottlieb2007,TadmorTanner2002,Tanner2006}, and that, in the form used for our main results, is a Markov operator, namely the convolution with a probability density on the group. 
Every step of the filtered iteration is then again monotone and sup-norm non-expansive. 
A vanishing-viscosity regularization of the equation is incorporated in the analysis and provides a further smoothing mechanism. 
The convergence proof for the Markov scheme does not proceed through consistency in the viscosity sense. 
A filtered Bellman step is the exact dynamic programming operator of a controlled Markov chain in which the deterministic transport is followed by a small centred random jump, so the filtered scheme belongs to the family of Markov chain approximations of Kushner and Dupuis~\cite{KushnerDupuis2001}. Their convergence theory is qualitative, whereas a coupling of the perturbed chain with the noise-free one yields explicit rates under Lipschitz data alone.

\subsection{Contribution}

We begin by formulating the semi-Lagrangian dynamic programming iteration in the Peter-Weyl coefficient domain of a general connected compact Lie group. Its Galerkin implementation is unstable, the error recursion being governed by the Lebesgue constant of the projection (\cref{lem:norm_mismatch} and \cref{thm:sup_norm}). A sup-norm rate is available only under a uniform Sobolev bound that fails for generic nonsmooth value functions in dimension $d\ge 3$ (\cref{ssec:Hs_discussion}).

We then analyze a filtered viscous variant of the scheme, in which the equation is regularized by a vanishing-viscosity term and the orthogonal projection is replaced by the heat-kernel filter, a Markov operator. For this scheme we prove convergence in the supremum norm to the viscosity solution for Lipschitz data, in dimension $d\le 3$ when the viscous step is an implicit resolvent and in every dimension when it is the heat semigroup (\cref{thm:lowreg}), together with a finite-rank realization based on a side truncation of the filter and a logarithmic oversampling rule (\cref{prop:finite_rank_heat,cor:lowreg_finite_rank}). In the second case the smoothing of the viscous step itself supplies the regularity that the data lack. Along the way we establish the Crandall-Lions vanishing-viscosity estimate on a compact Lie group (\cref{thm:VV}), for which we are not aware of a reference.

For the main result of the paper we construct a spectral filter that is finite-rank, positivity preserving, and non-expansive in the supremum norm (\cref{thm:finite_rank_markov_filter}). It is a Fej\'er-type kernel built from a smooth spectral window, and its construction rests on a product property of the Casimir spectral subspaces of $G$ (\cref{lem:spectral_product}). For the resulting Fej\'er-Markov scheme we prove a convergence theorem with an explicit rate, under Assumption~\ref{ass:standing} alone (\cref{thm:finite_rank_lipschitz_rate}). The argument views the filtering step as a small centred random perturbation of the controlled dynamics and compares the perturbed and unperturbed chains by a coupling. The artificial viscosity may be set to zero (\cref{cor:canonical_scaling}). The same coupling covers a fully discrete realization based on positive cubature, including the control discretization and the Lie-group flow integration (\cref{thm:fully_discrete_convergence}), so the theorem applies to the implemented iteration. In informal terms, the main result reads as follows.

\begin{theorem}[informal version of \cref{thm:finite_rank_lipschitz_rate}, \cref{cor:canonical_scaling}, and \cref{thm:fully_discrete_convergence}]
\label{thm:informal}
Assume that the dynamics and the costs are bounded and Lipschitz on $G$ (\cref{ass:standing}). Let $\delta=T/K$ be the time step, let $N$ be the spectral resolution, let $\eps\ge0$ be the artificial viscosity, and suppose that $\nu=\eps+(\delta N^{2})^{-1}\le1$. Then the iterates $U_{k}$ of the Fej\'er-Markov scheme satisfy
\[
\max_{0\le k\le K}\,\bigl\lVert U_{k}-\val(\cdot,t_{k})\bigr\rVert_{L^{\infty}(G)}
\le C_{T}\bigl(\sqrt{\delta}+\sqrt{\nu}\,\bigr),
\]
where $t_{k}=k\delta$ and $C_{T}$ depends only on the data, the horizon $T$, and the geometry of $G$. The viscosity may be set to zero, and the canonical coupling $\delta=N^{-1}$ gives the rate $C_{T}N^{-1/2}$. The estimate persists for the fully discrete realization, up to two additional terms accounting for the discretization of the control set and for the accuracy of the Lie-group integrator.
\end{theorem}

The paper closes with the realization on $\SO{3}$, where the transport acts exactly on the Wigner coefficients and the pointwise minimization is performed at the nodes of a positive cubature, so the implemented iteration is the one covered by \cref{thm:informal}. The experiments compare the schemes with a closed-form solution and with an independent upper-bound reference. The measured orders at the canonical coupling are $0.55$ and $0.61$ against the predicted $1/2$.



\subsection{Organization}

The paper is organized as follows. \Cref{sec:preliminaries} recalls the background on compact Lie groups, the optimal control problem, and the dynamic programming principle. 
\Cref{sec:specapprox} introduces the spectral Bellman scheme, establishes the norm-mismatch obstruction, and proves a conditional convergence result under a uniform Sobolev hypothesis. \Cref{sec:stab} develops the deterministic route, that is the vanishing-viscosity regularization and the intrinsic Crandall-Lions estimate (\cref{ssec:viscous_HJB}), the IMEX operators (\cref{ssec:Veps_delta}), and the filtered schemes with their convergence under minimal regularity (\cref{ssec:filter,ssec:lowreg}). 
\Cref{sec:markov} contains the main results of the paper, that is the finite-rank Markov filter (\cref{ssec:markov_filter}), the convergence theorem (\cref{ssec:markov_convergence}), and the fully discrete monotone realization (\cref{ssec:fully_discrete_markov}). 
\Cref{sec:numerics} reports the numerical experiments on $\SO{3}$. 

\section{Preliminaries}
\label{sec:preliminaries}

\subsection{Compact Lie groups and the Casimir Laplacian}
\label{ssec:Lie}

Throughout the paper, $G$ denotes a connected compact Lie group of dimension $d$, with identity $e$ and Lie algebra $\Lie:=T_{e}G$.
We write $\Ad:G\to\mathrm{Aut}(\Lie)$ for the adjoint representation, $L_{g}:h\mapsto gh$ for left translation, and $TL_{g}:T_{h}G\to T_{gh}G$ for its differential.
Recall that the exponential map $\exp:\Lie\to G$ is surjective by compactness of $G$
\cite{Knapp2002}.

We fix an $\Ad$-invariant inner product $\inner{\cdot}{\cdot}_{\Lie}$ on $\Lie$.
Such an inner product exists by compactness of $G$ \cite{Knapp2002}.
Moreover, it is unique up to a positive scalar on each simple ideal of $\Lie$, while on the abelian centre every inner product is $\Ad$-invariant.
Its left translation produces a bi-invariant Riemannian metric on $G$ \cite{Helgason2001}, whose associated geodesic distance, gradient, divergence and Haar probability measure on $G$ will be denoted by $\dist_{G}:G\times G\to[0,\infty)$, $\nabla_{G}$, $\mathrm{div}_{G}$, and $\mu$ (so that $\mu(G)=1$), respectively.
Under the left trivialization $TG\cong G\times\Lie$, $(g,X_{g})\mapsto(g,(TL_{g})^{-1}X_{g})$, a vector field $g\mapsto X_{g}$ corresponds to a map $\xi:G\to\Lie$ via $X_{g}=(TL_{g})\xi(g)$.
Combined with the identification $\Lie\cong\Lie^{*}\cong T_{g}^{*}G$ given by the $\Ad$-invariant inner product, the left trivialization defines the \emph{body-frame gradient}
$\nabla_{G}^{\mathrm{body}}:C^{1}(G)\to C(G;\Lie)$,
\begin{equation}
\inner{\nabla_{G}^{\mathrm{body}}f(g)}{Y}_{\Lie}=df_{g}\bigl((TL_{g})Y\bigr), \quad\forall\,g\in G,\ Y\in\Lie.
\label{eq:body_gradient}
\end{equation}
Equivalently, if $\{X_{j}\}_{j=1}^{d}$ is an $\inner{\cdot}{\cdot}_{\Lie}$-orthonormal basis of $\Lie$ with associated left-invariant vector fields $\widetilde X_{j}$ on $G$, then $\nabla_{G}^{\mathrm{body}}f(g)=\sum_{j=1}^{d}(\widetilde X_{j}f)(g)\,X_{j}\in\Lie$. The Riemannian gradient $\nabla_{G}$ and the body-frame gradient $\nabla_{G}^{\mathrm{body}}$ are related by $\nabla_{G}f(g)=(TL_{g})\nabla_{G}^{\mathrm{body}}f(g)\in T_{g}G$.

The Laplace-Beltrami operator on $G$, $\Lap_{G}=\mathrm{div}_{G}\nabla_{G}$, which is an operator on
$C^{\infty}(G)$, coincides with the (negative) Casimir element of the universal enveloping algebra $\mathcal{U}(\Lie)$, acting on $L^{2}(G)$ through the left regular representation.
Indeed, if $\{X_{1},\dots,X_{d}\}$ is an $\inner{\cdot}{\cdot}_{\Lie}$-orthonormal basis of $\Lie$ and $\widetilde X_{j}$ denotes the associated left-invariant vector field on $G$, then $\Lap_{G}=\sum_{j=1}^{d}\widetilde X_{j}^{2}$. It is known that $-\Lap_{G}$ is non-negative self-adjoint on $L^{2}(G,\mu)$ with discrete spectrum tending to infinity.

For $1\le p\le\infty$ and $k\ge 0$, the Lebesgue, continuous, and Sobolev spaces $L^{p}(G)$, $C^{k}(G)$ and $W^{k,p}(G)$ are defined in the usual way. The space of Lipschitz functions on $(G,\dist_{G})$ is denoted by $\Lip(G)$, with Lipschitz constant $\Lip(f)=\sup_{g\ne h}|f(g)-f(h)|/\dist_{G}(g,h)$. For $s\ge 0$, we define the Sobolev space $H^{s}(G)$ by means of the spectral calculus of $-\Lap_{G}$:
\begin{equation}
\begin{aligned}
H^{s}(G)&:=\bigl\{f\in L^{2}(G):(I-\Lap_{G})^{s/2}f\in L^{2}(G)\bigr\} \ ,\\
\norm{f}_{H^{s}(G)}&:=\bigl\lVert(I-\Lap_{G})^{s/2}f\bigr\rVert_{L^{2}(G)}.
\end{aligned}
\label{eq:Hs_def}
\end{equation}
That \eqref{eq:Hs_def} agrees with the integer-order definition through left-invariant derivatives (for $s\in\N$) and with the Slobodeckij seminorms (for non-integer $s$) is standard (see, e.g., \cite{Triebel2010,brezis2011functional}). We shall make repeated use of the Sobolev embedding $H^{s}(G)\hookrightarrow C^{k}(G)$, valid for $s>d/2+k$.

\subsection{Optimal control problem and Hamilton-Jacobi-Bellman equation}
\label{ssec:OCP}

We consider a deterministic optimal control problem with finite horizon $T$ and a nonempty compact control set $U\subset\R^{m}$. The data of the optimal control problem consists of a controlled vector field (dynamics) $F:G\times U\to TG$, with $F(g,u)\in T_{g}G$, written in body-frame coordinates as $F(g,u)=(TL_{g})\xi(g,u)$ for some map $\xi:G\times U\to\Lie$, a running cost $\ell:G\times U\to\R$ and a terminal cost $\phi:G\to\R$.
Throughout this paper, we make the following regularity assumptions on the data.

\begin{assumption}
\label{ass:standing}
\begin{itemize}
\item[i.] $U$ is non-empty and compact;
\item[ii.] $\xi\in C(G\times U;\Lie)$,
$\ell\in C(G\times U;\R)$, $\phi\in C(G;\R)$ are bounded;
\item[iii.] there
exist constants $L_{\xi},L_{\ell},L_{\phi}$ such that
\begin{equation}
	\left\{
	\begin{aligned}
&\;\norm{\xi(g_{1},u)-\xi(g_{2},u)}_{\Lie}\le L_{\xi}\dist_{G}(g_{1},g_{2}),\\
&\;|\ell(g_{1},u)-\ell(g_{2},u)|\le L_{\ell}\dist_{G}(g_{1},g_{2}), \\
&\;|\phi(g_{1})-\phi(g_{2})|\le L_{\phi}\dist_{G}(g_{1},g_{2}),
\end{aligned}
\right.
\label{eq:Lip_data}
\end{equation}
 for
all $g_{1},g_{2}\in G$, $u\in U$.
\end{itemize}
\end{assumption}

For any measurable control process $u:[t,T]\to U$ and any initial datum $(g,t)\in G\times[0,T)$, under \cref{ass:standing} the Cauchy problem
\[\dot\zeta(s)=F(\zeta(s),u(s)), \ s\in[t,T] \ ,\]
admits a unique absolutely continuous solution, with initial condition $\zeta(t)=g$,  $\zeta=\zeta(\cdot;g,t,u(\cdot)):[t,T]\to G$ (see, e.g., \cite{BardiCapuzzoDolcetta1997}). We denote by $\mathcal{U}_{t}$ the set of all admissible controls $u(\cdot)$ on $[t,T]$, and define the \emph{value function} by
\begin{equation}
\val(g,t):=\inf_{u(\cdot)\in\mathcal{U}_{t}}\biggl\{\int_{t}^{T}\ell\bigl(\zeta(s),u(s)\bigr)\,\mathrm{d}s + \phi\bigl(\zeta(T)\bigr)\biggr\} \ ,
\label{eq:value_function} 
\end{equation}
with the terminal condition $\val(g,T):=\phi(g)$.
Under Assumption~\ref{ass:standing}, $\val(\cdot,t)\in\Lip(G)$ uniformly in $t\in[0,T]$, and $t\mapsto\val(g,t)$ is Lipschitz on $[0,T]$ uniformly in $g$. Here, the boundedness of $\xi$ and $\ell$ is used for the time regularity (see, e.g., \cite{BardiCapuzzoDolcetta1997}).

Let us define the associated Hamiltonian
\begin{equation}
\Ham(g,p):=\sup_{u\in U}\Bigl\{-\inner{p}{\xi(g,u)}_{\Lie}-\ell(g,u)\Bigr\},\quad g\in G,\ p\in\Lie\cong T_{g}^{*}G,
\label{eq:Hamiltonian}
\end{equation}
where the identification $\Lie\cong T_{g}^{*}G$ is made via the bi-invariant inner product. It follows from Assumption~\ref{ass:standing} that $\Ham$ is continuous on $G\times\Lie$ and Lipschitz in the
covariable $p$, uniformly in $g$.

\begin{theorem}[see, e.g., \cite{CrandallIshiiLions1992,BardiCapuzzoDolcetta1997,AzagraFerreraLopezMesas2005}]
\label{thm:HJB}
Under Assumption~\ref{ass:standing}, the value function $\val$ is the unique
viscosity solution of the terminal-value HJB equation
\begin{equation}
	\left\{
	\begin{aligned}
&\;-\partial_{t}\val(g,t)+\Ham\bigl(g,\nabla_{G}^{\mathrm{body}}\val(g,t)\bigr)=0,
\quad (g,t)\in G\times(0,T),\\
 &\;\val(g,T)=\phi(g), \quad g \in G \ .
 \end{aligned}
 \right.
\label{eq:HJB}
\end{equation}
\end{theorem}

The notion of viscosity solution employed here is the standard one of~\cite{CrandallLions1983}. It carries over from $\R^{n}$ to a compact Riemannian manifold without modification. The continuous dynamic programming principle
\[\val(g,t)=\inf_{u(\cdot)}\bigl\{\int_{t}^{s}\ell (\zeta(\tau),u(\tau))\,\mathrm{d}\tau+\val(\zeta(s),s)\bigr\}\]
holds for $0\le t\le s\le T$ in the same generality.

Since the data $(\xi,\ell)$ are time-independent, the dynamic programming principle can be rewritten in evolutionary form. For $\tau\ge0$ and $\psi\in C(G)$, let $S_{\tau}\psi$ denote the value function of the problem \eqref{eq:value_function} posed on the horizon $\tau$ with terminal cost $\psi$, that is,
\begin{equation}
(S_{\tau}\psi)(g):=\inf_{u(\cdot)}\Bigl\{\int_{0}^{\tau}\ell\bigl(\zeta(s),u(s)\bigr)\,\mathrm{d}s+\psi\bigl(\zeta(\tau)\bigr)\Bigr\},
\quad \zeta(0)=g,
\label{eq:DP_semigroup}
\end{equation}
the infimum being taken over all measurable controls $u:[0,\tau]\to U$, so that $\val(\cdot,t)=S_{T-t}\phi$ for all $t\in[0,T]$. 
The dynamic programming principle implies that the family $(S_{\tau})_{\tau\ge0}$ is a one-parameter semigroup of nonlinear operators on $C(G)$, namely $S_{\tau+\sigma}=S_{\tau}\circ S_{\sigma}$ and $S_{0}=I$. We call $(S_{\tau})_{\tau\ge0}$ the \emph{dynamic programming semigroup} of the control problem. It is also known as the \emph{Lax-Oleinik semigroup}, a name most common in weak KAM theory \cite{Fathi2008}, where the Hamiltonian is additionally strictly convex and superlinear in $p$.  
 Each operator $S_{\tau}$ is monotone, additively homogeneous, and non-expansive in the supremum norm.  Equivalently, the semigroup is linear over the min-plus semiring, which is the point of view of max-plus based numerical methods \cite{McEneaney2006,AkianGaubertLakhoua2008}. The object approximated in this paper is this evolution semigroup. 

\subsection{Lie-group integrator and discrete dynamic programming}
\label{ssec:DPP_discrete}

For each constant control $u\in U$ and each $\delta\ge 0$, we denote by $\Phi^{u}_{\delta}:G\to G$ the time-$\delta$ flow of the autonomous vector field $g\mapsto F(g,u)$.
Let $K\in\N$, $\delta=T/K$ and $t_{k}=k\delta$, $k=0,\dots,K$.
The discrete Bellman operator $\mathcal{T}_{\delta}:C(G)\to C(G)$ is defined by
\begin{equation}
(\mathcal{T}_{\delta}\val)(g):=\inf_{u\in U}\Bigl\{\delta\,\ell(g,u)+\val\bigl(\Phi^{u}_{\delta}(g)\bigr)\Bigr\}.
\label{eq:Bellman}
\end{equation}
The operator $\mathcal{T}_{\delta}$ is the one-step semi-Lagrangian approximation of the dynamic programming semigroup \eqref{eq:DP_semigroup}. It is obtained from $S_{\delta}$ by restricting the controls to constants on the step and freezing the running cost at its initial value, and it inherits the structural properties of $S_{\delta}$.

\begin{proposition}
\label{prop:Tdelta_properties}
Under Assumption~\ref{ass:standing}, the operator $\mathcal{T}_{\delta}$ of
\eqref{eq:Bellman} is monotone, additively homogeneous, that is  $\mathcal{T}_{\delta}(\lambda+f)=\lambda+\mathcal{T}_{\delta}f$ for every $\lambda\in\R$, sup-norm non-expansive,
that is $\norm{\mathcal{T}_{\delta}f-\mathcal{T}_{\delta}g}_{\infty}\le\norm{f-g}_{\infty}$,
and Lipschitz-preserving with
$\Lip(\mathcal{T}_{\delta}f)\le e^{L_{\xi}\delta}(L_{\ell}\delta+\Lip(f))$
for every $f\in\Lip(G)$.
\end{proposition}

Let $\val_{\delta}:G\times\{t_{k}\}_{k=0}^{K}\to\R$ be defined by
$\val_{\delta}(\cdot,t_{K}):=\phi$ and
\begin{equation}
\val_{\delta}(\cdot,t_{k}):=\mathcal{T}_{\delta}\bigl(\val_{\delta}(\cdot,t_{k+1})\bigr),\quad k=K-1,\dots,0.
\label{eq:DPP_discrete}
\end{equation}

\begin{theorem}
\label{thm:DPP_convergence}
Under Assumption~\ref{ass:standing}, there exists $C>0$ depending only on the constants of Assumption~\ref{ass:standing}, the horizon $T$, and the geometry of $G$ such that
\begin{equation}
\sup_{0\le k\le K}\norm{\val_{\delta}(\cdot,t_{k})-\val(\cdot,t_{k})}_{L^{\infty}(G)}\le C\sqrt{\delta}.
\label{eq:DPP_rate}
\end{equation}
If $\val\in C^{2}(G\times[0,T])$, the rate improves to
$\norm{\val_{\delta}-\val}_{\infty}\le C\delta$.
\end{theorem}

\begin{remark}
\cref{thm:DPP_convergence} is the standard semi-Lagrangian consistency and stability argument of \cite{BardiCapuzzoDolcetta1997} or \cite{Falcone1987}, transposed to the compact Riemannian setting. The Euclidean argument is intrinsic in nature. It relies only on the Lipschitz continuity of the data, the non-expansiveness of the Bellman operator in $L^{\infty}$, and the doubling-of-variables comparison technique, in which the test functions are built from the squared geodesic distance $\dist_{G}^{2}$. On a compact group, $\dist_{G}^{2}$ is smooth on a uniform neighbourhood of the diagonal (the injectivity radius is bounded below by compactness), and the constants entering the comparison argument depend only on $T$, the data constants of \cref{ass:standing}, and curvature and injectivity-radius bounds of $(G,\dist_{G})$ (see \cite{CrandallIshiiLions1992} for the doubling machinery and \cite{AzagraFerreraLopezMesas2005}-type tools for nonsmooth analysis on Riemannian manifolds). The exponential map is used only locally, where it is a diffeomorphism (on a compact Lie group it is surjective but in general not injective). A complete proof, intrinsic to $(G,\dist_{G})$, is given in \cref{app:proofs}.
\end{remark}

\section{Spectral approximation via the Peter-Weyl basis}
\label{sec:specapprox}

In this section, we approximate the discrete value function $\val_{\delta}$ of \eqref{eq:DPP_discrete} by orthogonal projection onto a finite-dimensional subspace of $L^{2}(G)$ spanned by the matrix
coefficients of the irreducible unitary representations of $G$ (Peter-Weyl theorem~\cite{Knapp2002}).
We first present the scheme and then prove a sup-norm error bound under an explicit uniform Sobolev-regularity hypothesis on $\val_{\delta}$.
This analysis isolates the central analytical difficulty of the spectral approach, namely the $L^{2}/L^{\infty}$ norm mismatch between the orthogonal projector, which is non-expansive in $L^{2}$, and the Bellman operator, which is non-expansive in $L^{\infty}$, the norm intrinsic to dynamic programming and to the comparison of value functions. The regularity hypothesis turns out to be too strong for generic non-smooth value functions, and we will then address this difficulty in \cref{sec:stab,sec:markov}.

\subsection{Peter-Weyl decomposition, and spectral Bellman scheme}
\label{ssec:PW}

We denote by $\widehat G$ the unitary dual of $G$, namely, the set of equivalence classes of irreducible unitary representations $\pi:G\to U(d_{\pi})$ with $d_{\pi}=\dim\pi<\infty$. For each $\pi\in\widehat G$, we fix an orthonormal basis $\{e_{i}^{\pi}\}_{i=1}^{d_{\pi}}$ of $\C^{d_{\pi}}$ and we denote by
$\pi_{ij}$ the associated \emph{matrix coefficients}, defined by $\pi_{ij}(g):=\bigl\langle e_{i}^{\pi},\,\pi(g)e_{j}^{\pi}\bigr\rangle_{\C^{d_{\pi}}}$, $g\in G$.

\begin{theorem}[see, e.g., \cite{Folland2016}]
\label{thm:PW}
The family
$\bigl\{\sqrt{d_{\pi}}\,\pi_{ij}:\pi\in\widehat G,\,1\le i,j\le d_{\pi}\bigr\}$
is a complete orthonormal basis of $L^{2}(G,\mu)$.
Every $f\in L^{2}(G)$ admits the expansion
\begin{equation}
f(g)=\sum_{\pi\in\widehat G}\sum_{i,j=1}^{d_{\pi}}\widehat f^{\pi}_{ij}\,\sqrt{d_{\pi}}\,\pi_{ij}(g),
\label{eq:PW_expansion}
\end{equation}
with coefficients
\begin{equation}
\widehat f^{\pi}_{ij}:=\sqrt{d_{\pi}}\int_{G}f(g)\,\overline{\pi_{ij}(g)}\,\mathrm{d}\mu(g)
=\sqrt{d_{\pi}}\int_{G}f(g)\,\pi_{ji}(g^{-1})\,\mathrm{d}\mu(g),
\label{eq:PW_coeff}
\end{equation}
where the two equivalent integrals are related by the unitarity identity $\overline{\pi_{ij}(g)}=\pi_{ji}(g^{-1})$, and the series \eqref{eq:PW_expansion} converges in $L^{2}(G)$.
Furthermore, the Casimir Laplacian $\Lap_{G}$ acts on each isotypic component $V_{\pi}\otimes V_{\pi}^{*}\subset L^{2}(G)$ by multiplication by a constant $-c_{\pi}\le 0$, the \emph{Casimir eigenvalue} of $\pi$.
\end{theorem}

Throughout the paper, we use the conjugate form $\overline{\pi_{ij}(g)}$ in the forward Peter-Weyl transform.

By \Cref{thm:PW} and \eqref{eq:Hs_def}, the Sobolev norm admits the
Peter-Weyl form
\begin{equation}
\norm{f}_{H^{s}(G)}^{2}=\sum_{\pi\in\widehat G}\sum_{i,j=1}^{d_{\pi}}(1+c_{\pi})^{s}\,\bigl|\widehat f^{\pi}_{ij}\bigr|^{2},\quad s\ge 0.
\label{eq:Hs_PW}
\end{equation}

For $N\ge 1$, set
\begin{equation}
\begin{aligned}
\widehat G_{N}&:=\bigl\{\pi\in\widehat G:\sqrt{c_{\pi}}\le N\bigr\},\\
X_{N}&:=\overline{\mathrm{span}}\bigl\{\pi_{ij}:\pi\in\widehat G_{N},\,1\le i,j\le d_{\pi}\bigr\}\subset L^{2}(G).
\end{aligned}
\end{equation}
The orthogonal projector $\Pi_{N}:L^{2}(G)\to X_{N}$ is given by
truncating the expansion \eqref{eq:PW_expansion} to
$\pi\in\widehat G_{N}$.

The natural Galerkin discretization of the dynamic programming equation \eqref{eq:DPP_discrete} applies the orthogonal $L^{2}$ projector $\Pi_{N}$ after each Bellman step. In particular,
for $K\in\N$, $\delta=T/K$, $N\ge 1$, the \emph{spectral Bellman
iterate} $\val^{N}_{k}:G\to\R$, $k\in\{0,1,\dots,K\}$, is defined by
\begin{equation}
	\left\{
	\begin{aligned}
&\;\val^{N}_{K}:=\Pi_{N}\phi,\\
&\;\val^{N}_{k}:=\Pi_{N}\,\mathcal{T}_{\delta}\bigl(\val^{N}_{k+1}\bigr),\quad k=K-1,K-2,\dots,0.
\end{aligned}
\right.
\label{eq:VN_scheme}
\end{equation}

\subsection{Single-step error and the norm mismatch obstruction}
\label{ssec:err_obstruction}

We first define the spectral approximation error at time step $t_{k}$ as $E_{k}:=\val^{N}_{k}-\val_{\delta}(\cdot,t_{k})$. By an elementary computation, we have that for every $k\in\{0,1,\dots,K-1\}$,
\begin{equation}
	E_{k}=\Pi_{N}\bigl(\mathcal{T}_{\delta}\val^{N}_{k+1}-\mathcal{T}_{\delta}\val_{\delta}(\cdot,t_{k+1})\bigr)-\Pi_{N}^{\perp}\val_{\delta}(\cdot,t_{k}),
	\label{eq:err_identity}
\end{equation}
where $\Pi_{N}^{\perp}:=I-\Pi_{N}$.
In order to turn the identity \eqref{eq:err_identity} into a quantitative bound that iterates from $k=K$ down to $k=0$, we need a norm in which both the projector $\Pi_{N}$ and the Bellman operator $\mathcal{T}_{\delta}$ are non-expansive. However, the central analytical difficulty of the spectral approach is that both natural candidates fail. In particular, the projector is not non-expansive in $L^{\infty}$, and the Bellman operator is not non-expansive in $L^{2}$. The following lemma quantifies this obstruction.

\begin{lemma}
\label{lem:norm_mismatch}
The orthogonal projector $\Pi_{N}$ and the Bellman operator
$\mathcal{T}_{\delta}$ satisfy the following bounds, each of which
is sharp in general.
\begin{enumerate}[label=\textnormal{(\roman*)},leftmargin=2em]
\item $\norm{\Pi_{N}f}_{L^{2}(G)}\le\norm{f}_{L^{2}(G)}$ for all
$f\in L^{2}(G)$. Equivalently, $\norm{\Pi_{N}}_{L^{2}\to L^{2}}=1$.
\item $\norm{\mathcal{T}_{\delta}f-\mathcal{T}_{\delta}g}_{L^{\infty}(G)}\le\norm{f-g}_{L^{\infty}(G)}$
for all $f,g\in C(G)$.
\item $\norm{\mathcal{T}_{\delta}f-\mathcal{T}_{\delta}g}_{L^{2}(G)}\le\norm{f-g}_{L^{\infty}(G)}$
for all $f,g\in C(G)$, using $\mu(G)=1$.
\item The \emph{Lebesgue constant}
$\Lambda_{N}:=\norm{\Pi_{N}}_{L^{\infty}(G)\to L^{\infty}(G)}$
is unbounded as $N\to\infty$, with explicit growth rates, for the circle and for compact Lie groups of Lie rank one, i.e. compact Lie groups whose maximal torus is one-dimensional:
\begin{equation}
\Lambda_{N}\asymp\log N\ (\Sone),\quad
\Lambda_{N}\asymp N^{(d-1)/2}\quad(\text{rank-one},\ d\ge 2,\ \text{e.g.\ }\SO{3}),
\label{eq:Lambda_growth}
\end{equation}
(see \cite{Zygmund2002} for $d=1$ and \cite{BonamiClerc1973} for the rank-one case $d\ge 2$). 
For higher-rank compact Lie groups the explicit growth rate depends on the group, the truncation geometry, and the choice of dominant weight ordering.
\end{enumerate}
\end{lemma}

The norm mismatch displayed in \cref{lem:norm_mismatch} is the obstruction addressed in the remainder of the paper. We do not claim that no norm exists in which both $\Pi_{N}$ and $\mathcal{T}_{\delta}$ are non-expansive simultaneously, but the standard $L^{2}/L^{\infty}$ stability framework, in which each operator is natural, cannot be closed without the Lebesgue-constant amplification of \cref{lem:norm_mismatch}\,(iv). From the optimization and optimal control viewpoint, the properties at stake are not technical conveniences. Indeed, monotonicity and sup-norm non-expansiveness are the discrete counterparts of the comparison principle, and they are precisely the hypotheses on which the convergence framework of Barles and Souganidis \cite{BarlesSouganidis1991} for approximation schemes of viscosity solutions is built. 
A discretization that destroys them severs the link between the scheme and the dynamic programming principle.

\subsection{A conditional sup-norm convergence under uniform $H^{s}$ regularity hypothesis}
\label{ssec:sup_norm}

The norm mismatch of \cref{ssec:err_obstruction} does not rule out convergence.
Indeed, the error recursion can still be closed in the supremum norm, at the cost of a factor $\Lambda_{N}$ per time step. In this subsection we derive the resulting bound, in which this amplification is balanced against a per-step consistency error of algebraic order in $N$, valid under a uniform $H^{s}$ bound on $\val_{\delta}$.

Taking sup-norms in \eqref{eq:err_identity} and using \cref{lem:norm_mismatch}, we obtain the recursion
\begin{equation}
\norm{E_{k}}_{L^{\infty}}\le\Lambda_{N}\norm{E_{k+1}}_{L^{\infty}}+\norm{\Pi_{N}^{\perp}\val_{\delta}(\cdot,t_{k})}_{L^{\infty}}.
\label{eq:err_rec_sup}
\end{equation}
The per-step consistency error $\norm{\Pi_{N}^{\perp}\val_{\delta}(\cdot,t_{k})}_{L^{\infty}}$
requires spectral approximation in the sup-norm. By Sobolev embedding, this typically requires a regularity hypothesis on $\val_{\delta}$ which goes beyond Lipschitz continuity.

\begin{assumption}
\label{ass:Hs}
There exist $s>d/2$ and $M_{s}<\infty$ such that
\begin{equation}
\sup_{0\le k\le K}\norm{\val_{\delta}(\cdot,t_{k})}_{H^{s}(G)}\le M_{s} \ .
\tag{$\mathrm{H}_{s}$}
\label{eq:Hs_hypothesis}
\end{equation}
\end{assumption}

The per-step consistency error is then controlled by the following standard spectral estimate, whose short proof we include for completeness.

\begin{lemma}
\label{lem:per_step}
Let $s>d/2$ and fix $\eta\in(0,s-d/2)$. There exists $C=C(G,s,\eta)>0$ such that for every $f\in H^{s}(G)$ and every
$N\ge 1$,
\begin{equation}
\norm{\Pi_{N}^{\perp}f}_{L^{\infty}(G)}\le C\,N^{-(s-d/2-\eta)}\,\norm{f}_{H^{s}(G)}.
\label{eq:per_step_general}
\end{equation}
\end{lemma}

\begin{proof}
Set $\sigma:=d/2+\eta\in(d/2,s)$. By \eqref{eq:Hs_PW}, and since $(1+c_{\pi})^{\sigma-s}\le(1+N^{2})^{\sigma-s}$ on the tail $\sqrt{c_{\pi}}>N$,
\begin{align*}
\norm{\Pi_{N}^{\perp}f}_{H^{\sigma}}^{2}
&=\sum_{\sqrt{c_{\pi}}>N}\sum_{i,j}(1+c_{\pi})^{\sigma-s}(1+c_{\pi})^{s}\bigl|\widehat f^{\pi}_{ij}\bigr|^{2}\\
&\le\bigl(1+N^{2}\bigr)^{-(s-\sigma)}\norm{f}_{H^{s}}^{2}
 \le N^{-2(s-\sigma)}\,\norm{f}_{H^{s}}^{2}.
\end{align*}
The Sobolev embedding $H^{\sigma}(G)\hookrightarrow L^{\infty}(G)$, valid since $\sigma>d/2$, then gives
\[
\norm{\Pi_{N}^{\perp}f}_{L^{\infty}}\le C\norm{\Pi_{N}^{\perp}f}_{H^{\sigma}}\le C\,N^{-(s-d/2-\eta)}\norm{f}_{H^{s}}.\qedhere
\]
\end{proof}

\begin{theorem}
\label{thm:sup_norm}
Under Assumption~\ref{ass:standing} and Assumption~\ref{ass:Hs}, fix $\eta\in(0,s-d/2)$. There
exists $C=C(G,s,\eta)>0$ independent of $k,N,K$ such that
\begin{equation}
\norm{\val^{N}_{k}-\val_{\delta}(\cdot,t_{k})}_{L^{\infty}(G)}\le C\,\frac{\Lambda_{N}^{K-k+1}-1}{\Lambda_{N}-1}\,N^{-(s-d/2-\eta)}\,M_{s},\ k=0,1,\dots,K,
\label{eq:sup_norm}
\end{equation}
with the convention $(\Lambda_{N}^{K-k+1}-1)/(\Lambda_{N}-1)=K-k+1$ when $\Lambda_{N}=1$.
\end{theorem}

\begin{proof}
Let us denote
$\alpha:=C N^{-(s-d/2-\eta)}M_{s} $
the per-step consistency constant from \cref{lem:per_step} applied to $\val_{\delta}(\cdot,t_{j})$
and valid uniformly in $j$ under \cref{ass:Hs}. We prove
$
\norm{E_{k}}_{L^{\infty}}\;\le\;\alpha\,\frac{\Lambda_{N}^{K-k+1}-1}{\Lambda_{N}-1}
$
by downward induction on $k$.

Since $\val^{N}_{K}=\Pi_{N}\phi$ and $\val_{\delta}(\cdot,t_{K})=\phi$, we have $E_{K}=-\Pi_{N}^{\perp}\phi$ and hence $\norm{E_{K}}_{L^{\infty}}=\norm{\Pi_{N}^{\perp}\phi}_{L^{\infty}}\le\alpha$ by Lemma~\ref{lem:per_step} applied at $k=K$.

Suppose $\norm{E_{k+1}}_{L^{\infty}}\le\alpha\,(\Lambda_{N}^{K-k}-1)/(\Lambda_{N}-1)$
for some $k\in\{0,\dots,K-1\}$. From \eqref{eq:err_identity},
\cref{lem:norm_mismatch}(ii) and (iv), and \cref{lem:per_step},
\[
\norm{E_{k}}_{L^{\infty}}\le  \Lambda_{N}\norm{E_{k+1}}_{L^{\infty}}+\norm{\Pi_{N}^{\perp}\val_{\delta}(\cdot,t_{k})}_{L^{\infty}}\le \Lambda_{N}\cdot\alpha\,\frac{\Lambda_{N}^{K-k}-1}{\Lambda_{N}-1}+\alpha.
\]
Combining over a common denominator closes the induction. The case $\Lambda_{N}=1$ is handled by
the trivial recursion
$\norm{E_{k}}_{L^{\infty}}\le\norm{E_{k+1}}_{L^{\infty}}+\alpha$.
\end{proof}

\begin{remark}
\label{rem:read_bound}
The bound \eqref{eq:sup_norm} splits into two factors, the per-step consistency error $N^{-(s-d/2-\eta)}M_{s}$ which is algebraic in $N$ with rate governed by the regularity $s$, and the stability amplification $(\Lambda_{N}^{K-k+1}-1)/(\Lambda_{N}-1)$ which is geometric in the time index.
By \eqref{eq:Lambda_growth}, on $\Sone$ the amplification grows like $(\log N)^{K-k}$, which is moderate at fixed $K$ but unbounded in the coupled limit $K\to\infty$. On a rank-one group of dimension $d\ge 2$ it grows polynomially in $N$, and \eqref{eq:sup_norm} then provides no useful estimate unless $K-k$ is very small or $N$ grows much faster than the time discretization requires.
\end{remark}

\subsection{Verifiability of Assumption~\ref{ass:Hs}}
\label{ssec:Hs_discussion}

Theorem~\ref{thm:sup_norm} reduces the convergence question to verifying Assumption~\ref{ass:Hs} from the data of the control problem. Three regimes are to be distinguished.
\begin{enumerate}[label=(\arabic*),leftmargin=2em]
\item \emph{Smooth, short-horizon regime.} Suppose, beyond Assumption~\ref{ass:standing}, that the data are regular enough for the Hamiltonian \eqref{eq:Hamiltonian} to be of class $C^{r}$ on $G\times\Lie$, which requires $C^{r}$ regularity of $\xi$ and $\ell$ together with a nondegenerate maximization in $u$, and that $\phi\in C^{r}(G)$. If $T$ is less than the first conjugate or caustic time $T^{\ast}$ of the Hamiltonian flow, the method of characteristics yields $\val_{\delta}(\cdot,t_{k})\in C^{r}(G)$ uniformly in $k$ (see, e.g., \cite{CannarsaSinestrari2004}), hence Assumption~\ref{ass:Hs} holds with $s=r$.
\item \emph{Generic horizon, semiconcave regime.} Under additional structural hypotheses (namely semiconcave terminal data, $C^{1,1}$ dependence of $\xi$ and $\ell$ on the state, and $\ell$ strictly convex in $u$), Lax-Oleinik and maximum-principle arguments yield uniform semiconcavity of $\val_{\delta}(\cdot,t_{k})$ in $k$ (see \cite{CannarsaSinestrari2004} for precise statements, noting that strict convexity in $u$ and $W^{2,\infty}$ terminal data alone are not sufficient in general). We present this regime as illustrative. A typical semiconcave function with codimension-one jump in $\nabla V$ across a smooth hypersurface lies in $H^{s}(G)$ for $s<3/2$ (and is generally not in $H^{3/2}$). To see this, consider the local model $V(x_{1},\dots,x_{d})=-|x_{1}|$ on the torus $\T^{d}$.
Its Fourier coefficients are $\widehat V_{(k_{1},\dots,k_{d})}\sim k_{1}^{-2}\,\delta_{k_{2}=\cdots=k_{d}=0}$
(with $k_{1}\ne 0$), and a direct computation gives $\norm{V}_{H^{s}(\T^{d})}^{2}=\sum_{k_{1}\ne 0}(1+k_{1}^{2})^{s}\,k_{1}^{-4}$, finite if and only if $s<3/2$.  The same threshold applies on a general compact Lie group $G$. In particular, any smooth shock hypersurface $\Sigma\subset G$ can be covered by a finite family of charts $(U_{\alpha},\varphi_{\alpha})$ in each of which $\Sigma$ is locally the zero set of a single coordinate function and the above $-|x_{1}|$ model applies.  A partition of unity subordinate to this cover then yields the global $H^{s}$ regularity exponent $s<3/2$. Consequently, Assumption~\ref{ass:Hs} is verifiable from semiconcavity on $\Sone$ (that is, $d/2=1/2<3/2$) but \emph{not} on $G$ of dimension $d\ge 3$ (i.e., $d/2\ge 3/2$). For value functions with singularities of codimension other than one, the threshold is correspondingly different and must be analyzed case by case.
\item \emph{Smooth regularized regime.} One may use the so-called vanishing viscosity regularization (see, e.g., \cite{CrandallLions1983,CrandallLions1984}), replacing $V$ by the solution $V^{\eps}$ of a viscous HJB equation. Under Assumption~\ref{ass:standing} alone, $V^{\eps}$ is merely a continuous viscosity solution. If, however, the data and the Hamiltonian are smooth enough on the relevant cotangent region, parabolic regularity gives $V^{\eps}(\cdot,t)\in H^{s}(G)$ uniformly in $t$ for the corresponding range of $s$, and Assumption~\ref{ass:Hs} holds for such $s$, with constant growing as $\eps^{-\alpha(s)}$ near shock singularities.
\end{enumerate}
The third regime holds in every dimension, at the cost of a controllable bias of order $\sqrt{\eps}$. It is the route followed in \cref{sec:stab}, and it is eventually superseded by the probabilistic argument of \cref{sec:markov}, which requires no regularization at all.

\section{Vanishing-viscosity regularization and filtered approximation}
\label{sec:stab}

Recall that the spectral scheme of \cref{sec:specapprox} is obstructed by the $L^{2}/L^{\infty}$ norm mismatch of Lemma~\ref{lem:norm_mismatch}. 
In this section, we resolve this obstruction by combining two techniques, designed so that each step of the resulting iteration retains the structural properties of dynamic programming, namely monotonicity and sup-norm non-expansiveness, from which a convergence analysis in the viscosity framework can again be built.

\noindent\textbf{(A) Vanishing-viscosity regularization.}
We replace the HJB equation \eqref{eq:HJB} (called inviscid HJB in the following) by its viscous
counterpart
\begin{equation}
-\partial_{t}V^{\eps}+\Ham\bigl(g,\nabla_{G}^{\mathrm{body}}V^{\eps}\bigr)
  -\eps\,\Lap_{G}V^{\eps}=0,
\quad V^{\eps}(\cdot,T)=\phi,
\label{eq:viscous_HJB}
\end{equation}
indexed by $\eps>0$, written in the orientation in which the standard viscosity comparison conventions apply, where $\Lap_{G}$ is normalized so that $-\Lap_{G}\ge 0$ and the term $-\eps\,\Lap_{G}V^{\eps}$ is diffusive for the terminal-value problem. For each $\eps>0$ this equation admits a unique viscosity solution $V^{\eps}$, which is smoother than $V$ and stays at distance $O(\sqrt{\eps T})$ from it in the supremum norm (\cref{thm:VV} below).

\noindent\textbf{(B) Filtered Peter-Weyl approximation.}
We replace the orthogonal projector $\Pi_{N}$ by a filtered approximation operator $\mathcal A_{N}$ that is bounded on $L^{\infty}(G)$ uniformly in the truncation and consistent on Sobolev and Lipschitz classes. This removes the $\Lambda_{N}$ amplification at every step.

The schemes of this section iterate one viscous IMEX Bellman step followed by the filter. Their convergence analysis is deterministic. It takes the form of a single recursion against the continuous viscous solution $V^{\eps}$, complemented by a compact Lie group counterpart of the Crandall-Lions vanishing-viscosity bias:
\begin{equation}
\|V^{N,\eps}_{\delta,k}-V(\cdot,t_{k})\|_{L^{\infty}}
\le
\underbrace{\|V^{N,\eps}_{\delta,k}-V^{\eps}(\cdot,t_{k})\|_{L^{\infty}}}_{\substack{\text{spectral + IMEX}}}
+\underbrace{\|V^{\eps}-V\|_{L^{\infty}}}_{\substack{\text{viscous bias}}} \ .
\label{eq:error_decomposition}
\end{equation}
This route requires no further regularity of the value function at all, the viscous step being itself a smoother, but it yields convergence without an explicit rate. With the implicit-Euler diffusion step it is moreover limited to dimension $d\le 3$, a restriction lifted in every dimension by the exact heat (exponential) step.
The present section provides the convergence theory for the heat-filtered scheme used in the experiments, makes precise what artificial viscosity and spectral filtering deliver on their own, and establishes the vanishing-viscosity estimate on a compact Lie group (\cref{thm:VV}).


\subsection{Vanishing-viscosity regularization}
\label{ssec:viscous_HJB}

With the sign convention $\Lap_{G}\pi_{ij}=-c_{\pi}\pi_{ij}$, so that $-\Lap_{G}\ge 0$, equation \eqref{eq:viscous_HJB} is a backward parabolic regularization of the HJB equation \eqref{eq:HJB}.
Under the time reversal $\tau=T-t$, $u^{\eps}(g,\tau):=V^{\eps}(g,T-\tau)$ satisfies \[\partial_{\tau}u^{\eps}=-\Ham(g,\nabla u^{\eps})+\eps\Lap_{G}u^{\eps}\ , \]
which is a forward semilinear parabolic problem.
Freezing the nonlinear transport term, the linear part of the flow is the heat semigroup $e^{\eps\tau\Lap_{G}}$, which acts diagonally in the Peter-Weyl basis and damps the $\pi$-mode by the factor $e^{-\eps c_{\pi}\tau}\in(0,1]$. This is the quantitative form of the smoothing effect of the regularization. An artificial viscosity implemented in this way, as a spectral multiplier, is the mechanism of the spectral vanishing-viscosity method for nonlinear conservation laws~\cite{Tadmor1989,KaramanosKarniadakis2000}, where it stabilizes the spectral discretization while preserving its accuracy. For Hamilton-Jacobi equations on the torus, spectral viscosity approximations were shown to converge to the viscosity solution in \cite{Lepsky2000}. The present setting differs in that the equation is posed on a noncommutative group, the scheme is a semi-Lagrangian iteration rather than a method of lines, and the analysis must control the interaction of the regularization with the dynamic programming structure.

\begin{theorem}[Vanishing-viscosity estimate]
\label{thm:VV}
Under Assumption~\ref{ass:standing}, for every $\eps>0$, the viscous HJB equation
\eqref{eq:viscous_HJB} admits a unique viscosity solution
$V^{\eps}\in C(G\times[0,T])$ with $V^{\eps}(\cdot,T)=\phi$, and $V^{\eps}(\cdot,t)$ is Lipschitz uniformly in $\eps$ and $t$, with
\begin{equation}
  \Lip\bigl(V^{\eps}(\cdot,t)\bigr)\le L:=e^{L_{\xi}T}\bigl(L_{\phi}+L_{\ell}T\bigr).
  \label{eq:VV_global_Lip_constant}
\end{equation}
Moreover, there exist constants $\alpha_{0},A_{T},B_{G}>0$, depending only
on the bounds and Lipschitz constants in Assumption~\ref{ass:standing}, the horizon
$T$, and the geometry of $G$, such that
\begin{equation}
  \|V^{\eps}-V\|_{L^{\infty}(G\times[0,T])}
  \le A_{T}\alpha+B_{G}\frac{\eps T}{\alpha}
  \quad(0<\alpha\le\alpha_{0}).
  \label{eq:VV_rate_alpha}
\end{equation}
Consequently, there is a constant $C_{\mathrm{VV}}>0$, with the same
dependence, such that, for every $0<\eps\le1$,
\begin{equation}
  \|V^{\eps}-V\|_{L^{\infty}(G\times[0,T])}
  \le C_{\mathrm{VV}}\sqrt{\eps T}.
  \label{eq:VV_rate}
\end{equation}
Equivalently, for a fixed horizon one may write the right-hand side as
$C_{T}\sqrt{\eps}$.

If, in addition,
\[
  V\in W^{2,\infty}(G\times(0,T))\cap C(G\times[0,T]),
\]
where the Sobolev space is understood in space-time local charts, then
\begin{equation}
  \|V^{\eps}-V\|_{L^{\infty}(G\times[0,T])}
  \le
  \eps T\|\Lap_{G}V\|_{L^{\infty}(G\times(0,T))}
  \le C_{G}\eps T\|V\|_{W^{2,\infty}(G\times(0,T))}.
  \label{eq:VV_smooth_rate}
\end{equation}
\end{theorem}

\begin{remark}
The bound \eqref{eq:VV_rate} is an analogue of the classical vanishing-viscosity estimate of Crandall and Lions \cite{CrandallLions1984}, established there on $\R^{n}$. We are not aware of a reference stating it on a compact Lie group, and we therefore give a complete proof in \cref{app:VV_proof}. The transplant of the Euclidean argument is intrinsic and rests on two group-specific ingredients. The uniform Lipschitz bound \eqref{eq:VV_global_Lip_constant} is obtained from the stochastic representation of $V^{\eps}$ by a synchronous coupling of two controlled diffusions, driven by a common noise that acts by right translations and is therefore isometric for the bi-invariant metric. In the doubling-of-variables comparison, the bi-invariance of the metric makes the two body covectors generated by the penalization $\dist_{G}^{2}/(2\alpha)$ coincide, so that only the $G$-modulus of the Hamiltonian enters, while $\dist_{G}^{2}$ is smooth on the uniform diagonal neighbourhood $\mathcal U_{\iota}:=\{(g,h)\in G\times G:\dist_{G}(g,h)<\iota\}$, where $\iota>0$ is the injectivity radius, positive by compactness. All constants depend only on the data of Assumption~\ref{ass:standing}, the horizon $T$, and curvature and injectivity-radius bounds of $(G,\dist_{G})$.
\end{remark}

\subsection{IMEX one-step operators and the time-discrete viscous value function}
\label{ssec:Veps_delta}

We split the viscous HJB equation by performing one inviscid Bellman step followed by one diffusion step. The diffusion step may be realized implicitly, by one backward-Euler resolvent, or exactly, by the heat semigroup, giving the two \emph{viscous IMEX semi-Lagrangian operators}
\begin{equation}
\mathcal T^{\eps}_{\delta}:=(I-\delta\eps\Lap_{G})^{-1}\,\mathcal T_{\delta},
\quad
\mathcal E^{\eps}_{\delta}:=e^{\delta\eps\Lap_{G}}\,\mathcal T_{\delta},
\label{eq:Teps_def}
\end{equation}
both mapping $C(G)\to C(G)$. In the representation spectral setting the two diffusion factors are the diagonal Peter-Weyl multipliers $(1+\delta\eps c_{\pi})^{-1}$ and $e^{-\delta\eps c_{\pi}}$, of identical cost.
Notice that the resolvent $\mathcal T^{\eps}_{\delta}$, requiring only a single elliptic solve, generalises to non-spectral discretizations.
We state the results of this subsection for $\mathcal T^{\eps}_{\delta}$, and they hold verbatim for $\mathcal E^{\eps}_{\delta}$.
The proofs below derive the sup-norm non-expansiveness and monotonicity of $\mathcal T^{\eps}_{\delta}$ from the Markov properties of the heat semigroup, which $\mathcal E^{\eps}_{\delta}$ inherits directly.

\begin{lemma}
\label{lem:resolvent_Linf}
For every $a>0$,
$(I-a\Lap_{G})^{-1}:L^{\infty}(G)\to L^{\infty}(G)$ is well defined
and satisfies
\begin{equation}
\|(I-a\Lap_{G})^{-1}f\|_{L^{\infty}(G)}\;\le\;\|f\|_{L^{\infty}(G)}
\quad\forall f\in L^{\infty}(G).
\label{eq:resolvent_bound}
\end{equation}
\end{lemma}

\begin{proof}
Since $-\Lap_{G}\ge 0$ is non-negative self-adjoint, the Bochner formula gives
\begin{equation}
(I-a\Lap_{G})^{-1}f
=\int_{0}^{\infty}e^{-r}\,e^{ra\Lap_{G}}f\,\mathrm dr \ .
\label{eq:resolvent_Bochner}
\end{equation}
On a compact Riemannian manifold the heat semigroup is Markov, so that
\[
\|e^{t\Lap_{G}}\|_{L^{\infty}\to L^{\infty}}=1,\quad t\ge 0.
\]
Integrating against $e^{-r}\,\mathrm dr$ yields \eqref{eq:resolvent_bound}.
\end{proof}

\begin{corollary}
\label{prop:Teps_stab}
Under Assumption~\ref{ass:standing}, for every $\delta\ge 0$ and every
$\eps\ge 0$,
\begin{equation}
\|\mathcal T^{\eps}_{\delta}f-\mathcal T^{\eps}_{\delta}g\|_{L^{\infty}(G)}\;\le\;\|f-g\|_{L^{\infty}(G)}
\quad\forall f,g\in C(G).
\label{eq:Teps_nonexp}
\end{equation}
\end{corollary}

Moreover, the IMEX operator is also \emph{monotone}.

\begin{lemma}
\label{lem:Teps_monotone}
Under Assumption~\ref{ass:standing}, $\mathcal T^{\eps}_{\delta}$ is monotone, i.e.,  $f\le g$ pointwise implies $\mathcal T^{\eps}_{\delta}f\le\mathcal T^{\eps}_{\delta}g$ pointwise, for all $f,g\in C(G)$.
\end{lemma}

\begin{proof}
The Bellman operator $\mathcal T_{\delta}$ is monotone, being a pointwise infimum of the order-preserving maps $f\mapsto\delta\,\ell(\cdot,u)+f(\Phi^{u}_{\delta}(\cdot))$ (\cref{prop:Tdelta_properties}). The resolvent $(I-a \Lap_G)^{-1}$ is order-preserving, since by the Bochner representation \eqref{eq:resolvent_Bochner} it is an average of the heat operators $e^{ra\Lap_{G}}$, each of which is positivity-preserving by the Markov property. The composition of two monotone maps is monotone.
\end{proof}

Let us now consider the time discretization scheme, where the \emph{time-discrete viscous value function}
$V^{\eps}_{\delta}:G\times\{t_{k}\}_{k=0}^{K}\to\R$ is defined by
$V^{\eps}_{\delta}(\cdot,t_{K})=\phi$ and
\begin{equation}
V^{\eps}_{\delta}(\cdot,t_{k})=\mathcal T^{\eps}_{\delta}\,V^{\eps}_{\delta}(\cdot,t_{k+1}),
\quad k=K-1,\dots,0.
\label{eq:Veps_delta}
\end{equation}

Replacing $\mathcal T^{\eps}_{\delta}$ by $\mathcal E^{\eps}_{\delta}$ in this recursion gives the exponential variant of the scheme.

\subsection{Filtered Peter-Weyl approximation, and viscous filtered spectral scheme}
\label{ssec:filter}

We now define the filter that replaces the orthogonal projector $\Pi_{N}$. It is the heat semigroup at time $N^{-2}$,
\begin{equation}
\Pi_{N}^{\mathrm{heat}}:=e^{\Lap_{G}/N^{2}},
\label{eq:filter_def}
\end{equation}
the central convolution with the heat kernel $p_{1/N^{2}}$, where $p_{a}$ denotes the integral kernel of $e^{a\Lap_{G}}$, a probability density on $G$. On the Peter-Weyl decomposition it acts as the diagonal multiplier $e^{-c_{\pi}/N^{2}}$, so its application costs one blockwise scaling of the coefficients. The following properties drive the convergence analysis.

\begin{lemma}
\label{lem:filter_heat}
Let $N\ge1$.
\begin{enumerate}[label=\textnormal{(\roman*)},leftmargin=2em]
\item $\Pi_{N}^{\mathrm{heat}}$ is a Markov operator, in the sense that it is positivity preserving, it preserves constants, it is monotone, and for all $f,h\in C(G)$,
\begin{equation}
\|\Pi_{N}^{\mathrm{heat}}f-\Pi_{N}^{\mathrm{heat}}h\|_{L^{\infty}(G)}\le\|f-h\|_{L^{\infty}(G)},
\quad
\Lip(\Pi_{N}^{\mathrm{heat}}f)\le\Lip(f).
\label{eq:filter_Linfty_stab}
\end{equation}
\item For every $f\in\Lip(G)$,
\begin{equation}
\|f-\Pi_{N}^{\mathrm{heat}}f\|_{L^{\infty}(G)}\le\sqrt{C_{h}}\,\Lip(f)\,N^{-1},
\label{eq:filter_lip}
\end{equation}
where $C_{h}=C_{h}(G)$ is the small-time second-moment constant of the heat kernel, so that $\int_{G}p_{a}(g,h)\,\dist_{G}(g,h)^{2}\,\mathrm d\mu(h)\le C_{h}\,a$ for $a\in(0,1]$ (see, e.g., \cite{Grigoryan2009}).
\item For every $s>d/2$ and $\eta\in(0,s-d/2)$ there is $C=C(G,s,\eta)>0$ such that, for every $f\in H^{s}(G)$,
\begin{equation}
\|f-\Pi_{N}^{\mathrm{heat}}f\|_{L^{\infty}(G)}\le C\,N^{-\beta_{h}(s,\eta)}\,\|f\|_{H^{s}(G)},
\quad
\beta_{h}(s,\eta):=\min\{s-d/2-\eta,\,2\}.
\label{eq:filter_approx}
\end{equation}
\end{enumerate}
\end{lemma}

\begin{proof}
Positivity, constant preservation, monotonicity and the sup-norm bound in (i) are immediate for the convolution with a probability density, and the Lipschitz bound follows from the bi-invariance of the metric, since
\[
|\Pi_{N}^{\mathrm{heat}}f(g)-\Pi_{N}^{\mathrm{heat}}f(h)|\le\int_{G}p_{1/N^{2}}(e,z)\,|f(gz)-f(hz)|\,\mathrm d\mu(z)\le\Lip(f)\,\dist_{G}(g,h).
\]
For (ii), the Cauchy-Schwarz inequality gives
\[
|f(g)-\Pi_{N}^{\mathrm{heat}}f(g)|\le\Lip(f)\int_{G}p_{1/N^{2}}(g,h)\,\dist_{G}(g,h)\,\mathrm d\mu(h)\le\sqrt{C_{h}}\,\Lip(f)\,N^{-1}.
\]
For (iii), set $\sigma:=d/2+\eta$ and $\theta:=\min\{(s-\sigma)/2,\,1\}$. On the $\pi$-block the multiplier of $I-\Pi_{N}^{\mathrm{heat}}$ is $1-e^{-c_{\pi}/N^{2}}\le\min\{1,\,c_{\pi}/N^{2}\}\le(1+c_{\pi})^{\theta}N^{-2\theta}$, so by \eqref{eq:Hs_PW} and since $\sigma+2\theta\le s$,
\[
\|f-\Pi_{N}^{\mathrm{heat}}f\|_{H^{\sigma}(G)}\le N^{-2\theta}\,\|f\|_{H^{s}(G)},
\quad 2\theta=\beta_{h}(s,\eta).
\]
The Sobolev embedding $H^{\sigma}(G)\hookrightarrow L^{\infty}(G)$, valid since $\sigma>d/2$, concludes the proof.
\end{proof}

The operator $\Pi_{N}^{\mathrm{heat}}$ is not finite-rank, 
as its multiplier $e^{-c_{\pi}/N^{2}}$ is strictly positive on the whole dual $\widehat G$. 
Classical finite-rank filters of de la Vall\'ee-Poussin type, built from a smooth window equal to one on $[0,1]$ and supported in $[0,2]$, would avoid this defect, and they satisfy the approximation properties (ii) and (iii) of \cref{lem:filter_heat} even with the unsaturated exponent $s-d/2-\eta$ (see, e.g., \cite{FilbirMhaskar2010}). 
However, their $L^{\infty}$ operator norm generically exceeds one, and any constant above one is amplified geometrically over the $K=T/\delta$ steps of the iteration. 
We therefore retain the heat filter and implement it by a side truncation $\Pi_{M}\,e^{\Lap_{G}/N^{2}}$ at $M=M(N)\ge N$ chosen so that the tail beyond $X_{M}$ is negligible compared with the spectral term. 
The next proposition quantifies this truncation.

\begin{proposition}
\label{prop:finite_rank_heat}
For $N\ge1$ and $M\ge N$, set
\[
H_{N}:=e^{\Lap_{G}/N^{2}},
\quad
A_{N,M}:=\Pi_{M}H_{N}=H_{N}\Pi_{M},
\quad
R_{N,M}:=H_{N}-A_{N,M}.
\]
Then $A_{N,M}$ has range in the finite-dimensional space $X_{M}$. Moreover, there is a constant $C_{G}>0$ such that
\begin{equation}
\eta_{N,M}:=\|R_{N,M}\|_{L^{\infty}(G)\to L^{\infty}(G)}
\le C_{G}N^{d/2}\exp\!\left(-\frac{M^{2}}{2N^{2}}\right).
\label{eq:finite_heat_tail}
\end{equation}
Consequently,
\begin{equation}
\|A_{N,M}\|_{L^{\infty}\to L^{\infty}}\le 1+\eta_{N,M}.
\label{eq:finite_heat_stability}
\end{equation}
For every $f\in\Lip(G)$,
\begin{equation}
\|f-A_{N,M}f\|_{L^{\infty}}
\le \sqrt{C_{h}}\,\Lip(f)N^{-1}+\eta_{N,M}\|f\|_{L^{\infty}},
\label{eq:finite_heat_lip}
\end{equation}
and, for every $s>d/2$ and $\eta_{0}\in(0,s-d/2)$,
\begin{equation}
\|f-A_{N,M}f\|_{L^{\infty}}
\le C(G,s,\eta_{0})N^{-\beta_{h}(s,\eta_{0})}\|f\|_{H^{s}}
+\eta_{N,M}\|f\|_{L^{\infty}},
\label{eq:finite_heat_Hs}
\end{equation}
where $\beta_{h}(s,\eta_{0}):=\min\{s-d/2-\eta_{0},\,2\}$.
\end{proposition}

\begin{proof}
The convolution kernel of $R_{N,M}$ is
\[
r_{N,M}(g)=\sum_{\sqrt{c_{\pi}}>M}e^{-c_{\pi}/N^{2}}d_{\pi}\chi_{\pi}(g).
\]
Young's inequality, $\mu(G)=1$, and Peter-Weyl orthogonality give
\begin{align*}
\|R_{N,M}\|_{L^{\infty}\to L^{\infty}}
&\le \|r_{N,M}\|_{L^{1}}
\le \|r_{N,M}\|_{L^{2}}=\left(\sum_{\sqrt{c_{\pi}}>M}d_{\pi}^{2}e^{-2c_{\pi}/N^{2}}\right)^{1/2}.
\end{align*}
On the summation range,
$e^{-2c_{\pi}/N^{2}}\le e^{-M^{2}/N^{2}}e^{-c_{\pi}/N^{2}}$.
The diagonal heat-kernel estimate gives
\[
\sum_{\pi\in\widehat G}d_{\pi}^{2}e^{-c_{\pi}/N^{2}}
=p_{1/N^{2}}(e,e)\le C_{G}N^{d},
\]
which proves \eqref{eq:finite_heat_tail}. Since $H_{N}$ is a Markov contraction,
\eqref{eq:finite_heat_stability} follows from $A_{N,M}=H_{N}-R_{N,M}$.
Finally,
\[
\|f-A_{N,M}f\|_{\infty}
\le \|f-H_{N}f\|_{\infty}+\eta_{N,M}\|f\|_{\infty},
\]
and \eqref{eq:finite_heat_lip}--\eqref{eq:finite_heat_Hs} follow from the heat-filter estimates in \cref{lem:filter_heat}.
\end{proof}

\begin{remark}[Logarithmic oversampling]
\label{rem:finite_heat_oversampling}
Let $K=T/\delta$ and fix any $r>0$. For $N,K\ge2$, the choice
\begin{equation}
M\ge N\Bigl(2\log K+(d+2r)\log N\Bigr)^{1/2}
\label{eq:finite_heat_oversampling}
\end{equation}
implies $K\eta_{N,M}\le C_{G}N^{-r}$. Hence
$(1+\eta_{N,M})^{K}\le e^{K\eta_{N,M}}=1+o(1)$, and
$\sum_{j=0}^{K-1}(1+\eta_{N,M})^{j}=O(K)$. Thus the finite-rank side truncation has the same cumulative stability scale as the ideal Markov heat filter. Notice that the heat scale $N$ and the retained spectral band $M$ must be treated as distinct parameters, since taking $M=N$ does not yield a uniform contraction estimate from \eqref{eq:finite_heat_tail}.
\end{remark}

For $K\in\N$, $\delta=T/K$, $N\ge 1$ and $\eps>0$, the \emph{viscous filtered spectral iterate} $V^{N,\eps}_{\delta,k}:G\to\R$, $k\in\{0,1,\dots,K\}$, is defined by
\begin{equation}
	\left\{
	\begin{aligned}
&\;V^{N,\eps}_{\delta,K}:=\Pi_{N}^{\mathrm{heat}}\phi \ ,\\
&\;V^{N,\eps}_{\delta,k}:=\Pi_{N}^{\mathrm{heat}}\,\mathcal T^{\eps}_{\delta}\,V^{N,\eps}_{\delta,k+1},
\quad k=K-1,\dots,0.
\end{aligned}
\right.
\label{eq:scheme}
\end{equation}

\subsection{Convergence under minimal regularity}
\label{ssec:lowreg}

We give a convergence analysis that does not require \emph{any} a priori regularity of the viscous solution, asking only that the terminal cost be Lipschitz. 
Indeed, the viscous step applied at every IMEX iteration is itself a smoother, so that each iterate $V^{\eps}_{\delta,k}$ automatically lies in $H^{s}$ with a quantitative norm bound.
The scheme thus gives its own regularity, and no hypothesis on the Hamiltonian is needed. In exchange, the coupling between $N$, $\delta$ and $\eps$ becomes more restrictive and, because a single resolvent gains only two derivatives, the dimension is limited to $d\le 3$, a restriction that disappears in the exponential variant, whose smoothing is unsaturated. The exponential variant is the one used in the numerical experiments, and we retain the resolvent because it requires only a single elliptic solve and so applies equally to non-spectral discretizations. In either case the minimal-regularity route yields convergence, but no explicit rate against the inviscid value.

\begin{lemma}
\label{lem:resolvent_smoothing}
Let $a\in(0,1]$. For every $s\in[0,2]$,
\begin{equation}
\|(I-a\Lap_{G})^{-1}f\|_{H^{s}(G)}\le a^{-s/2}\,\|f\|_{L^{2}(G)},
\quad\forall f\in L^{2}(G),
\label{eq:resolvent_smoothing}
\end{equation}
and the exponent $s=2$ is the saturation point, i.e., a single resolvent does not map $L^{2}(G)$ boundedly into $H^{s}(G)$ for any $s>2$. The heat step is unsaturated, and for every $s\ge 0$,
\begin{equation}
\bigl\|e^{a\Lap_{G}}f\bigr\|_{H^{s}(G)}\le C_{s}\,a^{-s/2}\,\|f\|_{L^{2}(G)},
\quad C_{s}:=1+(s/2)^{s/2}e^{\,1-s/2}.
\label{eq:semigroup_smoothing}
\end{equation}
\end{lemma}

\begin{proof}
The multiplier of $(I-a\Lap_{G})^{-1}$ on the $\pi$-block is $(1+a c_{\pi})^{-1}$, so by \eqref{eq:Hs_PW},
\[
\|(I-a\Lap_{G})^{-1}f\|_{H^{s}}^{2}=\sum_{\pi}\frac{(1+c_{\pi})^{s}}{(1+a c_{\pi})^{2}}\sum_{i,j}|\widehat f^{\pi}_{ij}|^{2}\le \Bigl(\sup_{c\ge 0}\frac{(1+c)^{s}}{(1+a c)^{2}}\Bigr)\,\|f\|_{L^{2}}^{2}.
\]
For $a\le 1$ one has $1+c\le a^{-1}(1+ac)$, hence $(1+c)^{s}(1+ac)^{-2}\le a^{-s}(1+ac)^{s-2}\le a^{-s}$ for $s\le 2$, which gives \eqref{eq:resolvent_smoothing}, while for $s>2$ the ratio $(1+c)^{s}(1+ac)^{-2}$ tends to infinity with $c$, so no such bound holds. The bound \eqref{eq:semigroup_smoothing} follows in the same way by maximising the spectral multiplier $(1+c)^{s/2}e^{-ac}$ over $c\ge 0$, with no saturation.
\end{proof}

The iterates of the unprojected IMEX scheme inherit this gain. Let us denote  $C_{0}:=\|\phi\|_{L^{\infty}(G)}+T\|\ell\|_{L^{\infty}(G\times U)}$ in the following.

\begin{lemma}
\label{lem:Veps_delta_Hs}
Let $V^{\eps}_{\delta}$ be the time-discrete viscous value function of \eqref{eq:Veps_delta}, with $\delta\eps\le 1$. Then $\|V^{\eps}_{\delta,k}\|_{L^{\infty}(G)}\le C_{0}$ for all $k$ and, for every $s\in[0,2]$ and every $0\le k\le K-1$,
\begin{equation}
\|V^{\eps}_{\delta,k}\|_{H^{s}(G)}\le(\delta\eps)^{-s/2}\,C_{0}.
\label{eq:Veps_delta_Hs}
\end{equation}
For the exponential variant, in which $\mathcal E^{\eps}_{\delta}$ replaces $\mathcal T^{\eps}_{\delta}$ in \eqref{eq:Veps_delta}, the bound \eqref{eq:Veps_delta_Hs} holds for every $s\ge 0$ and $0\le k\le K-1$ with the constant $C_{s}\,(\delta\eps)^{-s/2}\,C_{0}$. The terminal iterate $V^{\eps}_{\delta,K}=\phi$ is only Lipschitz and is excluded.
\end{lemma}

\begin{proof}
The Bellman operator preserves $\|\mathcal T_{\delta}V\|_{\infty}\le\delta\|\ell\|_{\infty}+\|V\|_{\infty}$ (additive homogeneity and monotonicity, see \cref{prop:Tdelta_properties}), and the resolvent is $L^{\infty}$-non-expansive (see \cref{lem:resolvent_Linf}).  Since $V^{\eps}_{\delta,K}=\phi$, downward induction gives $\|V^{\eps}_{\delta,k}\|_{\infty}\le\|\phi\|_{\infty}+(K-k)\delta\|\ell\|_{\infty}\le C_{0}$. For $k\le K-1$, $V^{\eps}_{\delta,k}=(I-\delta\eps\Lap_{G})^{-1}g_{k}$ with $g_{k}:=\mathcal T_{\delta}V^{\eps}_{\delta,k+1}$ and $\|g_{k}\|_{\infty}\le C_{0}$. Since $\mu(G)=1$, $\|g_{k}\|_{L^{2}}\le\|g_{k}\|_{\infty}\le C_{0}$, and \cref{lem:resolvent_smoothing} with $a=\delta\eps$ yields \eqref{eq:Veps_delta_Hs}. The exponential variant is identical, using \eqref{eq:semigroup_smoothing} in place of \eqref{eq:resolvent_smoothing} and the fact that the heat step is an $L^{\infty}$ contraction.
\end{proof}

We can now bound the spectral error against the discrete viscous value $V^{\eps}_{\delta}$ and pass to the inviscid limit. 

\begin{theorem}
\label{thm:lowreg}
Under Assumption~\ref{ass:standing}, let $\phi\in\Lip(G)$. Run the heat-filtered scheme \eqref{eq:scheme} either with the resolvent operator $\mathcal T^{\eps}_{\delta}$, in which case let $d\le 3$ and fix $s\in(d/2,2]$, or with the exponential operator $\mathcal E^{\eps}_{\delta}$ in place of $\mathcal T^{\eps}_{\delta}$, both in \eqref{eq:Veps_delta} and in \eqref{eq:scheme}, in which case $d\ge 1$ is arbitrary and $s\in(d/2,\infty)$. Fix $\eta\in(0,s-d/2)$ and set $\beta=\beta_{h}(s,\eta)$ as in \eqref{eq:filter_approx}. Then for every $K\in\N$, $\delta=T/K$, $\eps>0$ with $\delta\eps\le1$, and $N\ge1$,
\begin{equation}
\max_{0\le k\le K}\bigl\|V^{N,\eps}_{\delta,k}-V^{\eps}_{\delta,k}\bigr\|_{L^{\infty}(G)}
\le C_{1}\,\Lip(\phi)\,N^{-1} + K\,C_{2}\,C_{0}\,(\delta\eps)^{-s/2}\,N^{-\beta},
\label{eq:lowreg_bound}
\end{equation}
where $C_{1}=\sqrt{C_{h}}$ and $C_{2}=C_{2}(G,s,\eta)$, the latter carrying the additional factor $C_{s}$ of \eqref{eq:semigroup_smoothing} in the exponential case. In particular, with $K=T/\delta$, the right-hand side tends to $0$ provided
\begin{equation}
N^{\beta}\;\gg\;\delta^{-1-s/2}\,\eps^{-s/2}\quad\text{and}\quad N\to\infty.
\label{eq:lowreg_coupling}
\end{equation}
If, in addition, $\delta\to0$ and $\eps\to0$ along a sequence satisfying \eqref{eq:lowreg_coupling}, then
\[
\max_{0\le k\le K}\bigl\|V^{N,\eps}_{\delta,k}-V(\cdot,t_{k})\bigr\|_{L^{\infty}(G)}\longrightarrow0.
\]
\end{theorem}

\begin{proof}
Write $E_{k}:=V^{N,\eps}_{\delta,k}-V^{\eps}_{\delta,k}$. Since $V^{N,\eps}_{\delta,k}=\Pi_{N}^{\mathrm{heat}}\mathcal T^{\eps}_{\delta}V^{N,\eps}_{\delta,k+1}$ by \eqref{eq:scheme} and $V^{\eps}_{\delta,k}=\mathcal T^{\eps}_{\delta}V^{\eps}_{\delta,k+1}$ by \eqref{eq:Veps_delta}, inserting $\Pi_{N}^{\mathrm{heat}}\mathcal T^{\eps}_{\delta}V^{\eps}_{\delta,k+1}$ and using $\mathcal T^{\eps}_{\delta}V^{\eps}_{\delta,k+1}=V^{\eps}_{\delta,k}$ gives
\[
E_{k}=\Pi_{N}^{\mathrm{heat}}\bigl(\mathcal T^{\eps}_{\delta}V^{N,\eps}_{\delta,k+1}-\mathcal T^{\eps}_{\delta}V^{\eps}_{\delta,k+1}\bigr)-(I-\Pi_{N}^{\mathrm{heat}})V^{\eps}_{\delta,k}.
\]
By the non-expansiveness \eqref{eq:Teps_nonexp} of $\mathcal T^{\eps}_{\delta}$ and the non-expansiveness of $\Pi_{N}^{\mathrm{heat}}$ in \eqref{eq:filter_Linfty_stab},
\[
\|E_{k}\|_{\infty}\le\|E_{k+1}\|_{\infty}+\|(I-\Pi_{N}^{\mathrm{heat}})V^{\eps}_{\delta,k}\|_{\infty},\quad k=0,\dots,K-1.
\]
For $k\le K-1$, the sup-norm approximation \eqref{eq:filter_approx} and \cref{lem:Veps_delta_Hs} give
$\|(I-\Pi_{N}^{\mathrm{heat}})V^{\eps}_{\delta,k}\|_{\infty}\le C_{2}\,N^{-\beta}\|V^{\eps}_{\delta,k}\|_{H^{s}}\le C_{2}\,C_{0}\,(\delta\eps)^{-s/2}N^{-\beta}$.
At the terminal step $E_{K}=-(I-\Pi_{N}^{\mathrm{heat}})\phi$, and the Lipschitz approximation \eqref{eq:filter_lip} gives $\|E_{K}\|_{\infty}\le\sqrt{C_{h}}\,\Lip(\phi)N^{-1}$. Summing the recursion from $k=K$ down to each $k$ yields \eqref{eq:lowreg_bound}. In the exponential case the identical recursion applies, since $\mathcal E^{\eps}_{\delta}$ is monotone and sup-norm non-expansive as well, and the unsaturated bound of \cref{lem:Veps_delta_Hs} admits every $s>d/2$, which removes the restrictions $s\le 2$ and $d\le 3$. Condition \eqref{eq:lowreg_coupling} makes the dominant term $K(\delta\eps)^{-s/2}N^{-\beta}=T\delta^{-1-s/2}\eps^{-s/2}N^{-\beta}$ vanish.

For the convergence, set $u_{j}:=V^{\eps}_{\delta}(\cdot,t_{K-j})$, so that the backward recursion \eqref{eq:Veps_delta} becomes the forward-in-time scheme $u_{j+1}=\mathcal T^{\eps}_{\delta}u_{j}$, $u_{0}=\phi$, regarded as a family of schemes indexed by $(\delta,\eps)$. Every member is monotone (\cref{lem:Teps_monotone}) and $L^{\infty}$-stable (\cref{prop:Teps_stab}), and the family is consistent, in the joint limit $(\delta,\eps)\to(0,0)$, with the inviscid forward equation $\partial_{\tau}u=-\Ham(g,\nabla_{G}^{\mathrm{body}}u)$ obtained from \eqref{eq:HJB} by the time reversal $\tau=T-t$. Indeed, for smooth test functions a Taylor expansion along the flows and the resolvent expansion give $\delta^{-1}(\mathcal T^{\eps}_{\delta}\psi-\psi)=-\Ham(\cdot,\nabla_{G}^{\mathrm{body}}\psi)+\eps\Lap_{G}\psi+o(1)$, and the diffusion term vanishes with $\eps$. By the Barles-Souganidis theorem \cite{BarlesSouganidis1991}, $V^{\eps}_{\delta}\to V$ locally uniformly as $(\delta,\eps)\to(0,0)$, hence uniformly on the compact $G\times[0,T]$, and combining this with \eqref{eq:lowreg_bound} along \eqref{eq:lowreg_coupling} proves the convergence. 
\end{proof}

\begin{corollary}[Finite-rank version under minimal regularity]
\label{cor:lowreg_finite_rank}
Run the scheme in its exponential variant, with the ideal heat filter $H_{N}$ replaced by its side truncation $A_{N,M}=\Pi_{M}H_{N}$ of \cref{prop:finite_rank_heat}, and denote the resulting iterates by $\widetilde V^{N,M,\eps}_{\delta,k}$. Under the hypotheses of \cref{thm:lowreg}, fix $\eta_{0}\in(0,s-d/2)$ and set $\beta_{h}=\min\{s-d/2-\eta_{0},2\}$. Then
\begin{equation}
	\begin{aligned}
\|\widetilde V^{N,M,\eps}_{\delta,0}-V^{\eps}_{\delta,0}\|_{L^{\infty}} &\le e^{K\eta_{N,M}}\Big(\sqrt{C_{h}}\,\Lip(\phi)N^{-1}+\eta_{N,M}\|\phi\|_{\infty} \\
&+K\bigl(C_{2}C_{0}(\delta\eps)^{-s/2}N^{-\beta_{h}}+\eta_{N,M}C_{0}\bigr)\Big).
\end{aligned}
\label{eq:lowreg_finite_rank}
\end{equation}
Consequently, the finite-rank iterates converge to the inviscid value function provided, in addition to $\delta,\eps\to0$, that $K(\delta\eps)^{-s/2}N^{-\beta_{h}}\to0$, $K\eta_{N,M}\to0$, and $N\to\infty$. In particular, the oversampling rule \eqref{eq:finite_heat_oversampling} enforces the second condition.
\end{corollary}

\begin{proof}
Set $E_{k}:=\widetilde V^{N,M,\eps}_{\delta,k}-V^{\eps}_{\delta,k}$. By \eqref{eq:finite_heat_stability}, the non-expansiveness of the exponential viscous Bellman step, \eqref{eq:finite_heat_Hs}, and the smoothing estimate \eqref{eq:semigroup_smoothing},
\[
\|E_{k}\|_{\infty}
\le (1+\eta_{N,M})\|E_{k+1}\|_{\infty}
+C_{2}C_{0}(\delta\eps)^{-s/2}N^{-\beta_{h}}
+\eta_{N,M}C_{0},
\]
and the terminal estimate follows from \eqref{eq:finite_heat_lip}. Iterating downward and using $(1+\eta_{N,M})^{K}\le e^{K\eta_{N,M}}$ gives \eqref{eq:lowreg_finite_rank}. If $K\eta_{N,M}\to0$, the prefactor is $1+o(1)$, and the convergence assertion follows as in the proof of \cref{thm:lowreg}.
\end{proof}

\section{A finite-rank Markov filter and uniform convergence for Lipschitz data}
\label{sec:markov}

This section contains the main results of the paper. \Cref{ssec:markov_filter} constructs a spectral filter on $G$ that is finite-rank, positivity preserving, and non-expansive in the supremum norm. 
\Cref{ssec:markov_convergence} treats the filtering step as a small centred random perturbation of the controlled dynamics and proves a convergence theorem for the resulting Fej\'er-Markov scheme with an explicit rate, valid for Lipschitz data alone, with or without artificial viscosity.  \Cref{ssec:fully_discrete_markov} extends the analysis to a fully discrete realization based on positive cubature, covering the control discretization and the Lie-group flow integration, so that the convergence theorem applies to the implemented iteration. 

\subsection{An exactly finite-rank Markov Peter-Weyl filter}
\label{ssec:markov_filter}

The analysis of \cref{ssec:lowreg} rests on a filter whose $L^{\infty}$ operator norm equals one. In a coupled limit with $K=T/\delta\to\infty$ steps, any constant exceeding one is amplified geometrically. One therefore needs a filter whose stability constant is exactly one, or differs from one by an error whose cumulative effect is controlled. The heat semigroup has the desired Markov property but is not finite-rank, and its side truncation restores contractivity only asymptotically, through the oversampling rule of \cref{rem:finite_heat_oversampling}. A compactly supported smooth multiplier, on the other hand, is finite-rank but is not, in general, positivity preserving. We now construct a filter having all three properties simultaneously, namely finite rank, positivity, and exact sup-norm contractivity.

We first record the elementary product property of Casimir spectral spaces that will be used below.

\begin{lemma}[Product property]
\label{lem:spectral_product}
There exists a constant $\kappa_G\ge1$, depending only on $G$ and on the fixed bi-invariant metric, such that
\begin{equation}
 X_R\,X_S\subset X_{\kappa_G(R+S+1)}
 \quad\text{for every }R,S\ge1.
\label{eq:spectral_product}
\end{equation}
Here $X_RX_S$ denotes the linear span of the pointwise products $fg$, with $f\in X_R$ and $g\in X_S$.
\end{lemma}

\begin{proof}
It is enough to consider products of matrix coefficients of two irreducible representations $\pi$ and $\rho$. Such a product is a matrix coefficient of $\pi\otimes\rho$, and hence decomposes into matrix coefficients of the irreducible constituents of this tensor product. We must therefore bound the Casimir eigenvalue of every irreducible constituent $\tau\subset\pi\otimes\rho$.

We recall the standard parameterization in a form adapted to this estimate. Up to a finite central quotient, a connected compact Lie group is the product of a torus and a compact connected semisimple group. The irreducible representations are described by an integral weight on the toral factor and a dominant integral highest weight on the semisimple factor. With respect to the fixed bi-invariant metric, the Casimir eigenvalue is a positive-definite quadratic expression in these weights, up to the usual linear $2\varrho$ shift on the semisimple factor. Consequently there is a constant $C_G$ such that, for every irreducible representation $\sigma$ with combined weight $\lambda_\sigma$,
\begin{equation}
 C_G^{-1}(1+|\lambda_\sigma|)
 \le 1+\sqrt{c_\sigma}
 \le C_G(1+|\lambda_\sigma|).
\label{eq:casimir_weight_equiv}
\end{equation}
The finite central quotient only restricts the admissible weight lattice and does not affect the estimate.

If $\tau$ occurs in $\pi\otimes\rho$, then the toral weight of $\tau$ is the sum of the toral weights of $\pi$ and $\rho$, and the semisimple highest weight of $\tau$ is dominated by the sum of the highest weights of $\pi$ and $\rho$. In particular
\[
 |\lambda_\tau|\le C_G\bigl(|\lambda_\pi|+|\lambda_\rho|+1\bigr).
\]
Using \eqref{eq:casimir_weight_equiv} gives
\[
 1+\sqrt{c_\tau}
 \le C_G\bigl(1+\sqrt{c_\pi}+\sqrt{c_\rho}\bigr).
\]
Thus, if $\pi\in\widehat G_R$ and $\rho\in\widehat G_S$, every constituent $\tau$ of $\pi\otimes\rho$ belongs to $\widehat G_{\kappa_G(R+S+1)}$ after increasing $\kappa_G$, which proves \eqref{eq:spectral_product}.
\end{proof}

Fix a real-valued function $\vartheta\in C_c^\infty([0,\infty))$ such that
\[
 0\le\vartheta\le1,
 \quad
 \vartheta\equiv1\text{ on }[0,1],
 \quad
 \operatorname{supp}\vartheta\subset[0,2].
\]
For $N\ge1$, define the central smooth spectral kernel
\begin{equation}
 D_N(g):=
 \sum_{\pi\in\widehat G}
 \vartheta\!\left(\frac{\sqrt{c_\pi}}{N}\right)d_\pi\chi_\pi(g),
 \quad
 Z_N:=\|D_N\|_{L^2(G)}^2,
\label{eq:DN_def}
\end{equation}
and the Fej\'er-type kernel
\begin{equation}
 F_N(g):=\frac{|D_N(g)|^2}{Z_N}.
\label{eq:positive_kernel}
\end{equation}
We denote convolution with $F_N$ by
\begin{equation}
 \mathcal M_Nf:=f*F_N.
\label{eq:MN_def}
\end{equation}

\begin{theorem}[Finite-rank Markov filter]
\label{thm:finite_rank_markov_filter}
There are constants $\kappa\ge1$ and $C_q<\infty$, depending only on $G$, the metric, $\vartheta$, and $q$, such that, for every $N\ge1$, the following properties hold.
\begin{enumerate}[label=\textnormal{(\roman*)},leftmargin=2em]
\item $F_N$ is central, non-negative, inversion invariant, and has unit mass:
\[
 F_N(g)\ge0,
 \quad F_N(g^{-1})=F_N(g),
 \quad \int_GF_N\,\mathrm d\mu=1.
\]
\item $F_N\in X_{\kappa N}$.  Consequently $\mathcal M_N$ has finite-dimensional range contained in $X_{\kappa N}$.
\item For every $L>0$ there is $C_L<\infty$ such that
\begin{equation}
 0\le F_N(g)
 \le C_L\frac{N^d}{(1+N\dist_G(g,e))^L}.
\label{eq:FN_localization}
\end{equation}
In particular, for $q=1,2$,
\begin{equation}
 \int_G\dist_G(g,e)^qF_N(g)\,\mathrm d\mu(g)
 \le C_qN^{-q}.
\label{eq:FN_moments}
\end{equation}
\item $\mathcal M_N$ is a Markov operator.  More precisely, it is linear, positivity preserving, constant preserving, monotone, and
\begin{equation}
 \|\mathcal M_Nf-\mathcal M_Nh\|_{L^\infty}
 \le\|f-h\|_{L^\infty}.
\label{eq:MN_contract}
\end{equation}
Moreover,
\begin{equation}
 \Lip(\mathcal M_Nf)\le\Lip(f),
 \quad
 \|\mathcal M_Nf-f\|_{L^\infty}
 \le C_1N^{-1}\Lip(f)
\label{eq:MN_Jackson}
\end{equation}
for every $f\in\Lip(G)$.
\end{enumerate}
\end{theorem}

\begin{proof}
We prove the four assertions separately.

First, let $\pi^*$ denote the contragredient representation. Since $c_{\pi^*}=c_\pi$ and $\chi_{\pi^*}(g)=\overline{\chi_\pi(g)}=\chi_\pi(g^{-1})$, the summands corresponding to $\pi$ and $\pi^*$ occur in \eqref{eq:DN_def} with the same real coefficient. Therefore $D_N$ is real-valued and satisfies $D_N(g^{-1})=D_N(g)$. It is central because it is a linear combination of characters. Hence $F_N=|D_N|^2/Z_N$ is central, non-negative and inversion invariant. Moreover,
\[
 \int_G F_N\,d\mu=Z_N^{-1}\int_G |D_N|^2\,d\mu=1,
\]
which proves (i).

Second, the multiplier defining $D_N$ is supported in $\{\sqrt{c_\pi}\le 2N\}$, hence $D_N\in X_{2N}$. Since $D_N$ is real-valued, $|D_N|^2=D_N^2$ is a product of two functions in $X_{2N}$. The product property \cref{lem:spectral_product} gives
\[
 |D_N|^2\in X_{\kappa_G(4N+1)}\subset X_{\kappa N}
\]
for a larger group-dependent constant $\kappa$. This proves (ii), and convolution with $F_N$ has range contained in the finite-dimensional space generated by the spectral support of $F_N$.

Third, by Peter-Weyl orthogonality,
\[
 Z_N=\sum_{\pi\in\widehat G}
 \vartheta\!\left(\frac{\sqrt{c_\pi}}{N}\right)^2d_\pi^2.
\]
Since $\vartheta\equiv1$ on $[0,1]$ and is supported in $[0,2]$, the Weyl estimate for the Casimir spectral counting function yields
\begin{equation}
 cN^d\le Z_N\le CN^d.
\label{eq:ZN_Weyl}
\end{equation}
We use the standard localization estimate for $C_c^\infty$ spectral multipliers of the Casimir operator on compact Lie groups, or more generally on compact manifolds with Gaussian heat-kernel bounds (see, e.g., \cite{FilbirMhaskar2010,Grigoryan2009}). Applied to the multiplier $\lambda\mapsto\vartheta(\sqrt\lambda/N)$, it gives, for every $A>0$,
\begin{equation}
 |D_N(g)|\le C_A\frac{N^d}{(1+N\dist_G(g,e))^A},
 \quad g\in G,
\label{eq:DN_localization}
\end{equation}
with constants independent of $N$. Combining \eqref{eq:DN_localization} with the lower bound in \eqref{eq:ZN_Weyl}, and then choosing $A$ sufficiently large, gives \eqref{eq:FN_localization}. Finally, integrating \eqref{eq:FN_localization} over $B(e,N^{-1})$ and over the dyadic annuli
\[
 \{2^j/N\le \dist_G(g,e)<2^{j+1}/N\},\quad j\ge0,
\]
and taking $L>d+q$ yields \eqref{eq:FN_moments}.

Fourth, convolution against a probability density is positivity preserving, constant preserving and an $L^\infty$ contraction. Thus $\mathcal M_N$ is monotone and satisfies \eqref{eq:MN_contract}. The bi-invariance of the metric gives
\[
 |\mathcal M_Nf(g)-\mathcal M_Nf(h)|
 \le\int_GF_N(z)|f(gz)-f(hz)|\,d\mu(z)
 \le\Lip(f)\dist_G(g,h),
\]
which proves the Lipschitz contraction. Similarly,
\[
 |\mathcal M_Nf(g)-f(g)|
 \le\Lip(f)\int_GF_N(z)\dist_G(z,e)\,d\mu(z),
\]
and the Jackson estimate in \eqref{eq:MN_Jackson} follows from \eqref{eq:FN_moments} with $q=1$.
\end{proof}

\begin{remark}[Fourier multipliers]
\label{rem:positive_filter_multiplier}
The kernel $F_N$ is central, hence $\mathcal M_N$ acts by a scalar multiplier $m_N(\tau)$ on every irreducible block.  If $N_{\pi\rho}^{\tau}$ denotes the multiplicity of $\tau$ in $\pi\otimes\rho$, then
\begin{equation}
 m_N(\tau)=\frac{1}{Z_Nd_\tau}
 \sum_{\pi,\rho\in\widehat G}
 \vartheta\!\left(\frac{\sqrt{c_\pi}}{N}\right)
 \vartheta\!\left(\frac{\sqrt{c_\rho}}{N}\right)
 d_\pi d_\rho N_{\pi\rho}^{\tau}.
\label{eq:positive_filter_multiplier}
\end{equation}
The sum is finite and $m_N(\tau)=0$ for $\sqrt{c_\tau}>\kappa N$.  Thus the Markov filter is implemented by the same blockwise scalar multiplication as the heat and de la Vall\'ee-Poussin filters.  On $\SO{3}$ the multiplicities are given by the Clebsch-Gordan rule and can be precomputed once for each $N$.
\end{remark}

We combine the finite-rank Markov filter with the viscous heat step.  Put
\begin{equation}
 \mathcal Q_{\delta,N}^{\eps}
 :=\mathcal M_Ne^{\delta\eps\Lap_G}
 =e^{\delta\eps\Lap_G}\mathcal M_N.
\label{eq:Q_def}
\end{equation}
The two factors commute because both are central spectral multipliers.  The operator $\mathcal Q_{\delta,N}^{\eps}$ is again finite-rank and Markov.  Its convolution kernel is
\[
 q_{\delta,N}^{\eps}:=F_N*p_{\delta\eps},
\]
where $p_a$ is the heat kernel at time $a$.  It is central and inversion invariant.  From \eqref{eq:FN_moments}, the heat-kernel second-moment estimate, and
$\dist_G(e,xy)^2\le2\dist_G(e,x)^2+2\dist_G(e,y)^2$, we obtain
\begin{equation}
 m_2(q_{\delta,N}^{\eps})
 :=\int_G\dist_G(e,z)^2q_{\delta,N}^{\eps}(z)\,\mathrm d\mu(z)
 \le C_G\bigl(\delta\eps+N^{-2}\bigr).
\label{eq:Q_second_moment}
\end{equation}
This suggests the effective noise intensity
\begin{equation}
 \nu_{\delta,N}^{\eps}
 :=\eps+\frac{1}{\delta N^2},
 \quad
 \delta\nu_{\delta,N}^{\eps}=\delta\eps+N^{-2}.
\label{eq:effective_viscosity}
\end{equation}
For the heat filter $e^{\Lap_G/N^2}$ this interpretation is exact at the operator level, since
\[
 e^{\Lap_G/N^2}e^{\delta\eps\Lap_G}
 =e^{\delta\nu_{\delta,N}^{\eps}\Lap_G}.
\]
For the finite-rank Markov filter it remains exact at the scale of the one-step second moment, which is all that is needed in the coupling argument below.

We formulate the scheme in a way that makes its dynamic-programming interpretation transparent.  Define the prefiltered values $U_{\delta,N,k}^{\eps}$ and the stored band-limited continuations $Z_{\delta,N,k}^{\eps}$ by
\begin{equation}
 \left\{
 \begin{aligned}
 &U_{\delta,N,K}^{\eps}:=\phi,
 \qquad
 Z_{\delta,N,K}^{\eps}:=\mathcal Q_{\delta,N}^{\eps}\phi,\\
 &U_{\delta,N,k}^{\eps}:=\mathcal T_\delta Z_{\delta,N,k+1}^{\eps},
 \qquad
 Z_{\delta,N,k}^{\eps}:=\mathcal Q_{\delta,N}^{\eps}U_{\delta,N,k}^{\eps},
 \quad k=K-1,\ldots,0.
 \end{aligned}
 \right.
\label{eq:Markov_filtered_scheme}
\end{equation}
Equivalently,
\[
 U_{\delta,N,k}^{\eps}
 =\mathcal T_\delta\mathcal Q_{\delta,N}^{\eps}
 U_{\delta,N,k+1}^{\eps},
 \quad
 Z_{\delta,N,k}^{\eps}
 =\mathcal Q_{\delta,N}^{\eps}\mathcal T_\delta
 Z_{\delta,N,k+1}^{\eps}.
\]
Only $Z_{\delta,N,k}^{\eps}$ is stored, and it belongs to $X_{\kappa N}$.  The quantity $U_{\delta,N,k}^{\eps}$ is the value before the filtering/noise step and is obtained from the stored continuation by one Bellman evaluation.  This distinction removes the apparent reversal between ``filter after Bellman'' in the coefficient recursion and ``noise after control'' in the stochastic dynamic-programming interpretation.

\subsection{A quantitative convergence theorem}
\label{ssec:markov_convergence}

The main estimate uses a controlled-chain coupling.  We state it for a general small symmetric Markov perturbation, because the same result will also cover the fully discrete cubature scheme of \cref{ssec:fully_discrete_markov}.

Choose once and for all a faithful finite-dimensional unitary representation
$\rho:G\to U(r)$.  Since $G$ is compact, there are constants $c_\rho,C_\rho>0$ such that
\begin{equation}
 c_\rho\dist_G(g,h)
 \le\|\rho(g)-\rho(h)\|_{\mathrm F}
 \le C_\rho\dist_G(g,h),
 \quad g,h\in G.
\label{eq:faithful_embedding_metric}
\end{equation}

\begin{lemma}[Controlled small-noise comparison]
\label{lem:controlled_small_noise}
Let $P$ be a Feller Markov transition kernel on $G$, so that $Pf\in C(G)$ for every $f\in C(G)$. Suppose that for some $a\in(0,1]$ and every $y\in G$,
\begin{align}
 \int_G\|\rho(z)-\rho(y)\|_{\mathrm F}^2P(y,\mathrm dz)&\le C_Pa,
\label{eq:transition_variance}\\
 \left\|\int_G\bigl(\rho(z)-\rho(y)\bigr)P(y,\mathrm dz)\right\|_{\mathrm F}
 &\le C_Pa.
\label{eq:transition_drift}
\end{align}
Let $U^P_{\delta,K}=\phi$ and
\begin{equation}
 U^P_{\delta,k}(g)
 :=\inf_{u\in U}\left\{
 \delta\ell(g,u)+
 (PU^P_{\delta,k+1})(\Phi_\delta^u(g))
 \right\}.
\label{eq:noisy_DPP}
\end{equation}
Then, for $\delta\in(0,1]$ and $K\delta=T$,
\begin{equation}
 \max_{0\le k\le K}
 \|U^P_{\delta,k}-V_\delta(\cdot,t_k)\|_{L^\infty(G)}
 \le C_T\left(\sqrt{\frac{a}{\delta}}+\frac{a}{\delta}\right),
\label{eq:small_noise_DPP_bound}
\end{equation}
where $C_T$ depends only on the data in \cref{ass:standing}, $T$, $G$, $\rho$, and $C_P$.
\end{lemma}

\begin{proof}
The recursion \eqref{eq:noisy_DPP} is the dynamic programming equation of a controlled Markov chain. At state $X_n$, a control $u_n$ is chosen, the deterministic transport point is
\[
 Y_{n+1}=\Phi_\delta^{u_n}(X_n),
\]
the running cost $\delta\ell(X_n,u_n)$ is paid, and $X_{n+1}$ is sampled from $P(Y_{n+1},\cdot)$. We allow policies to be history-dependent and randomized. This does not change the deterministic value $V_\delta$. Since $V_\delta$ satisfies the pointwise dynamic programming recursion, randomization and history dependence cannot lower the deterministic infimum, and deterministic Markov selectors are optimal up to an arbitrarily small error.

Fix an admissible policy for the noisy chain. On the same probability space define the noise-free trajectory driven by the same realized controls,
\[
 \overline X_{n+1}=\Phi_\delta^{u_n}(\overline X_n),
 \quad \overline X_0=X_0.
\]
We first compare $X_n$ and $\overline X_n$. For a fixed constant control $u$, the embedded curve $s\mapsto\rho(\Phi_s^u(g))$ solves
\[
 \frac{d}{ds}\rho(\Phi_s^u(g))
 =\rho(\Phi_s^u(g))\,d\rho\bigl(\xi(\Phi_s^u(g),u)\bigr).
\]
Using the Lipschitz bound on $\xi$, the equivalence of the embedded and Riemannian distances in \eqref{eq:faithful_embedding_metric}, and Gronwall's inequality, we obtain the one-step Lipschitz estimate
\begin{equation}
 \|\rho(\Phi_\delta^u(g))-\rho(\Phi_\delta^u(h))\|_{\mathrm F}
 \le(1+C\delta)\|\rho(g)-\rho(h)\|_{\mathrm F}.
\label{eq:flow_embedding_lip}
\end{equation}
Set $D_n:=\rho(X_n)-\rho(\overline X_n)$ and
\[
 \widetilde Y_{n+1}:=\Phi_\delta^{u_n}(\overline X_n).
\]
Then
\[
 D_{n+1}
 =\underbrace{\rho(Y_{n+1})-\rho(\widetilde Y_{n+1})}_{A_{n+1}}
 +\underbrace{\rho(X_{n+1})-\rho(Y_{n+1})}_{B_{n+1}}.
\]
By \eqref{eq:flow_embedding_lip}, $\|A_{n+1}\|_{\mathrm F}\le(1+C\delta)\|D_n\|_{\mathrm F}$. By \eqref{eq:transition_variance}--\eqref{eq:transition_drift}, conditionally on the past $\mathcal F_n$,
\[
 \mathbb E(\|B_{n+1}\|_{\mathrm F}^2\mid\mathcal F_n)\le C_Pa,
 \quad
 \|\mathbb E(B_{n+1}\mid\mathcal F_n)\|_{\mathrm F}\le C_Pa.
\]
Expanding $\|A_{n+1}+B_{n+1}\|^2$ and estimating the cross term gives
\[
 2\langle A_{n+1},\mathbb E(B_{n+1}\mid\mathcal F_n)\rangle
 \le C a\|D_n\|_{\mathrm F}
 \le \delta\|D_n\|_{\mathrm F}^2+C\frac{a^2}{\delta}.
\]
Consequently
\begin{equation}
 \mathbb E\bigl[\|D_{n+1}\|_{\mathrm F}^2\mid\mathcal F_n\bigr]
 \le(1+C\delta)\|D_n\|_{\mathrm F}^2
 +C\left(a+\frac{a^2}{\delta}\right).
\label{eq:coupling_recursion}
\end{equation}
Since $D_0=0$, discrete Gronwall and $n\delta\le T$ yield
\[
 \max_{0\le n\le K}\mathbb E\|D_n\|_{\mathrm F}^2
 \le C_T\left(\frac{a}{\delta}+\frac{a^2}{\delta^2}\right).
\]
Using \eqref{eq:faithful_embedding_metric} and Jensen's inequality,
\begin{equation}
 \max_{0\le n\le K}\mathbb E\dist_G(X_n,\overline X_n)
 \le C_T\left(\sqrt{\frac{a}{\delta}}+\frac{a}{\delta}\right).
\label{eq:state_coupling}
\end{equation}

The Lipschitz estimates on $\ell$ and $\phi$ now transfer \eqref{eq:state_coupling} to costs. Let $J^P(g,k,\varpi)$ be the expected noisy cost generated by a policy $\varpi$ from state $g$ at step $k$, and let $\overline J(g,k,\varpi)$ be the cost of the coupled noise-free trajectory driven by the same realized controls. Then
\[
 |J^P(g,k,\varpi)-\overline J(g,k,\varpi)|
 \le C_T\left(\sqrt{\frac{a}{\delta}}+\frac{a}{\delta}\right).
\]
For every realization of the controls, the noise-free cost is bounded below by the deterministic discrete value $V_\delta(g,t_k)$. Hence, taking expectations and then the infimum over noisy policies gives
\[
 U^P_{\delta,k}(g)\ge V_\delta(g,t_k)
 -C_T\left(\sqrt{\frac{a}{\delta}}+\frac{a}{\delta}\right).
\]
Conversely, fix $\eta>0$. Since the discrete problem \eqref{eq:DPP_discrete} is deterministic, there is an open-loop control sequence $(u^{\eta}_{k},\dots,u^{\eta}_{K-1})\in U^{K-k}$ whose noise-free trajectory issued from $g$ has cost at most $V_\delta(g,t_k)+\eta$. We use this sequence, which does not depend on the state, as an admissible policy for the noisy chain. The coupled noise-free trajectory driven by the same controls is then precisely the $\eta$-optimal deterministic trajectory, so its cost is at most $V_\delta(g,t_k)+\eta$, and the coupling estimate gives
\[
 U^P_{\delta,k}(g)\le V_\delta(g,t_k)+\eta
 +C_T\left(\sqrt{\frac{a}{\delta}}+\frac{a}{\delta}\right).
\]
Letting $\eta\downarrow0$ proves \eqref{eq:small_noise_DPP_bound}.
\end{proof}

\begin{theorem}[Finite-rank convergence for Lipschitz data]
\label{thm:finite_rank_lipschitz_rate}
Under Assumption~\ref{ass:standing}, let $U_{\delta,N,k}^{\eps}$ and $Z_{\delta,N,k}^{\eps}$ be defined by \eqref{eq:Markov_filtered_scheme}.  There exists $C_T<\infty$, independent of $\delta$, $N$, and $\eps$, such that, for 
\[
 0<\delta\le1,
 \quad
 0\le\nu_{\delta,N}^{\eps}
 =\eps+(\delta N^2)^{-1}\le1,
\]
one has
\begin{align}
 \max_{0\le k\le K}
 \|U_{\delta,N,k}^{\eps}-V(\cdot,t_k)\|_{L^\infty(G)}
 &\le C_T\left(
 \sqrt\delta+\sqrt{\nu_{\delta,N}^{\eps}}
 +\nu_{\delta,N}^{\eps}
 \right),
\label{eq:main_unconditional_U}\\
 \max_{0\le k\le K}
 \|Z_{\delta,N,k}^{\eps}-V(\cdot,t_k)\|_{L^\infty(G)}
 &\le C_T\left(
 \sqrt\delta+\sqrt{\nu_{\delta,N}^{\eps}}
 +\nu_{\delta,N}^{\eps}
 \right).
\label{eq:main_unconditional_Z}
\end{align}
In particular, the scheme converges uniformly to the viscosity solution under the three conditions
\begin{equation}
 \delta\longrightarrow0,
 \quad
 \eps\longrightarrow0,
 \quad
 \delta N^2\longrightarrow\infty.
\label{eq:unconditional_coupling}
\end{equation}
\end{theorem}

\begin{proof}
Let $P_{\delta,N}^{\eps}$ be the transition kernel corresponding to convolution by $q_{\delta,N}^{\eps}$.  It is central and inversion invariant.  Put
\[
 A_\rho:=\int_G\rho(z)q_{\delta,N}^{\eps}(z)\,\mathrm d\mu(z).
\]
Inversion invariance gives $A_\rho=A_\rho^*$ and hence the exact identity
\[
 I_r-A_\rho
 =\frac12\int_G(\rho(z)-I_r)^*(\rho(z)-I_r)
 q_{\delta,N}^{\eps}(z)\,\mathrm d\mu(z).
\]
Consequently,
\[
 \|I_r-A_\rho\|
 \le C\int_G\|\rho(z)-I_r\|_{\mathrm F}^2
 q_{\delta,N}^{\eps}(z)\,\mathrm d\mu(z)
 \le C\bigl(\delta\eps+N^{-2}\bigr).
\]
Together with \eqref{eq:Q_second_moment}, this verifies
\eqref{eq:transition_variance}--\eqref{eq:transition_drift} with
$a=\delta\eps+N^{-2}=\delta\nu_{\delta,N}^{\eps}\le1$, the constants being absorbed into $C_{P}$.
The recursion for $U_{\delta,N,k}^{\eps}$ is \eqref{eq:noisy_DPP}.  Hence \cref{lem:controlled_small_noise} gives
\begin{equation}
 \max_k\|U_{\delta,N,k}^{\eps}-V_\delta(\cdot,t_k)\|_\infty
 \le C_T\left(
 \sqrt{\nu_{\delta,N}^{\eps}}+\nu_{\delta,N}^{\eps}
 \right).
\label{eq:U_vs_Vdelta}
\end{equation}
Adding the deterministic time-discretization estimate
$\max_k\|V_\delta(\cdot,t_k)-V(\cdot,t_k)\|_\infty\le C\sqrt\delta$
from \cref{thm:DPP_convergence} proves \eqref{eq:main_unconditional_U}.

To estimate the stored continuation, write
\[
 Z_{\delta,N,k}^{\eps}(g)-V_\delta(g,t_k)
 =\int_G\bigl(U_{\delta,N,k}^{\eps}(gz)-V_\delta(gz,t_k)\bigr)
 q_{\delta,N}^{\eps}(z)\,\mathrm d\mu(z)
 +\mathcal R_k(g).
\]
The first term is bounded by \eqref{eq:U_vs_Vdelta}.  Since the deterministic discrete values have a Lipschitz constant bounded uniformly in $k$ and $\delta$ by \cref{prop:Tdelta_properties},
\[
 |\mathcal R_k(g)|
 \le C\int_G\dist_G(z,e)q_{\delta,N}^{\eps}(z)\,\mathrm d\mu(z)
 \le C\sqrt{\delta\nu_{\delta,N}^{\eps}}.
\]
This is dominated by $C\sqrt{\nu_{\delta,N}^{\eps}}$ for $\delta\le1$, and \eqref{eq:main_unconditional_Z} follows.
\end{proof}

\begin{corollary}[A canonical parameter choice]
\label{cor:canonical_scaling}
The artificial viscosity may be set equal to zero, because the finite-rank Markov filter itself supplies the regularization needed for stability.  With
\[
 \eps=0,
 \quad
 \delta=N^{-1},
\]
\cref{thm:finite_rank_lipschitz_rate} gives
\begin{equation}
 \max_k\left(
 \|U_{\delta,N,k}^{0}-V(\cdot,t_k)\|_\infty
 +\|Z_{\delta,N,k}^{0}-V(\cdot,t_k)\|_\infty
 \right)
 \le C_TN^{-1/2}.
\label{eq:canonical_rate}
\end{equation}
More generally, if $\eps\lesssim\delta$ and $\delta\asymp N^{-1}$, the same rate holds.
\end{corollary}

\begin{lemma}[Second-order consistency of the Markov filter]
\label{lem:markov_filter_C2}
For every $f\in C^2(G)$ and every $0\le\delta\eps\le1$,
\begin{equation}
 \|\mathcal Q_{\delta,N}^{\eps}f-f\|_{L^\infty(G)}
 \le C_G\bigl(\delta\eps+N^{-2}\bigr)\|f\|_{C^2(G)}.
\label{eq:Q_C2_consistency}
\end{equation}
\end{lemma}

\begin{proof}
The heat part satisfies
\[
 \|e^{\delta\eps\Lap_G}f-f\|_\infty
 \le \delta\eps\|\Lap_G f\|_\infty
 \le C_G\delta\eps\|f\|_{C^2}.
\]
It remains to estimate $\mathcal M_Nf-f$. Since $F_N$ is inversion invariant, the first-order term in the Taylor expansion cancels. More explicitly, on the ball where the logarithm is single-valued, write $z=\exp Y$. The symmetry $F_N(z)=F_N(z^{-1})$ and $\log(z^{-1})=-Y$ cancel the linear part of $f(gz)-f(g)$, leaving a remainder bounded by $C\|f\|_{C^2}|Y|^2$. The contribution of the complement is controlled by the rapid localization \eqref{eq:FN_localization}. Using the second moment estimate \eqref{eq:FN_moments} gives
\[
 \|\mathcal M_Nf-f\|_\infty\le C_GN^{-2}\|f\|_{C^2}.
\]
Combining the two estimates and using the contraction of both factors gives \eqref{eq:Q_C2_consistency}.
\end{proof}

\begin{corollary}[Smooth-regime refinement under a discrete $C^2$ bound]
\label{cor:smooth_markov_refinement}
Under Assumption~\ref{ass:standing}, suppose in addition that the deterministic discrete values satisfy
\begin{equation}
 \sup_{0\le k\le K}\|V_\delta(\cdot,t_k)\|_{C^2(G)}\le M_2
\label{eq:discrete_C2_bound}
\end{equation}
with $M_2$ independent of $\delta$, and that the deterministic semi-Lagrangian error improves to
\begin{equation}
 \max_{0\le k\le K}\|V_\delta(\cdot,t_k)-V(\cdot,t_k)\|_\infty\le C_T\delta.
\label{eq:smooth_DPP_rate_assumption}
\end{equation}
Then the Markov-filtered iterates of \eqref{eq:Markov_filtered_scheme} satisfy
\begin{align}
 \max_{0\le k\le K}\|U_{\delta,N,k}^{\eps}-V(\cdot,t_k)\|_\infty
 &\le C_T\left(\delta+\eps+\frac{1}{\delta N^2}\right),
\label{eq:smooth_markov_U}\\
 \max_{0\le k\le K}\|Z_{\delta,N,k}^{\eps}-V(\cdot,t_k)\|_\infty
 &\le C_T\left(\delta+\eps+\frac{1}{\delta N^2}\right),
\label{eq:smooth_markov_Z}
\end{align}
where $C_T$ may depend on $M_2$ but is independent of $\delta,N,\eps$.
\end{corollary}

\begin{proof}
Set $E_k:=\|U_{\delta,N,k}^{\eps}-V_\delta(\cdot,t_k)\|_\infty$. Using the recursion $U_k=\mathcal T_\delta\mathcal Q_{\delta,N}^{\eps}U_{k+1}$, the deterministic recursion $V_\delta(\cdot,t_k)=\mathcal T_\delta V_\delta(\cdot,t_{k+1})$, and the non-expansiveness of $\mathcal T_\delta$,
\[
 E_k\le E_{k+1}+
 \|\mathcal Q_{\delta,N}^{\eps}V_\delta(\cdot,t_{k+1})-V_\delta(\cdot,t_{k+1})\|_\infty.
\]
By \cref{lem:markov_filter_C2} and \eqref{eq:discrete_C2_bound}, the second term is bounded by $C(\delta\eps+N^{-2})$. Summing over $K=T/\delta$ steps gives
\[
 \max_kE_k\le C_T\left(\eps+\frac{1}{\delta N^2}\right).
\]
Adding \eqref{eq:smooth_DPP_rate_assumption} proves \eqref{eq:smooth_markov_U}. Finally,
\[
 Z_{\delta,N,k}^{\eps}-V_\delta(\cdot,t_k)
 =\mathcal Q_{\delta,N}^{\eps}(U_{\delta,N,k}^{\eps}-V_\delta(\cdot,t_k))
 +\bigl(\mathcal Q_{\delta,N}^{\eps}V_\delta(\cdot,t_k)-V_\delta(\cdot,t_k)\bigr),
\]
and the contraction of $\mathcal Q_{\delta,N}^{\eps}$ together with \cref{lem:markov_filter_C2} gives the stored-value estimate.  This proves \eqref{eq:smooth_markov_Z}.
\end{proof}

\begin{remark}[Coupling the viscosity and the time step to the resolution]
\label{prop:scaling}
The scheme has three parameters, and refinement is most easily organized along one-parameter families
\begin{equation}
\eps=N^{-\nu_{\eps}},\quad \delta=N^{-\nu_{\delta}},
\quad \nu_{\eps}>0,\ \nu_{\delta}\in(0,2),
\label{eq:admissible_scaling}
\end{equation}
for which the coupling conditions \eqref{eq:unconditional_coupling} hold automatically. The rate in \cref{thm:finite_rank_lipschitz_rate} is then optimized at $\nu_{\delta}=1$, the canonical choice of \cref{cor:canonical_scaling}. In the smooth regime the balance changes, because the time error improves to the first-order Lie-Trotter rate $O(\delta)$, the accumulated bias of the heat filter saturates at $O(KN^{-2})=O(\delta^{-1}N^{-2})$ (\cref{lem:filter_heat}), and the viscous bias improves to $O(\eps)$ (\cref{thm:VV}). Equating the three contributions suggests $\delta\asymp N^{-1}$ together with a viscosity tied to the resolution, $\eps\asymp N^{-1}$. 
\end{remark}

\begin{remark}
A deterministic estimate of the filter defect at every time step, as in the proof of \cref{thm:lowreg}, requires a quantitative Sobolev bound on the iterates, and any such bound on $V^{\eps}$ deteriorates as $\eps\downarrow0$.  The Markov construction instead interprets filtering as a centred small random perturbation.  Independent perturbations accumulate in quadratic mean, producing the scale
\[
 \sqrt{K}\,N^{-1}=\frac{\sqrt T}{N\sqrt\delta},
\]
rather than the deterministic sum $K/N$.  This is the mechanism behind \eqref{eq:main_unconditional_U}, and the reason why neither the regularity of $V$ and $V^{\eps}$ nor the smoothness of the Hamiltonian enters \cref{thm:finite_rank_lipschitz_rate}.
\end{remark}

\subsection{A fully discrete monotone realization}
\label{ssec:fully_discrete_markov}

We next show that the preceding structure can be retained after spatial cubature, control discretization, and numerical flow integration.  The method is then finite-dimensional, and its convergence theorem covers the implemented iteration rather than a semi-discrete core.

Let $\Lambda>\kappa$ be a fixed group-dependent constant large enough for the product spaces occurring below.  We first record that the spatial cubature needed by the method may be chosen positive without any interpolation matrix.

\begin{lemma}[Positive exact cubature]
\label{lem:positive_cubature}
For every $N\ge1$ there exist nodes and positive weights
\begin{equation}
 \mathcal G_N=\{(g_j,w_j)\}_{j=1}^{J_N},
 \quad
 w_j>0,
 \quad
 J_N\le C_GN^d,
\label{eq:positive_cubature}
\end{equation}
such that
\begin{equation}
 \int_Gf\,\mathrm d\mu
 =\sum_{j=1}^{J_N}w_jf(g_j)
 \quad\text{for every }f\in X_{\Lambda N}.
\label{eq:cubature_exactness}
\end{equation}
The weights may be normalized so that $\sum_jw_j=1$.
\end{lemma}

\begin{proof}
Let $E_N$ be the real vector space generated by the real and imaginary parts of functions in $X_{\Lambda N}$, and choose a basis $1,f_1,\ldots,f_D$ of $E_N$.  The vector
\[
 b:=\left(\int_G f_1\,\mathrm d\mu,\ldots,
          \int_G f_D\,\mathrm d\mu\right)
\]
belongs to the convex hull of the compact set
$\{(f_1(g),\ldots,f_D(g)):g\in G\}$.  Indeed, it is the barycenter of the push-forward of $\mu$ under this map.  Carath\'eodory's theorem therefore provides at most $D+1$ points $g_j$ and non-negative coefficients $w_j$ with $\sum_jw_j=1$ representing $b$.  After discarding zero weights, all weights are positive and the resulting formula is exact on $E_N$, hence on $X_{\Lambda N}$.  Finally, the Weyl counting law gives
$D+1\le 2\dim_{\mathbb C}X_{\Lambda N}+1\le C_GN^d$.
\end{proof}

\begin{remark}[An explicit positive cubature on $\SO{3}$]
\label{rem:SO3_positive_cubature}
For a target Wigner band $R$, take
$2R+1$ equispaced nodes in each of the Euler angles $\alpha$ and $\gamma$, and take $R+1$ Gauss-Legendre nodes in $x=\cos\beta$, with the corresponding positive product weights normalized for Haar probability measure.  This cubature is exact on $X_R$.  Indeed,
\[
 D^\ell_{mn}(\alpha,\beta,\gamma)
 =e^{-im\alpha}d^\ell_{mn}(\beta)e^{-in\gamma}.
\]
The two trapezoidal sums vanish unless $m=n=0$. In the remaining case
$d^\ell_{00}(\beta)=P_\ell(\cos\beta)$, which is integrated exactly by the Gauss-Legendre rule.  Thus the abstract positive cubature in \cref{lem:positive_cubature} has a direct tensor-product realization on $\SO{3}$, and the standard Wigner transform grid can be enlarged to the band $R=\Lambda N$ required by \eqref{eq:cubature_exactness}.
\end{remark}

For a vector $\mathbf a=(a_j)_{j=1}^{J_N}$, define the band-limited extension
\begin{equation}
 (\mathcal Q_{\delta,N,J}^{\eps}\mathbf a)(g)
 :=\sum_{j=1}^{J_N}
 w_jq_{\delta,N}^{\eps}(g^{-1}g_j)a_j.
\label{eq:discrete_Q}
\end{equation}

\begin{proposition}[Discrete Markov and moment properties]
\label{prop:discrete_markov_filter}
If $\Lambda$ in \eqref{eq:cubature_exactness} is chosen sufficiently large, then
\begin{enumerate}[label=\textnormal{(\roman*)},leftmargin=2em]
\item $\mathcal Q_{\delta,N,J}^{\eps}$ maps $\mathbb R^{J_N}$ into $X_{\kappa N}$, is positivity preserving, and
\begin{equation}
 \sum_{j=1}^{J_N}w_jq_{\delta,N}^{\eps}(g^{-1}g_j)=1,
 \quad g\in G.
\label{eq:discrete_row_sum}
\end{equation}
Consequently
\[
 \|\mathcal Q_{\delta,N,J}^{\eps}\mathbf a
 -\mathcal Q_{\delta,N,J}^{\eps}\mathbf b\|_{L^\infty(G)}
 \le\|\mathbf a-\mathbf b\|_{\ell^\infty}.
\]
\item The transition probabilities
\[
 P_j(g):=w_jq_{\delta,N}^{\eps}(g^{-1}g_j)
\]
satisfy, uniformly in $g$,
\begin{align}
 \sum_jP_j(g)\|\rho(g_j)-\rho(g)\|_{\mathrm F}^2
 &\le C_G(\delta\eps+N^{-2}),
\label{eq:discrete_variance}\\
 \left\|\sum_jP_j(g)\bigl(\rho(g_j)-\rho(g)\bigr)\right\|_{\mathrm F}
 &\le C_G(\delta\eps+N^{-2}).
\label{eq:discrete_drift}
\end{align}
\end{enumerate}
\end{proposition}

\begin{proof}
For fixed $g$, the function $h\mapsto q_{\delta,N}^{\eps}(g^{-1}h)$ belongs to $X_{\kappa N}$, because the heat multiplier does not enlarge the spectral support of $F_N$.  Hence \eqref{eq:discrete_row_sum} follows from the cubature exactness.  Positivity and the contraction estimate are immediate.

The matrix entries of the faithful representation $\rho$, and the products of two such entries, belong to fixed Peter-Weyl spectral spaces independent of $N$.  By \cref{lem:spectral_product}, the products of these functions with
$h\mapsto q_{\delta,N}^{\eps}(g^{-1}h)$ lie in $X_{\Lambda N}$ after choosing $\Lambda$ sufficiently large.  Therefore the cubature identities for the left-hand sides of \eqref{eq:discrete_variance}--\eqref{eq:discrete_drift} agree exactly with their Haar-integral counterparts.  The latter satisfy the asserted bounds by \eqref{eq:Q_second_moment} and the symmetrization argument in the proof of \cref{thm:finite_rank_lipschitz_rate}.
\end{proof}

To include the control and flow discretizations, assume in addition that there are constants $L_{\xi,U},L_{\ell,U}$ such that
\begin{equation}
 \|\xi(g,u)-\xi(g,v)\|_{\Lie}
 \le L_{\xi,U}|u-v|,
 \quad
 |\ell(g,u)-\ell(g,v)|
 \le L_{\ell,U}|u-v|.
\label{eq:control_Lipschitz}
\end{equation}
Let $U_h\subset U$ be a finite control set with fill distance
\begin{equation}
 h_U:=\sup_{u\in U}\inf_{v\in U_h}|u-v|.
\label{eq:control_fill}
\end{equation}
Let $\widehat\Phi_\delta^u:G\to G$ be a Lie-group integrator of order $p\ge1$ (see \cite{IserlesMuntheKaasNorsettZanna2000,HallLeok2017}) satisfying, uniformly in $g,h$ and $u$,
\begin{align}
 \dist_G(\widehat\Phi_\delta^u(g),\Phi_\delta^u(g))
 &\le C_\Phi\delta^{p+1},
\label{eq:flow_local_error}\\
 \|\rho(\widehat\Phi_\delta^u(g))-\rho(\widehat\Phi_\delta^u(h))\|_{\mathrm F}
 &\le(1+C_\Phi\delta)\,\|\rho(g)-\rho(h)\|_{\mathrm F}.
\label{eq:flow_discrete_lip}
\end{align}
The Lipschitz stability \eqref{eq:flow_discrete_lip} is stated in the embedded metric \eqref{eq:faithful_embedding_metric} rather than in $\dist_G$, because the coupling argument of \cref{lem:controlled_small_noise} is run in that metric, and a bound in $\dist_G$ would transfer to it only at the cost of the fixed factor $C_\rho/c_\rho\ge1$ at every step, which is not affordable over $K$ steps. The condition is nevertheless harmless. The exact flow satisfies it, by \eqref{eq:flow_embedding_lip}, and so does every integrator assembled from finitely many exponential updates $g\mapsto g\exp\bigl(\delta\,\eta(g,u)\bigr)$ with $\eta$ Lipschitz in $g$ uniformly in $u$, a class containing the Lie-Euler method and the commutator-free Runge-Kutta schemes \cite{IserlesMuntheKaasNorsettZanna2000}. Indeed $\rho\bigl(g\exp(\delta\eta(g,u))\bigr)=\rho(g)\,e^{\delta\,\mathrm d\rho(\eta(g,u))}$, the left factor is unitary, and $\|e^{\delta A}-e^{\delta B}\|\le\delta\|A-B\|\,e^{\delta\max(\|A\|,\|B\|)}$, so the difference of two such updates issued from $g$ and $h$ is bounded by $(1+C\delta)\|\rho(g)-\rho(h)\|_{\mathrm F}$, exactly as in the derivation of \eqref{eq:flow_embedding_lip}.

Define $\widehat U_K:=\phi$ and
$\widehat Z_K:=\mathcal Q_{\delta,N,J}^{\eps}(\phi(g_j))_{j=1}^{J_N}$.  For $k=K-1,\ldots,0$, set
\begin{align}
 \widehat U_k(g)
 &:=\min_{u\in U_h}\left\{
 \delta\ell(g,u)+
 \widehat Z_{k+1}(\widehat\Phi_\delta^u(g))
 \right\},
\label{eq:fully_discrete_U}\\
 \widehat Z_k
 &:=\mathcal Q_{\delta,N,J}^{\eps}
 \bigl(\widehat U_k(g_j)\bigr)_{j=1}^{J_N}.
\label{eq:fully_discrete_Z}
\end{align}
Every $\widehat Z_k$ belongs to $X_{\kappa N}$ and can therefore be stored either by its nodal values or by its Peter-Weyl coefficients.

\begin{lemma}[Control and flow perturbations]
\label{lem:control_flow_perturbation}
Let $\widehat V_{\delta,k}^{h,p}$ be the deterministic discrete value obtained from \eqref{eq:Bellman} by replacing $U$ with $U_h$ and $\Phi_\delta^u$ with $\widehat\Phi_\delta^u$.  Under \eqref{eq:control_Lipschitz} and
\eqref{eq:flow_local_error}--\eqref{eq:flow_discrete_lip},
\begin{equation}
 \max_{0\le k\le K}
 \|\widehat V_{\delta,k}^{h,p}-V_\delta(\cdot,t_k)\|_{L^\infty(G)}
 \le C_T(h_U+\delta^p).
\label{eq:control_flow_perturbation}
\end{equation}
\end{lemma}

\begin{proof}
First keep the control set $U_h$ fixed and compare the exact and numerical flows under the same piecewise-constant control sequence, working in the embedded metric of \eqref{eq:faithful_embedding_metric}, in which both flows are $(1+C\delta)$-Lipschitz by \eqref{eq:flow_embedding_lip} and \eqref{eq:flow_discrete_lip}, and in which the local error \eqref{eq:flow_local_error} reads $\|\rho(\widehat\Phi_\delta^u(g))-\rho(\Phi_\delta^u(g))\|_{\mathrm F}\le C_\rho C_\Phi\delta^{p+1}$.  This gives
\[
 e_{n+1}\le(1+C\delta)e_n+C\delta^{p+1},
 \quad e_0=0,
\]
for the embedded distance between the two trajectories.  Thus $\max_ne_n\le C_T\delta^p$, and by \eqref{eq:faithful_embedding_metric} the same bound holds for $\dist_G$.  The Lipschitz bounds on the costs imply an $O(\delta^p)$ difference between the corresponding discrete costs, uniformly over control sequences, and hence between the two value functions.

Next use the exact flow and compare controls in $U$ and $U_h$.  Given a control value $u\in U$, choose $v\in U_h$ with $|u-v|\le h_U+o(1)$.  Standard continuous dependence of the flow on a constant control, following from \eqref{eq:control_Lipschitz}, gives a one-step perturbation $C\delta h_U$.  Discrete Gronwall therefore yields an $O(h_U)$ trajectory perturbation over the full horizon, and the running and terminal costs change by $O(h_U)$.  Applying this construction to a near-optimal control sequence gives one value inequality, and the reverse inequality follows from $U_h\subset U$.  Combining the two comparisons proves \eqref{eq:control_flow_perturbation}.
\end{proof}

\begin{theorem}[Fully discrete convergence]
\label{thm:fully_discrete_convergence}
Under Assumption~\ref{ass:standing}, assume \eqref{eq:control_Lipschitz} and
\eqref{eq:flow_local_error}--\eqref{eq:flow_discrete_lip}.  Let
$\widehat U_k,\widehat Z_k$ be defined by
\eqref{eq:fully_discrete_U}--\eqref{eq:fully_discrete_Z} using a positive cubature satisfying \eqref{eq:cubature_exactness}.  There exists $C_T<\infty$, independent of $\delta$, $N$, $\eps$, $h_U$, and the cubature nodes, such that, whenever $\delta\le1$ and $\nu_{\delta,N}^{\eps}\le1$,
\begin{align}
 \max_{0\le k\le K}
 \|\widehat U_k-V(\cdot,t_k)\|_{L^\infty(G)}
 &\le C_T\left(
 \sqrt\delta+\sqrt{\nu_{\delta,N}^{\eps}}
 +\nu_{\delta,N}^{\eps}
 +h_U+\delta^p
 \right),
\label{eq:fully_discrete_rate_U}\\
 \max_{0\le k\le K}
 \|\widehat Z_k-V(\cdot,t_k)\|_{L^\infty(G)}
 &\le C_T\left(
 \sqrt\delta+\sqrt{\nu_{\delta,N}^{\eps}}
 +\nu_{\delta,N}^{\eps}
 +h_U+\delta^p
 \right).
\label{eq:fully_discrete_rate_Z}
\end{align}
Consequently, the fully discrete scheme converges uniformly provided
\begin{equation}
 \delta\to0,
 \quad
 \eps\to0,
 \quad
 \delta N^2\to\infty,
 \quad
 h_U\to0.
\label{eq:fully_discrete_conditions}
\end{equation}
The spatial cubature introduces no Lebesgue-constant amplification, since it enters through an exactly stochastic transition matrix.
\end{theorem}

\begin{proof}
By \cref{prop:discrete_markov_filter}, the cubature step is a Markov transition satisfying the drift and variance hypotheses of \cref{lem:controlled_small_noise} with
$a=\delta\eps+N^{-2}\le1$.  Apply the proof of that lemma with the admissible controls restricted to $U_h$ and with the deterministic maps $\widehat\Phi_\delta^u$, whose uniform Lipschitz property in the embedded metric, \eqref{eq:flow_discrete_lip}, replaces \eqref{eq:flow_embedding_lip} in the coupling recursion.  This compares $\widehat U_k$ with the deterministic value $\widehat V_{\delta,k}^{h,p}$ and gives
\[
 \max_k\|\widehat U_k-\widehat V_{\delta,k}^{h,p}\|_\infty
 \le C_T\left(\sqrt{\nu_{\delta,N}^{\eps}}+\nu_{\delta,N}^{\eps}\right).
\]
The deterministic perturbation estimate \eqref{eq:control_flow_perturbation} contributes $C_T(h_U+\delta^p)$, and \cref{thm:DPP_convergence} contributes $C_T\sqrt\delta$.  The triangle inequality proves \eqref{eq:fully_discrete_rate_U}.  Finally, apply the discrete Markov operator once to both $\widehat U_k$ and the uniformly Lipschitz deterministic value $V_\delta(\cdot,t_k)$.  The contraction property and \eqref{eq:discrete_variance} give an additional $O(\sqrt{\delta\nu_{\delta,N}^{\eps}})$ term, which is absorbed by the right-hand side, proving \eqref{eq:fully_discrete_rate_Z}.
\end{proof}

\begin{algorithm}[t]
\caption{Finite-rank Markov spectral Bellman iteration}
\label{alg:viscous_spectral}
\begin{algorithmic}[1]
\Require Horizon $T$, $K$ steps with $\delta=T/K$; bandwidth $N$; optional viscosity $\eps\ge0$; positive cubature $\{(g_j,w_j)\}_{j=1}^{J_N}$ exact on $X_{\Lambda N}$; control mesh $U_h$; Lie-group flow map $\widehat\Phi_\delta^u$.
\Ensure Band-limited continuations $\widehat Z_k\in X_{\kappa N}$ and prefiltered values $\widehat U_k$.
\State Construct the central Markov multiplier $m_N(\pi)$ from \eqref{eq:positive_filter_multiplier} and combine it with the heat multiplier $e^{-\delta\eps c_\pi}$.
\State Set $\widehat U_K(g_j)\leftarrow\phi(g_j)$ and apply the combined Markov filter to obtain $\widehat Z_K$.
\For{$k=K-1,K-2,\ldots,0$}
\State For every cubature node $g_j$, compute
\[
 \widehat U_k(g_j)\leftarrow
 \min_{u\in U_h}\left\{
 \delta\ell(g_j,u)+
 \widehat Z_{k+1}(\widehat\Phi_\delta^u(g_j))
 \right\},
\]
where $\widehat Z_{k+1}$ is evaluated off grid from its finite Peter-Weyl expansion.
\State Apply the finite-rank Markov multiplier to the nodal values
$\{\widehat U_k(g_j)\}$ to obtain $\widehat Z_k\in X_{\kappa N}$.
\EndFor
\State \Return $\{\widehat U_k,\widehat Z_k\}_{k=0}^K$.
\end{algorithmic}
\end{algorithm}

\begin{remark}[Implementation on $\SO{3}$]
For $\SO{3}$, the product formula
$\chi_{\ell_1}\chi_{\ell_2}=\sum_{\ell=|\ell_1-\ell_2|}^{\ell_1+\ell_2}\chi_\ell$
turns \eqref{eq:positive_filter_multiplier} into a finite triangular sum.  The transport of each control candidate still acts exactly on Wigner coefficients.  The nonlinear minimization is performed at the positive cubature nodes, and the Markov filter is then a diagonal block multiplier.  Thus the new construction changes neither the representation-theoretic character of the method nor its compatibility with fast Wigner transforms, while it removes the loss of monotonicity caused by an orthogonal forward transform of the Bellman output.
\end{remark}

\begin{remark}[Relation to the side-truncated heat filter]
The side-truncated heat operator $\Pi_Me^{\Lap_G/N^2}$ remains a useful alternative, and the estimate \eqref{eq:finite_heat_tail} quantifies its loss of contractivity.  The present filter is stronger for the convergence theory. It is finite-rank and exactly Markov for every $N$, so no oversampling parameter $M$ and no condition $K\eta_{N,M}\to0$ are required.  Its multipliers, on the other hand, are computed from tensor-product multiplicities rather than by the single closed formula $e^{-c_\pi/N^2}$.
\end{remark}

\section{Numerical studies}
\label{sec:numerics}

We test the fully discrete Fej\'er-Markov scheme of \cref{ssec:fully_discrete_markov} and the heat-filter variant of \cref{ssec:filter} on the rotation group $\SO{3}$. Example~1 validates the schemes against closed-form solutions. Example~2 is a nonsmooth multi-target problem without a known closed form, compared against an upper-bound reference and against grid-based semi-Lagrangian baselines. All experiments were run on a single Apple M3 Pro core.

\subsection{Setup on $\SO{3}$ and Wigner $D$-matrices}
\label{ssec:SO3_setup}

The unitary dual of $\SO{3}$ is the discrete set $\widehat{\SO{3}}=\{\,\pi_{\ell}\,:\,\ell\in\N\,\}$, with $d_{\ell}=2\ell+1$. The Casimir eigenvalue of $\pi_\ell$ is
\begin{equation}
	c_{\ell}\;=\;\ell(\ell+1),\quad\Lap_{\SO{3}}\,\pi_{\ell,mn}\;=\;-\ell(\ell+1)\,\pi_{\ell,mn},
	\label{eq:SO3_casimir}
\end{equation}
where $-\ell\le m,n\le\ell$. In ZYZ Euler angles $(\alpha,\beta,\gamma)\in[0,2\pi)\times[0,\pi]\times[0,2\pi)$ the matrix coefficients are the \emph{Wigner $D$-matrices}
\begin{equation}
	\pi_{\ell,mn}(\alpha,\beta,\gamma)\;=\;D^{\ell}_{mn}(\alpha,\beta,\gamma)\;=\;e^{-\mathrm{i} m\alpha}\,d^{\ell}_{mn}(\beta)\,e^{-\mathrm{i} n\gamma},
	\label{eq:Wigner_D}
\end{equation}
with $d^{\ell}_{mn}(\beta)$ the real Wigner (small) $d$-matrix \cite{BroeckerTomDieck1985}, and the normalized Haar measure is $\mathrm{d}\mu=(8\pi^{2})^{-1}\sin\beta\,\mathrm{d}\alpha\,\mathrm{d}\beta\,\mathrm{d}\gamma$. The truncation $\sqrt{c_{\pi}}\le N$ becomes $\ell\le L_{N}=\lfloor\frac{1}{2}(-1+\sqrt{1+4N^{2}})\rfloor$, and
\begin{equation}
	\dim X_{N}=\sum_{\ell=0}^{L_{N}}(2\ell+1)^{2}=\frac{(L_{N}+1)(2L_{N}+1)(2L_{N}+3)}{3}=O(L_{N}^{3}).
	\label{eq:dimXN_SO3}
\end{equation}
From here on we parameterize by the spin band $L$, so that coefficients are stored for $\ell\le L$.

\subsection{Implementation details}
\label{ssec:SO3_implementation}

\subsubsection{Transforms and positive cubature} The forward Peter-Weyl transform separates in the Euler angles,
\begin{equation}
	\widehat f^{\ell}_{mn}=\sqrt{2\ell+1}\int_{0}^{2\pi}\!\!\int_{0}^{\pi}\!\!\int_{0}^{2\pi}f(\alpha,\beta,\gamma)\,e^{\mathrm{i} m\alpha}\,d^{\ell}_{mn}(\beta)\,e^{\mathrm{i} n\gamma}\,\mathrm{d}\mu,
	\label{eq:SO3_forward}
\end{equation}
and is computed by FFTs in $\alpha,\gamma$ and a Gauss-Legendre rule in $\cos\beta$, with $2(L+1)$ equispaced nodes in each of $\alpha$ and $\gamma$ and $L+1$ Gauss-Legendre nodes in $\cos\beta$, hence $J=4(L+1)^{3}$ grid points. Fast $\SO{3}$ transforms reduce this cost (see, e.g., \cite{ShenXuZhang2018}), and at the bands used here the direct separated transform suffices. The weights of this tensor-product rule are positive, and the rule is the explicit positive cubature of \cref{rem:SO3_positive_cubature}. With $2(L+1)$ nodes in each angle and $L+1$ Gauss-Legendre nodes, it integrates $X_{R}$ exactly for every $R\le 2L+1$, which covers the kernels and the products required below.
The following proposition records that the implemented pipeline realizes the discrete Markov operator of \cref{ssec:fully_discrete_markov} without error.

\begin{proposition}
	\label{prop:exact_realization}
	Let $q\in X_{L}$ be a central, inversion-invariant kernel with spin multipliers $\widehat q(\ell)$, let $\{(g_{j},w_{j})\}_{j=1}^{J}$ be the cubature above, and let $\mathbf a=(a_{j})_{j=1}^{J}$ be any vector of nodal values. Then the composite operation consisting of the discrete forward transform, the multiplication of the spin-$\ell$ block by $\widehat q(\ell)$, and the evaluation of the resulting expansion at $g\in\SO{3}$ returns
	\[
	\sum_{j=1}^{J}w_{j}\,q(g^{-1}g_{j})\,a_{j},
	\]
	that is, the discrete Markov operator \eqref{eq:discrete_Q} with kernel $q$. In particular, for $\eps=0$ the implemented Fej\'er-Markov iteration, with multiplier \eqref{eq:SO3_fejer_multiplier}, is exactly the scheme \eqref{eq:fully_discrete_U}--\eqref{eq:fully_discrete_Z} covered by \cref{thm:fully_discrete_convergence}, and for $\eps>0$ it differs only by the terminal convention recorded below.
\end{proposition}

\begin{proof}
	The discrete forward transform of $\mathbf a$ is \eqref{eq:SO3_forward} with the integral replaced by the cubature sum, $\widehat{\mathbf a}^{\ell}_{mn}=\sqrt{2\ell+1}\sum_{j}w_{j}a_{j}\overline{D^{\ell}_{mn}(g_{j})}$. Multiplying by $\widehat q(\ell)$ and evaluating at $g$ gives
	\[
	\sum_{\ell\le L}\widehat q(\ell)\sum_{m,n}\widehat{\mathbf a}^{\ell}_{mn}\sqrt{2\ell+1}\,D^{\ell}_{mn}(g)
	=\sum_{j}w_{j}a_{j}\sum_{\ell\le L}\widehat q(\ell)\,(2\ell+1)\,\chi_{\ell}(g_{j}^{-1}g),
	\]
	by the unitarity identity $\sum_{m,n}\overline{D^{\ell}_{mn}(g_{j})}\,D^{\ell}_{mn}(g)=\chi_{\ell}(g_{j}^{-1}g)$. The inner sum equals $q(g_{j}^{-1}g)=q(g^{-1}g_{j})$, since $q$ is central, inversion invariant, and supported in the spins $\ell\le L$. No band limitation of the input and no cubature exactness enter this identity. Exactness enters through \cref{prop:discrete_markov_filter}, where it identifies the row sums and the moments of the discrete kernel with their Haar integrals, and this is what the range $R\le 2L+1$ above guarantees.
\end{proof}

In particular, neither an interpolation matrix nor any Lebesgue-constant amplification enters the loop.

For $\eps>0$ one detail differs. The implementation initializes the stored continuation with the filter alone, while the scheme of \cref{sec:markov} initializes with $\mathcal Q_{\delta,N}^{\eps}\phi$, so one viscous factor $e^{\delta\eps\Lap_{G}}$ is omitted at the terminal time. The difference changes the values by at most a constant multiple of $L_{\phi}\sqrt{\delta\eps}$, propagates without amplification through the non-expansive iteration, and vanishes for $\eps=0$, the setting of every Fej\'er-Markov run below.

\subsubsection{Exact coefficient-space transport} We discretize the control set by a finite family $\{u_{j}\}_{j=1}^{K_{u}}\subset U$, and keep each candidate control constant over a step. All dynamics used in the experiments are left-invariant, of the form $\dot g=g\,\widehat{(\omega_{0}+u)}$, so the one-step flow is the exact right translation $g\mapsto g\,h_{j}$, with $h_{j}=\exp_{\SO{3}}(\delta\,\widehat{(\omega_{0}+u_{j})})$ from Rodrigues' formula. On Peter-Weyl coefficients this shift is the exact block multiplication
\begin{equation}
	\widehat{V(\cdot\,h_{j})}^{\ell}=\widehat V^{\ell}\,D^{\ell}(h_{j})^{\!\top},
	\quad \ell=0,\dots,L,
	\label{eq:exact_transport}
\end{equation}
one per-$\ell$ product with the precomputed Wigner matrices $D^{\ell}(h_{j})$. The off-grid evaluation of $\widehat Z_{k+1}(\widehat\Phi^{u}_{\delta}(g_{j}))$ in \cref{alg:viscous_spectral} is thereby performed exactly, and the one-step flow itself is exact, so the flow-integration term $\delta^{p}$ of \cref{thm:fully_discrete_convergence} is absent. The only discretization errors of the Markov-filtered scheme are therefore the filter and time errors of the theorem and the control fill distance \eqref{eq:control_fill}. The heat-filter variant, whose multiplier is not finite-rank, retains the usual spectral aliasing of nonlinear terms \cite{KirbyKarniadakis2003}, and de-aliasing checks are reported with the experiments.

\subsubsection{Filters}
\label{ssec:impl_filters}
The Fej\'er-type Markov multiplier of \cref{rem:positive_filter_multiplier} is computed once per band from the Clebsch-Gordan rule. Writing $\theta_{\ell}:=\vartheta(\sqrt{c_{\ell}}/N)$, it reads
\begin{equation}
	m_{N}(\ell)=\frac{1}{Z_{N}(2\ell+1)}
	\sum_{\substack{\ell_{1},\ell_{2}\le \ell_{D}\\ |\ell_{1}-\ell_{2}|\le\ell\le\ell_{1}+\ell_{2}}}
	\theta_{\ell_{1}}\theta_{\ell_{2}}(2\ell_{1}+1)(2\ell_{2}+1),
	\label{eq:SO3_fejer_multiplier}
\end{equation}
which is a finite triangular sum, with the normalization $Z_{N}=\sum_{\ell\le\ell_{D}}\theta_{\ell}^{2}(2\ell+1)^{2}$. We take $\vartheta\in C_{c}^{\infty}$ equal to $1$ on $[0,1]$ with support in $[0,2]$, together with the kernel spin band $\ell_{D}=\lfloor L/2\rfloor$ and $N=\sqrt{c_{\ell_{D}}}/2$, so that the filter output lies in the stored band. Two measured constants of this filter are used below. The Jackson bound \eqref{eq:MN_Jackson} holds with $\|\mathcal M_{N}f-f\|_{\infty}\approx1.66\,\Lip(f)/N$, and the small-mode bias obeys the second-moment law $1-m_{N}(\ell)\approx0.65\,c_{\ell}/N^{2}$, the finite-rank analogue of the heat multiplier law $1-e^{-c_{\ell}/L_{N}^{2}}\approx c_{\ell}/L_{N}^{2}$. The heat filter and the exponential viscous step combine as before into the single diagonal multiplier $\exp(-c_{\ell}(\delta\eps+L_{N}^{-2}))$. In either case the terminal datum is filtered once, so a $K$-step run applies $K{+}1$ filters.

\subsubsection{Controls} Unless stated otherwise, we discretize the control ball $U=\{u\in\R^{3}:|u|\le u_{\max}\}$ by a symmetric stencil consisting of the zero control together with $n_{\mathrm{dir}}\in\{6,14\}$ unit directions, the signed coordinate axes and optionally the normalized cube diagonals, each scaled to $n_{\mathrm{mag}}\in\{1,2,3\}$ magnitudes in $(0,u_{\max}]$. This gives $K_{u}=1+n_{\mathrm{dir}}n_{\mathrm{mag}}$ candidates. Example~1 instead uses Fibonacci-distributed directions at full magnitude, $200$ of them (spherical covering radius $10.6^{\circ}$) in the canonical-scaling study and $400$ (covering radius $7.7^{\circ}$) in the refinement study of \cref{app:experiments}.

\subsection{Example 1. Validation against closed-form solutions}
\label{ssec:ex1}

\subsubsection{The convergence theorem at the canonical scaling}
\label{sssec:ex1_canonical}
Consider the problem of steering the attitude directly to the identity,
\begin{equation}
	\dot g=g\,\widehat u,\quad |u|\le 1,\quad \ell\equiv 0,\quad \phi(g)=\theta(g)=\dist_{\SO{3}}(g,I),\quad T=1.5.
	\label{eq:ex1_eikonal}
\end{equation}
The value function of \eqref{eq:ex1_eikonal} is known in closed form, since greedy geodesic descent is optimal, and
\begin{equation}
	V(g,t)=\max\bigl(\theta(g)-(T-t),\,0\bigr).
	\label{eq:ex1_eikonal_V}
\end{equation}
The terminal cost is $1$-Lipschitz but not $C^{1}$, and $V$ has a conical singularity at the identity, a kink on the sphere $\{\theta=T-t\}$, and the cut locus at $\theta=\pi$. The problem therefore lies in the minimal-regularity regime of \cref{thm:finite_rank_lipschitz_rate}. We run the Markov-filtered scheme at the canonical scaling of \cref{cor:canonical_scaling}, without artificial viscosity, taking $K=\lfloor T N\rceil$ steps of size $\delta=T/K$, so that $\delta\asymp N^{-1}$. \Cref{tab:ex1_canonical} reports the sup error over the full grid at $t=0$, for the prefiltered value $U_{0}$ and for the stored continuation $Z_{0}$.

\begin{table}[t]
	\centering
	\caption{Example 1. Sup error over the full grid at $t=0$ against the exact solution \eqref{eq:ex1_eikonal_V}, for the canonical scaling $\eps=0$, $\delta=T/\lfloor TN\rceil$ of \cref{cor:canonical_scaling}. The fitted order in $N$ is $0.55$, while the theorem predicts $1/2$.}
	\label{tab:ex1_canonical}
	\begin{tabular}{rrrrcc}
		\toprule
		$L$ & $\ell_{D}$ & $N$ & $K$ & $\norm{U_{0}-V(\cdot,0)}_{\infty}$ & $\norm{Z_{0}-V(\cdot,0)}_{\infty}$\\
		\midrule
		8 &  4 & 2.24 &  3 & 0.909 & 0.970\\
		12 &  6 & 3.24 &  5 & 0.833 & 0.873\\
		16 &  8 & 4.24 &  6 & 0.682 & 0.711\\
		20 & 10 & 5.24 &  8 & 0.616 & 0.636\\
		24 & 12 & 6.24 &  9 & 0.529 & 0.545\\
		28 & 14 & 7.25 & 11 & 0.492 & 0.504\\
		\bottomrule
	\end{tabular}
\end{table}

The errors decay with fitted order $0.55$ in $N$ (left panel of \cref{fig:ex1_canonical}), and the observed pre-asymptotic slopes are consistent with the predicted exponent $1/2$ of \eqref{eq:canonical_rate}. The computed iterates satisfy the predicted bounds. Since $\ell\equiv0$ and $\phi\ge0$, the exact iterates satisfy $0\le V_{k}\le\pi$, and every computed $U_{k}$ respects these bounds at every node and at every resolution although $\eps=0$. The Markov filter alone supplies the stability. On this problem optimal controls have full magnitude, and steering along the nearest available direction reduces the useful angular speed by the factor $1-\cos r_{\mathrm{cov}}$. The quantity $u_{\max}T(1-\cos r_{\mathrm{cov}})\approx0.026$ is therefore a directional indicator of the control-discretization effect, heuristic rather than a proven bound, and the stencil cannot stop within a step, which may add an effect of order $\delta$ near the target. The directional term is well below the measured errors. The absolute errors are large, roughly $30\%$ of the value range at the finest band. They come from the accumulated bias of the $K{+}1$ filter applications, of order $\sqrt{K+1}/N$ by the quadratic-mean mechanism of \cref{ssec:markov_convergence}, while for the continuous control ball the time discretization contributes nothing on this problem, greedy geodesic descent being optimal over a step of any length. Under the admissible couplings of \eqref{eq:admissible_scaling} the sup error decreases to $5.7\cdot10^{-2}$ at band $L=64$, and this refinement study is reported in \cref{app:experiments}.

\subsubsection{Filter saturation on a commuting problem}
\label{sssec:ex1_saturation}
We next consider the uncontrolled dynamics $\dot g=g\,\widehat\omega_{0}$ with $\omega_{0}=(0.3,-0.5,0.9)$, $\ell\equiv 0$ and $T=1$, together with the analytic class-function terminal cost $\phi(g)=e^{\cos\theta(g)}=\sum_{\ell}a_{\ell}\chi_{\ell}$. Transport and viscous flow act exactly on this expansion, so
\begin{equation}
	V^{\eps}(g,t)\;=\;\sum_{\ell\ge 0}a_{\ell}\,e^{-\eps(T-t)c_{\ell}}\,\chi_{\ell}\bigl(g\,e^{(T-t)\widehat\omega_{0}}\bigr)
	\label{eq:ex1_exact}
\end{equation}
is an exact solution, computable to machine precision as a one-dimensional sum, so the only scheme error is the accumulated filter bias. We use $K=10$ steps and $\eps=0$. With the filter off the scheme is exact to round-off, below $10^{-6}$, so transport and transforms add nothing on this commuting problem. The heat filter follows the saturated order $\beta_{h}=2$ of \cref{lem:filter_heat}, with local orders rising from $1.43$ to $1.91$ over $L=8,\dots,28$. The Fej\'er-Markov filter obeys the same second-moment law $1-m_{N}(\ell)\approx 0.65\,c_{\ell}/N^{2}$, with orders rising from $1.13$ to $1.44$ over $L=20,\dots,36$. Its error constant at equal stored band is larger by a factor of about ten, because its kernel scale is $N\approx L/4$ rather than $L$ (right panel of \cref{fig:ex1_canonical}). This factor measures the accuracy given up for exact finite-rank Markovianity, in line with the closing remark of \cref{ssec:fully_discrete_markov}. The viscous bias and the rate in $\delta$, measured on smooth variants of this setting, are reported in \cref{app:experiments}.

\begin{figure}[t]
	\centering
	\includegraphics[width=0.98\linewidth]{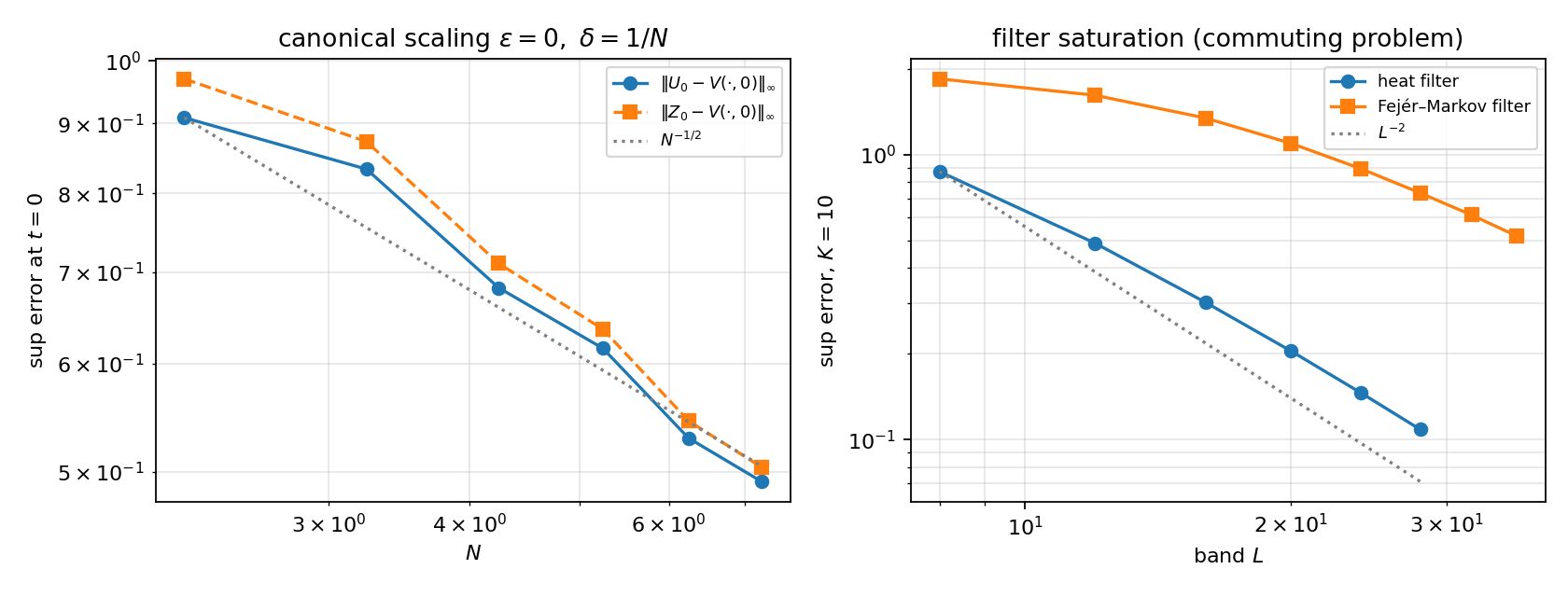}
	\caption{Example 1. The left panel shows the canonical-scaling errors of \cref{tab:ex1_canonical} against the exact nonsmooth solution \eqref{eq:ex1_eikonal_V}, together with the predicted slope $N^{-1/2}$ of \cref{cor:canonical_scaling}. The right panel shows the filter saturation on the commuting problem \eqref{eq:ex1_exact}. The heat filter attains the saturated order $2$ of \cref{lem:filter_heat}, and the Fej\'er-Markov filter follows the same $c_{\ell}/N^{2}$ law with kernel scale $N\approx L/4$.}
	\label{fig:ex1_canonical}
\end{figure}

\subsection{Example 2. Multi-target steering under an uncontrollable drift}
\label{ssec:ex2}

The second example is nonsmooth and carries an uncontrollable drift. The dynamics are
\begin{equation}
	\dot g=g\,(\widehat\omega_{0}+\widehat u),\quad \omega_{0}=(0,0,0.8),\quad |u|\le u_{\max}=1.5,
	\label{eq:ex2_dyn}
\end{equation}
with running cost $\ell(u)=1+\frac{0.4}{2}|u|^{2}$, horizon $T=1.5$ and $K_{u}=29$ control candidates, and the terminal cost is the Lipschitz multi-target function
\begin{equation}
	\begin{gathered}
		\phi(g)=\min_{j=1,2,3}\,\dist_{\SO{3}}(g,Q_{j}),\\
		Q_{1}=I,\quad
		Q_{2}=e^{0.9\widehat e_{1}}e^{0.5\widehat e_{2}},\quad
		Q_{3}=e^{-0.7\widehat e_{2}}e^{1.1\widehat e_{3}},
	\end{gathered}
	\label{eq:ex2_phi}
\end{equation}
whose gradient jumps across the equidistance sets and which has a conical singularity at each target. The value function then has gradient discontinuities, so \cref{thm:sup_norm} is unavailable and the filtered schemes of \cref{sec:stab,sec:markov} are needed. No closed-form solution is known. The example documents the instability predicted by \cref{lem:norm_mismatch}, measures the accuracy of the filtered schemes against a reference, tests the bias mechanism of \cref{ssec:markov_convergence}, and compares with classical grid-based semi-Lagrangian baselines.

All schemes below use the same control set, the same grids and the same exact one-step flows. They therefore approximate the same control-discretized problem, and the control-fill error of \eqref{eq:control_fill} is common to every entry. Errors are measured in the sup norm over a fixed set $E$ of $983$ points, consisting of the one-parameter subgroup $s\mapsto e^{s\widehat e_{3}}$ through $Q_{1}$, which lies entirely on the singular set of the Euler-angle chart, the geodesics through $Q_{1},Q_{2}$ and through $Q_{2},Q_{3}$, each sampled at $161$ points, and $500$ Haar-distributed points drawn once with a fixed seed. Grid-based schemes are evaluated on $E$ by their own interpolation rule and spectral schemes by exact synthesis.

\subsubsection{An upper-bound reference by trajectory optimization}
\label{sssec:ex2_reference}

No filtered run can serve as reference here. A $K$-step run applies the filter $K{+}1$ times, and each application removes a fraction $\approx c_{\ell}\,\nu_{1}$ of every mode, so a run accumulates the exposure
\begin{equation}
	\nu:=\;(K+1)\nu_{1}+\eps T,
	\quad
	\nu_{1}=
	\begin{cases}
		L^{-2} & \text{heat filter at scale } L,\\[1mm]
		0.65\,N^{-2} & \text{Fej\'er-Markov filter at scale } N,
	\end{cases}
	\label{eq:ex2_exposure}
\end{equation}
where $\nu_{1}$ is the second moment of one filter application (\cref{ssec:impl_filters}) and $\eps T$ accounts for the viscous multiplier. Runs of similar exposure smooth the value function by the same amount and agree with each other while sharing a common bias. Two heat runs at $(L,K,\eps)=(18,96,0.6/18)$ and $(20,120,0.6/20)$, whose exposures are $\sqrt\nu=0.591$ and $0.589$, agree to $2.9\cdot10^{-3}$ in the sup norm, yet both lie between $+0.31$ and $+0.76$ above the reference constructed below, with median $+0.65\approx1.1\sqrt{\nu}$, on a value function whose range on $E$ is about $[1.5,3.1]$.

The reference we adopt is instead obtained by open-loop trajectory optimization, independently of every scheme under test. For each $x\in E$ we minimize the exact rollout cost over controls that are piecewise constant on $M$ equal subintervals of $[0,T]$ with values in the $29$-point control set,
\begin{equation}
	J_{M}(x)=\min_{u_{0},\dots,u_{M-1}\in U_{29}}
	\Bigl\{\tfrac{T}{M}\sum_{k=0}^{M-1}\ell(u_{k})
	+\phi\bigl(x\,h_{u_{0}}\cdots h_{u_{M-1}}\bigr)\Bigr\},
	\ h_{u}=e^{(T/M)(\widehat\omega_{0}+\widehat u)}.
	\label{eq:ex2_cert}
\end{equation}
The multistart Gauss-Seidel search described next returns a feasible value $\widetilde J_{M}(x)\ge J_{M}(x)$, and every feasible value is a rigorous upper bound for the value at $x$. Since changing one control exchanges exactly one factor of the product $x\,h_{u_{0}}\cdots h_{u_{M-1}}$, each coordinate minimization in a sweep, performed exhaustively over the candidates, is exact, and the sweeps are iterated to convergence from six independent seeds per point. Refining $M=64\to128\to256$ lowers $\widetilde J_{M}$ by $3.7\cdot10^{-2}$ and then $1.6\cdot10^{-2}$, consistent with first-order convergence in $1/M$. The reference $\widetilde J:=\widetilde J_{256}$ is therefore a rigorous upper bound of the control-discretized value $J_{256}$, with a one-sided slack that we estimate, by extrapolation of the refinement differences, at about $2\cdot10^{-2}$. The estimate is empirical, and no lower bound is claimed. Every spectral and grid-based run of this example approaches $\widetilde J$ from above and none undershoots it by more than $6\cdot10^{-3}$, which is a consistency check, since a loose upper bound would be undershot by a converging scheme. Recomputing the $L=16$ heat run with the band-$2L$ quadrature of \cref{prop:discrete_markov_filter} changes it by $3.7\cdot10^{-3}$, so sampling aliasing is negligible at this accuracy.

\subsubsection{The norm-mismatch obstruction}
\label{sssec:ex2_instability}
Since $\ell\ge 1$ and $\phi\ge0$, the exact discrete value satisfies $V_{k}\ge(K-k)\delta$ in the global indexing, that is, $r\delta$ at $r:=K-k$ steps before the horizon. The bound is scheme-independent and requires no reference. Writing $U^{(r)}$ for a computed iterate $r$ steps before the horizon, the left panel of \cref{fig:ex2_instability} tracks $\min_{g}U^{(r)}-r\delta$ at $K=48$ for the plain Galerkin iteration \eqref{eq:VN_scheme} with $\eps=0$ at the bands $L=8,12,16$, and for the Markov-filtered scheme with $\eps=0$ at $L=16$. At every band the plain iterate breaks the bound within the first six steps and then loses a factor of about $1.1$ per step, reaching values between $-14$ and $-6$ at $r=48$, and at $L=12$ its energy fraction in the shells $\ell\ge8$ grows to $1.6\cdot10^{-2}$. This behavior is consistent with the amplification mechanism of \cref{lem:norm_mismatch}. The observed growth stays far below the worst-case bound $\Lambda_{N}\asymp N$, but the instability is present at every resolution and destroys the iterates well before the horizon. The Markov-filtered iterate respects the bound at every band with margin at least $+0.38$ and keeps its high-shell energy near $10^{-29}$ at $L=12$, although $\eps=0$. The finite-rank filter truncates and contracts the very shells that the projection amplifies, in agreement with the exact-cubature identity of \cref{prop:discrete_markov_filter}.

\begin{figure}[t]
	\centering
	\includegraphics[width=0.98\linewidth]{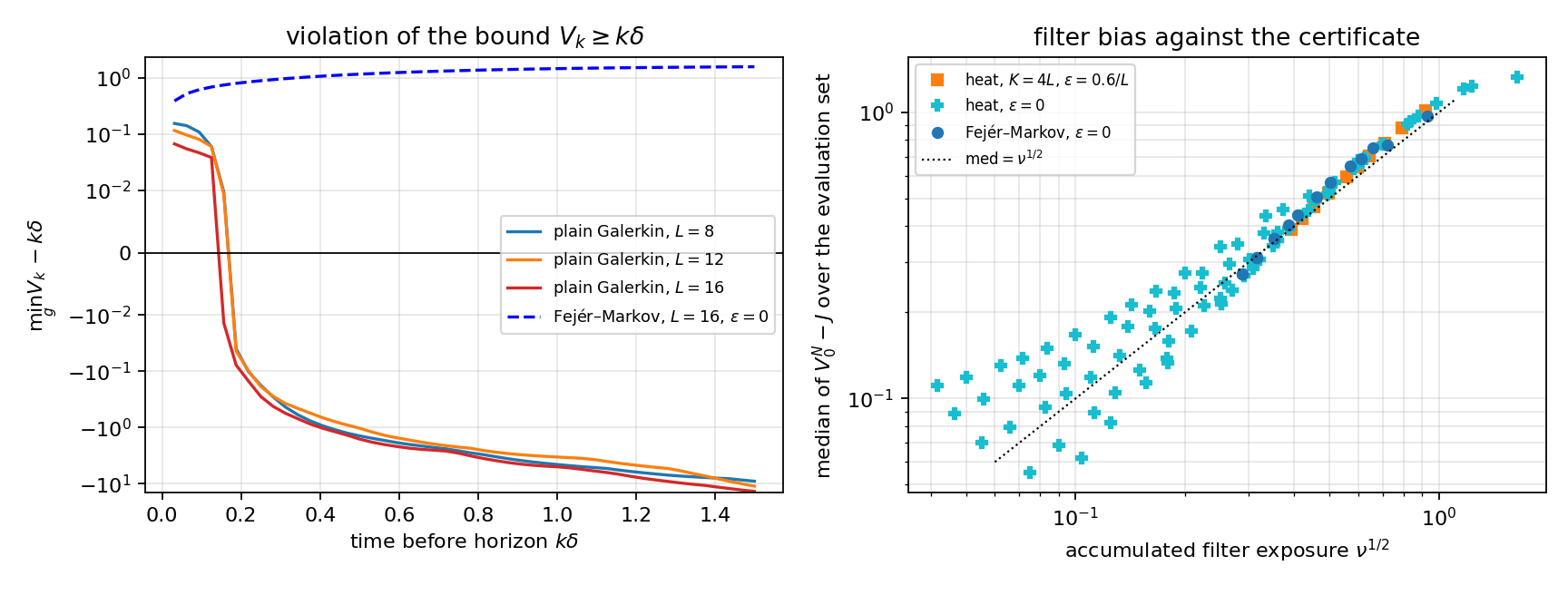}
	\caption{Example 2. The left panel shows the violation of the elementary bound $U^{(r)}\ge r\delta$ for the plain Galerkin iteration at $L=8,12,16$ and for the Fej\'er-Markov iteration at $L=16$, all with $\eps=0$ and $K=48$. The right panel shows the median over the evaluation set of the deviation $V^{N}_{0}-\widetilde J$ from the reference $\widetilde J$ for all spectral runs of this example, plotted against the accumulated filter exposure \eqref{eq:ex2_exposure}. Both filter families collapse onto the quadratic-mean law $\mathrm{med}\approx\nu^{1/2}$.}
	\label{fig:ex2_instability}
\end{figure}

\subsubsection{The size of the accumulated filter bias}
\label{sssec:ex2_bias}
The reference makes the systematic deviation of the filtered schemes directly measurable. For every spectral run of this example, both filter families and all couplings over $6\le L\le48$, the deviation $V^{N}_{0}-\widetilde J$ is positive on essentially all of $E$ and its median obeys
\[
\operatorname{med}_{E}\,\bigl(V^{N}_{0}-\widetilde J\bigr)\approx c\,\sqrt{\nu}
\quad\text{with}\quad c\in[0.75,\,1.15],
\]
where $\nu$ is the accumulated exposure \eqref{eq:ex2_exposure}. The right panel of \cref{fig:ex2_instability} shows the collapse. This is the quadratic-mean mechanism behind \cref{thm:finite_rank_lipschitz_rate}, whose error term $\sqrt{\eps+(\delta N^{2})^{-1}}$ equals $\sqrt{\nu/T}$ up to the second-moment constant of the filter. On these nonsmooth problems the measured deviation is also consistent with the predicted magnitude, with the caveat that a median against an upper-bound reference gauges the bias mechanism rather than the worst-case constant of the theorem. At a fixed band, couplings with fewer steps are therefore more accurate until the time error takes over, visible as the departures from the law at the smallest exposures.

\subsubsection{Comparison with grid-based semi-Lagrangian schemes}
\label{sssec:ex2_table}
As classical baselines we use the two grid-based semi-Lagrangian schemes, on the same grid, with the same exact flows and the same control set. The first evaluates the shifted value at the nearest grid node. The second interpolates trilinearly in the Euler-angle coordinates $(\alpha,\cos\beta,\gamma)$, periodically in $\alpha$ and $\gamma$ and clamped at the $\beta$-chart boundary. Both are monotone, like the Fej\'er-Markov scheme, so the comparison is between structure-preserving methods. Higher-order interpolation improves the constants by giving up monotonicity and accepting a stronger dependence on the chart \cite{FalconeFerretti2014}. The accuracy of a semi-Lagrangian scheme depends on the balance between time and interpolation errors, so each baseline is run for $K\in\{3,4,6,12,24,48,96\}$, with $K\le24$ for $L\ge24$, and \cref{tab:ex2} reports its best error, an oracle tuning against the reference. The Fej\'er-Markov scheme runs at the coupling $\delta\asymp N^{-1}$, $\eps=0$ of \cref{cor:canonical_scaling}, fixed in advance. The theory does not tie the heat variant with $\eps=0$ to a single coupling, so it is granted the same oracle over $K$ as the baselines, and the remaining schemes are thus tuned symmetrically.

\begin{table}[t]
	\centering
	\caption{Example 2. Sup deviation over the evaluation set from the upper-bound reference $\widetilde J$, whose one-sided slack is estimated at about $2\cdot10^{-2}$. The Fej\'er-Markov scheme runs at the coupling $\delta\asymp N^{-1}$, $\eps=0$ fixed in advance, and the resulting $K$ is listed. The heat filter at $\eps=0$ and the two grid-based baselines report their best error over the step counts $K\in\{3,4,6,12,24,48,96\}$, with $K\le24$ for $L\ge24$, chosen a posteriori against the reference, and the minimizing $K$ is listed. All schemes use the same grid, exact flows and control set. A subset of the computed bands is tabulated, and all of them appear in \cref{fig:ex2_comparison}.}
	\label{tab:ex2}
	\small
	\setlength{\tabcolsep}{4.2pt}
	\begin{tabular}{rrrrrrrrrr}
		\toprule
		& & \multicolumn{2}{c}{Fej\'er-Markov} & \multicolumn{2}{c}{heat, $\eps=0$} &
		\multicolumn{2}{c}{nearest SL} & \multicolumn{2}{c}{trilinear SL}\\
		\cmidrule(lr){3-4}\cmidrule(lr){5-6}\cmidrule(lr){7-8}\cmidrule(lr){9-10}
		$L$ & grid points & $K$ & err & err & $K$ & err & $K$ & err & $K$\\
		\midrule
		6 &   1\,372 &  3 & 1.150 & 0.527 & 3 & 0.510 & 3 & 0.428 & 3\\
		8 &   2\,916 &  3 & 0.939 & 0.405 & 3 & 0.480 & 3 & 0.328 & 6\\
		12 &   8\,788 &  5 & 0.842 & 0.299 & 4 & 0.330 & 3 & 0.239 & 6\\
		16 &  19\,652 &  6 & 0.702 & 0.233 & 4 & 0.267 & 4 & 0.191 & 6\\
		20 &  37\,044 &  8 & 0.628 & 0.198 & 6 & 0.240 & 4 & 0.174 & 6\\
		24 &  62\,500 &  9 & 0.543 & 0.169 & 6 & 0.200 & 4 & 0.147 & 12\\
		32 & 143\,748 & 12 & 0.450 & 0.145 & 6 & 0.172 & 6 & 0.111 & 12\\
		48 & 470\,596 & 18 & 0.349 & 0.102 & 12 & 0.125 & 6 & 0.090 & 12\\
		\bottomrule
	\end{tabular}
\end{table}

\begin{figure}[t]
	\centering
	\includegraphics[width=0.98\linewidth]{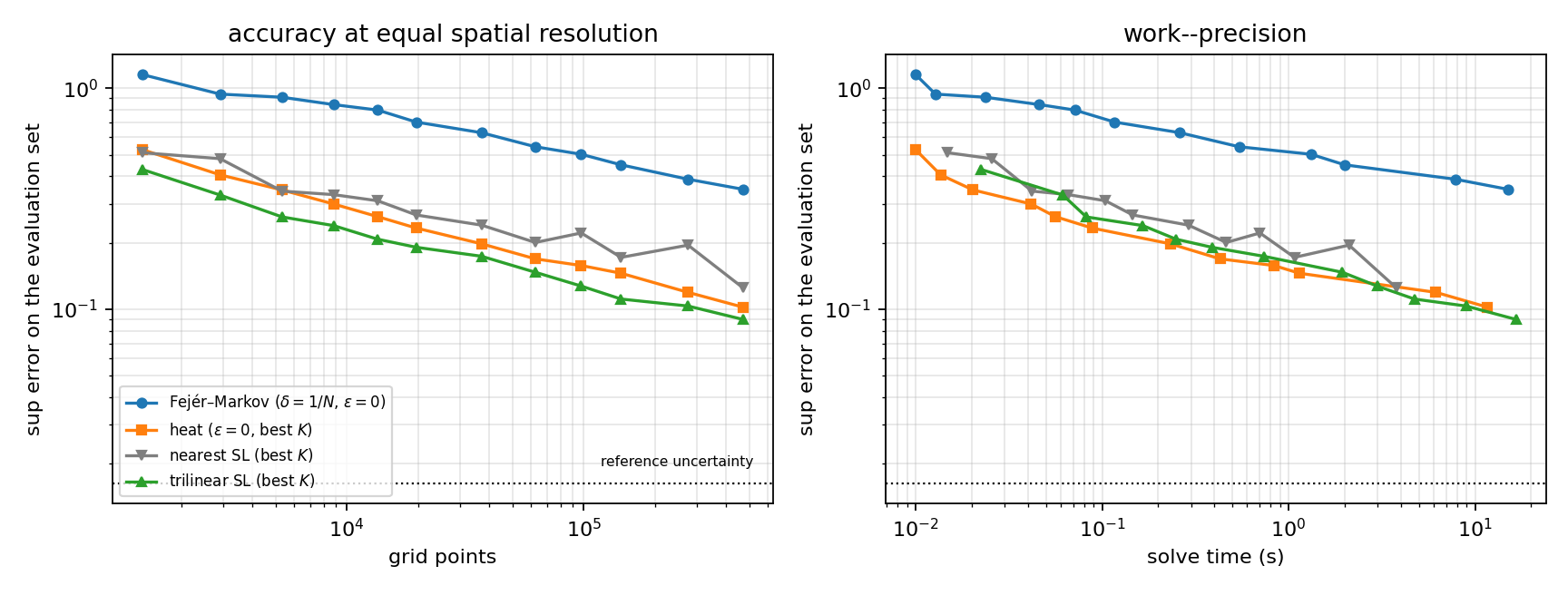}
	\caption{Example 2. The left panel shows the sup deviation on the evaluation set from the reference $\widetilde J$ at equal spatial resolution, as in \cref{tab:ex2}. The right panel shows the same errors against the solve time on a single CPU core. The dotted line marks the estimated one-sided slack of the reference.}
	\label{fig:ex2_comparison}
\end{figure}

The deviations from the reference decrease for every method as the grid is refined.
The nearest-node baseline converges erratically, because its error is sensitive to the alignment of the shifted points with the grid, and from $L=12$ on it is dominated by the tuned heat variant. The tuned heat variant and the trilinear baseline stay within about $30\%$ of each other at every resolution, with fitted orders $0.77$ and $0.73$ in the band, and they are indistinguishable in the work-precision plane (right panel of \cref{fig:ex2_comparison}).
At the finest grid they reach $0.102$ and $0.090$, about five times the estimated slack of the reference itself, at comparable solve times.
The Fej\'er-Markov scheme converges more slowly. Its fitted order against the reference is $0.61$ in $N$ over $L=6,\dots,48$, in agreement with the $N^{-1/2}$ rate of \cref{cor:canonical_scaling}, on a reference produced independently of the scheme family.
Its error is dominated by the accumulated bias of \cref{sssec:ex2_bias}.
One practical difference appears already at this level. The Fej\'er-Markov scheme requires no tuning, its parameters being tied to the single resolution $N$, whereas the other schemes are tuned over $K$ a posteriori. At the fixed step $\delta=T/48$, the trilinear errors roughly double and the nearest-node errors grow up to ninefold.

\subsubsection{Frame-independence and closed-loop quality}
\label{sssec:ex2_structure}

The accuracy table does not capture the structural properties for which the representation-theoretic formulation was built, and we quantify two of them. The first rests on an exact identity, the second on comparisons through $\widetilde J$.

The first is frame-independence. The problem \eqref{eq:ex2_dyn}--\eqref{eq:ex2_phi} is covariant under left translation, so replacing the targets $Q_{j}$ by $FQ_{j}$ for a fixed rotation $F$ replaces the value function $V$ by $V(F^{-1}\cdot)$, exactly, both for the continuous problem and for the semi-discrete iteration with exact flows. We solve the rotated problem for three Haar-random frames $F$ and measure the defect $\sup_{x\in E}\lvert V'(Fx)-V(x)\rvert$ at $L=32$, with every scheme at its configuration of \cref{tab:ex2} and with the heat variant also at $K=2L$.

For the spectral schemes the only non-equivariant ingredient is the sampling of the pointwise minimum, and the measured defects, $1.4\cdot10^{-3}$ for the Fej\'er-Markov filter and $8.7\cdot10^{-3}$ and $3.3\cdot10^{-3}$ for the two heat couplings, lie at the level of the aliasing effect and amount to at most $6\%$ of the respective errors. The trilinear defect is $0.044$, which is $40\%$ of its sup error, and the nearest-node defect is $0.23$, which exceeds its sup error. A significant share of the tuned baselines' accuracy is thus tied to the coordinate frame in which they are run, whereas the spectral schemes are substantially less sensitive to the chosen frame in these tests.

The second is the quality of the induced feedback law, which is the relevant metric when the value function is used as a controller. For each scheme we extract the greedy one-step-lookahead policy from its own value levels on its own time grid and roll out the closed loop from $164$ initial states, evaluating spectral lookaheads by exact coefficient transport and synthesis and grid lookaheads by the scheme's own interpolation. The realized cost is the cost of an admissible control and is evaluated exactly, so its excess over $\widetilde J$ compares two admissible strategies for the same control-discretized problem. \Cref{tab:ex2_struct} reports the results. Every policy performs better than its value error suggests. The heat variant at $K=2L$ has value error $0.30$ and median excess $0.028$, close to the trilinear baseline at $0.030$ whose value error is $2.7$ times smaller. About $9\%$ of its value error reaches the closed loop, consistent with a bias that is largely a state-independent shift and hence invisible to the $\operatorname{argmin}$, as the mechanism of \cref{sssec:ex2_bias} predicts.

The step count that minimizes the value error is not the one that minimizes the policy excess. The heat variant tuned to $K=6$ in \cref{tab:ex2} has policy excess $0.044$, which improves to $0.028$ at $K=2L=64$ while the value error doubles, so, in this experiment, the closed loop benefits more from frequent decisions than from value accuracy. The Fej\'er-Markov scheme, whose value column is the weakest, reaches $0.053$, improving from $0.142$ at $L=14$ as the band refines. The most negative excesses, of about $-0.03$, improve the upper bound at a handful of points and are consistent with its estimated slack.

\begin{table}[t]
	\centering
	\caption{Example 2. Structural comparison at $L=32$. The first four rows are the schemes at their configuration of \cref{tab:ex2}, with the listed step count $K$, and the last row is the heat variant at the longer-horizon coupling $K=2L$. The columns list the sup deviation from the reference $\widetilde J$, the equivariance defect maximized over three Haar-random frames, and the median and $90$th percentile of the closed-loop excess cost over the reference $\widetilde J$ from $164$ initial states.}
	\label{tab:ex2_struct}
	\small
	\begin{tabular}{lrrrrr}
		\toprule
		scheme & $K$ & sup err & defect & policy med & policy p90\\
		\midrule
		Fej\'er-Markov & 12 & 0.450 & $1.4\cdot10^{-3}$ & 0.053 & 0.107\\
		heat, $\eps=0$  &  6 & 0.145 & $8.7\cdot10^{-3}$ & 0.044 & 0.077\\
		nearest SL      &  6 & 0.172 & $2.3\cdot10^{-1}$ & 0.094 & 0.172\\
		trilinear SL    & 12 & 0.111 & $4.4\cdot10^{-2}$ & 0.030 & 0.051\\
		heat, $\eps=0$  & 64 & 0.297 & $3.3\cdot10^{-3}$ & 0.028 & 0.046\\
		\bottomrule
	\end{tabular}
\end{table}

\subsection{Discussion}
\label{ssec:num_discussion}

The experiments are consistent with the theoretical scaling, and the deviation of every filtered run from the reference follows the exposure law \eqref{eq:ex2_exposure} with a constant between $0.75$ and $1.15$, so the practical coupling question is that of budgeting $\nu$ against the time step. Example~2 also carries a methodological point. Filtered schemes of a common family agree with each other far more closely than with the value function, so reference solutions for nonsmooth HJB problems should be validated by scheme-independent means, such as the trajectory-based upper bound $\widetilde J$ obtained from \eqref{eq:ex2_cert}, rather than by internal refinement alone.

\bibliographystyle{plain}
\bibliography{ref}

\appendix
\section{Proof of the semi-Lagrangian rate estimate}
\label{app:proofs}

The proof transposes the corresponding Euclidean argument to the bi-invariant Riemannian structure of $(G,\dist_{G})$.
The only group specific facts used are those already collected in \cref{ssec:viscous_HJB}: the squared distance $\dist_{G}^{2}$ is smooth on the uniform diagonal neighbourhood $\mathcal U_{\iota}:=\{(g,h)\in G\times G \mid \dist_{G}(g,h)<\iota\}$, where $\iota>0$ is the injectivity radius (positive by compactness of $G$), and all constants below depend only on the data of \cref{ass:standing}, the horizon $T$, and curvature/injectivity-radius bounds of $G$. Recall that we write $\nabla^{\mathrm{body}}_{G}$ for the body-frame gradient \eqref{eq:body_gradient}.

\begin{proof}[Proof of \cref{thm:DPP_convergence}]
	Write $V_{k}:=\val_{\delta}(\cdot,t_{k})$, so that $V_{K}=\phi$ and $V_{k}=\mathcal T_{\delta}V_{k+1}$ for $k=K-1,\dots,0$.
	
	By \cref{prop:Tdelta_properties} the operator $\mathcal T_{\delta}$ is monotone, additively homogeneous, sup-norm non-expansive, and satisfies $\Lip(\mathcal T_{\delta}f)\le e^{L_{\xi}\delta}(L_{\ell}\delta+\Lip(f))$. Downward induction from $V_{K}=\phi$ gives, uniformly in $k$ and $\delta$,
	\[
	\|V_{k}\|_{\infty}\le \|\phi\|_{\infty}+T\|\ell\|_{\infty},
	\quad
	\Lip(V_{k})\le L_{1}:=e^{L_{\xi}T}\bigl(L_{\phi}+L_{\ell}T\bigr),
	\]
	and, since $\dist_{G}(\Phi^{u}_{\delta}(g),g)\le\delta\|\xi\|_{\infty}$, the discrete time-regularity $\|V_{k}-V_{k+1}\|_{\infty}\le\delta(\|\ell\|_{\infty}+L_{1}\|\xi\|_{\infty})$.

	Define $\mathcal F[\psi]:=-\partial_{t}\psi+\Ham(\cdot,\nabla^{\mathrm{body}}_{G}\psi)$, so that \eqref{eq:HJB} reads $\mathcal F[\val]=0$. Let $\psi$ be $C^{2}$ in space and time near the points involved. For a constant control $u$,
	\[
	\psi\bigl(\Phi^{u}_{\delta}(g),t\bigr)-\psi(g,t)
	=\int_{0}^{\delta}\bigl\langle\nabla^{\mathrm{body}}_{G}\psi\bigl(\Phi^{u}_{s}(g),t\bigr),\,\xi\bigl(\Phi^{u}_{s}(g),u\bigr)\bigr\rangle_{\Lie}\,\mathrm ds,
	\]
	and since $\dist_{G}(\Phi^{u}_{s}(g),g)\le s\|\xi\|_{\infty}$ while $\nabla^{\mathrm{body}}_{G}\psi$ and $\xi(\cdot,u)$ are Lipschitz on the relevant compact set, the integrand deviates from its value at $s=0$ by $Cs\|\psi\|_{C^{2}}$, uniformly in $u$. No differentiability of the flow in its initial condition is used. Adding $\delta\ell(g,u)$, minimizing over $u$, and expanding $\psi(g,t_{k})$ around $t_{k+1}$ in time, the residual $\rho_{k}[\psi]:=\psi(\cdot,t_{k})-\mathcal T_{\delta}\psi(\cdot,t_{k+1})$ satisfies
	\begin{equation}
		\bigl\|\rho_{k}[\psi]-\delta\,\mathcal F[\psi](\cdot,t_{k+1})\bigr\|_{\infty}\le C\delta^{2}\|\psi\|_{C^{2}}.
		\label{eq:app_trunc}
	\end{equation}
	If $\val\in C^{2}(G\times[0,T])$, then $\mathcal F[\val]=0$, and non-expansiveness of $\mathcal T_{\delta}$ gives $e_{k}\le e_{k+1}+C\delta^{2}\|\val\|_{C^{2}}$ for $e_{k}:=\|V_{k}-\val(\cdot,t_{k})\|_{\infty}$, whence the $O(\delta)$ rate.

	For merely Lipschitz $\val$ we regularize in space and time simultaneously, since a spatial regularization alone does not control the time expansion in \eqref{eq:app_trunc}. Recall that $\val$ is $L_{1}$-Lipschitz in space, after enlarging $L_{1}$, and $L_{2}$-Lipschitz in time, with $L_{2}:=\|\ell\|_{\infty}+L_{1}\|\xi\|_{\infty}$. For $\alpha>0$ define
	\[
	\val^{\alpha}(g,t):=\sup_{(h,s)\in G\times[0,T]}\Bigl\{\val(h,s)-\frac{\dist_{G}(g,h)^{2}+(t-s)^{2}}{2\alpha}\Bigr\},
	\]
	together with the inf-convolution $\val_{\alpha}$, defined with $\inf$ and $+$ in place of $\sup$ and $-$. The Lipschitz bounds give $0\le\val^{\alpha}-\val\le C_{2}\alpha$ and $0\le\val-\val_{\alpha}\le C_{2}\alpha$ with $C_{2}:=\tfrac12(L_{1}^{2}+L_{2}^{2})$, uniform space-time Lipschitz bounds for $\val^{\alpha}$ and $\val_{\alpha}$, and the localization
	$\dist_{G}(g,h^{*})\le C_{1}\alpha$, $|t-s^{*}|\le C_{1}\alpha$, $C_{1}:=2(L_{1}+L_{2})$, of every extremizing pair $(h^{*},s^{*})$ for the point $(g,t)$. We impose
	\begin{equation}
		0<\alpha\le\alpha_{0}<\min\Bigl\{1,\ \frac{\iota}{4C_{1}},\ \frac{T}{4C_{1}}\Bigr\},
		\qquad
		0<\delta\le\min\{\delta_{0},\alpha^{2}\},
		\label{eq:app_localization}
	\end{equation}
	with $\delta_{0}$ such that $C_{1}\alpha_{0}+\delta_{0}\|\xi\|_{\infty}<\iota$. Then all extremizing pairs lie in the smooth diagonal neighbourhood $\mathcal U_{\iota}$, one-step trajectories remain in it, and $s^{*}\in(0,T)$ whenever $C_{1}\alpha<t_{k}$ and $t_{k+1}<T-C_{1}\alpha$. Finally, for $\dist_{G}(g,h)<\iota$, differentiating $\dist_{G}(g\exp(sY),h\exp(sY))=\dist_{G}(g,h)$ at $s=0$ gives, in body coordinates,
	\begin{equation}
		-\frac{1}{2\alpha}(TL_{g})^{*}D_{g}\dist_{G}(g,h)^{2}
		=\frac{1}{2\alpha}(TL_{h})^{*}D_{h}\dist_{G}(g,h)^{2},
		\label{eq:app_covector}
	\end{equation}
	a common body covector $p$ with $\|p\|_{\Lie}=\dist_{G}(g,h)/\alpha$.
	
	Set $m_{k}:=\inf_{G}\bigl(V_{k}-\val^{\alpha}(\cdot,t_{k})\bigr)$. Since $V_{k+1}\ge\val^{\alpha}(\cdot,t_{k+1})+m_{k+1}$, monotonicity and additive homogeneity of $\mathcal T_{\delta}$ give $m_{k}\ge m_{k+1}-\sup_{G}\rho_{k}[\val^{\alpha}]$. Fix $k$ in the interior range, $g\in G$, and a maximizing pair $(h^{*},s^{*})$ for $\val^{\alpha}(g,t_{k})$, and set
	\[
	\psi(g',t'):=\val(h^{*},s^{*})-\frac{\dist_{G}(g',h^{*})^{2}+(t'-s^{*})^{2}}{2\alpha}.
	\]
	Then $\psi\le\val^{\alpha}$ with equality at $(g,t_{k})$, and $\psi$ is smooth on $\{\dist_{G}(\cdot,h^{*})<\iota\}$ with first derivatives $O(1)$ and second derivatives $O(1/\alpha)$, by the Hessian bound for $\dist_{G}^{2}$ on $\mathcal U_{\iota}$. Monotonicity and \eqref{eq:app_trunc} give
	\[
	\rho_{k}[\val^{\alpha}](g)\le\rho_{k}[\psi](g)\le\delta\,\mathcal F[\psi](g,t_{k+1})+C\frac{\delta^{2}}{\alpha}.
	\]
	Write $q:=-(t_{k}-s^{*})/\alpha$ and let $p$ be the covector of \eqref{eq:app_covector} at $(g,h^{*})$, so that $\partial_{t}\psi(g,t_{k+1})=q-\delta/\alpha$, $\nabla^{\mathrm{body}}_{G}\psi(g,\cdot)=p$ and $\|p\|_{\Lie}\le C_{1}$. The function $\varphi(h,s):=\val^{\alpha}(g,t_{k})+\bigl(\dist_{G}(g,h)^{2}+(t_{k}-s)^{2}\bigr)/(2\alpha)$ touches $\val$ from above at the interior point $(h^{*},s^{*})$, where $\partial_{s}\varphi=q$ and, again by \eqref{eq:app_covector}, the body covector is $p$. For every constant control $u$, with trajectory $\zeta$ issued from $(h^{*},s^{*})$, the dynamic programming principle and $\val\le\varphi$ give
	\[
	\varphi(h^{*},s^{*})=\val(h^{*},s^{*})\le\int_{s^{*}}^{s^{*}+\tau}\ell(\zeta_{r},u)\,\mathrm dr+\varphi(\zeta_{s^{*}+\tau},s^{*}+\tau),
	\]
	so dividing by $\tau$, letting $\tau\downarrow0$ and minimizing over $u$ yields the subsolution inequality $-q+\Ham(h^{*},p)\le0$ at the contact point. With the $g$-modulus \eqref{eq:Hamiltonian_g_modulus} and $\dist_{G}(g,h^{*})\le C_{1}\alpha$,
	\[
	\mathcal F[\psi](g,t_{k+1})=-q+\frac{\delta}{\alpha}+\Ham(g,p)\le C\alpha+\frac{\delta}{\alpha},
	\quad\text{hence}\quad
	\sup_{G}\rho_{k}[\val^{\alpha}]\le C\Bigl(\delta\alpha+\frac{\delta^{2}}{\alpha}\Bigr).
	\]

	For $t_{k}\ge T-2C_{1}\alpha$, the decomposition $V_{k}-\val^{\alpha}=(V_{k}-\val)-(\val^{\alpha}-\val)$, the discrete time-regularity $\|V_{k}-\val(\cdot,t_{k})\|_{\infty}\le C(T-t_{k})\le C\alpha$ and the sandwich bound give $m_{k}\ge-C\alpha$. For $t_{k}\le C_{1}\alpha$, at most $C_{1}\alpha/\delta+1$ indices, the uniform Lipschitz bounds of $\val^{\alpha}$ give the crude estimate $\rho_{k}[\val^{\alpha}]\le C\delta$, contributing $C\alpha$ in total. Iterating the recursion over the at most $K=T/\delta$ interior steps and using the sandwich once more,
	\[
	V_{k}-\val(\cdot,t_{k}) \ge -C\Bigl(\alpha+\frac{\delta}{\alpha}\Bigr),
	\qquad 0\le k\le K.
	\]
	
	With $M_{k}:=\sup_{G}(V_{k}-\val_{\alpha}(\cdot,t_{k}))$, monotonicity and additive homogeneity give $M_{k}\le M_{k+1}-\inf_{G}\rho_{k}[\val_{\alpha}]$. At a minimizing pair $(h^{*},s^{*})$ for $\val_{\alpha}(g,t_{k})$, the outer function $\widetilde\psi:=\val(h^{*},s^{*})+\bigl(\dist_{G}(\cdot,h^{*})^{2}+(\cdot-s^{*})^{2}\bigr)/(2\alpha)$ satisfies $\widetilde\psi\ge\val_{\alpha}$ with equality at $(g,t_{k})$, so that $\rho_{k}[\val_{\alpha}](g)\ge\rho_{k}[\widetilde\psi](g)\ge\delta\,\mathcal F[\widetilde\psi](g,t_{k+1})-C\delta^{2}/\alpha$, while the inner function $\widetilde\varphi:=\val_{\alpha}(g,t_{k})-\bigl(\dist_{G}(g,\cdot)^{2}+(t_{k}-\cdot)^{2}\bigr)/(2\alpha)$ touches $\val$ from below at $(h^{*},s^{*})$. Applying the dynamic programming principle with an $o(\tau)$-optimal control, expanding $\widetilde\varphi$ along the corresponding trajectory, dividing by $\tau$ and letting $\tau\downarrow0$ yields
	\[
	\widetilde q+\inf_{u\in U}\bigl\{\ell(h^{*},u)+\langle\widetilde p,\xi(h^{*},u)\rangle_{\Lie}\bigr\}\le0,
	\ \text{that is}\
	-\widetilde q+\Ham(h^{*},\widetilde p)\ge0,
	\]
	with $\widetilde q:=(t_{k}-s^{*})/\alpha$ and $\widetilde p$ the common covector of \eqref{eq:app_covector}. Since $\partial_{t}\widetilde\psi(g,t_{k+1})=\widetilde q+\delta/\alpha$ and $\nabla^{\mathrm{body}}_{G}\widetilde\psi(g,\cdot)=\widetilde p$, the same Hamiltonian modulus gives $\mathcal F[\widetilde\psi](g,t_{k+1})\ge-C\alpha-\delta/\alpha$, hence $\inf_{G}\rho_{k}[\val_{\alpha}]\ge-C(\delta\alpha+\delta^{2}/\alpha)$. The boundary layers are treated as before, and the backward iteration gives $V_{k}-\val(\cdot,t_{k})\le C(\alpha+\delta/\alpha)$ for all $k$.

	Both bounds hold under \eqref{eq:app_localization}, and $\alpha=\sqrt{\delta}$ is admissible, which gives $\sup_{k}\|V_{k}-\val(\cdot,t_{k})\|_{\infty}\le C\sqrt{\delta}$. This is the transplant to $(G,\dist_{G})$ of the Crandall-Lions rate estimate \cite{CrandallLions1984}, in the monotone-scheme form of \cite{Falcone1987,BardiCapuzzoDolcetta1997}.
\end{proof}

\section{Proof of the vanishing-viscosity estimate}
\label{app:VV_proof}

This appendix proves the vanishing-viscosity estimate, \cref{thm:VV}. The following almost-everywhere criterion is used in the last step of the proof.

\begin{lemma}[Almost-everywhere criterion for $W^{2,\infty}$ functions]
\label{lem:ae_viscosity}
Let $Q:=G\times(0,T)$ and $a>0$. Suppose that
$w\in W^{2,\infty}_{\mathrm{loc}}(Q)\cap C(Q)$ and set
\[
  \mathcal F_{a}[w]
  :=\partial_{\tau}w
  +\Ham\bigl(g,\nabla_{G}^{\mathrm{body}}w\bigr)
  -a\Lap_{G}w,
\]
where the derivatives are understood almost everywhere. If
$\mathcal F_{a}[w]\le0$ almost everywhere in $Q$, then $w$ is a
viscosity subsolution of $\mathcal F_{a}[w]=0$. If
$\mathcal F_{a}[w]\ge0$ almost everywhere, then $w$ is a viscosity
supersolution.
\end{lemma}

\begin{proof}
The assertion is local. In a smooth space-time chart, a
$W^{2,\infty}_{\mathrm{loc}}$ function has a $C^{1,1}_{\mathrm{loc}}$
representative and is both semiconvex and semiconcave. Suppose that
$\varphi\in C^{2}$ touches $w$ from above at $z_{0}$. Jensen's maximum
principle gives points $z_{n}\to z_{0}$ at which $w$ is twice
differentiable and
\[
  Dw(z_{n})\to D\varphi(z_{0}),
  \quad
  D^{2}w(z_{n})\le D^{2}\varphi(z_{0})+o(1).
\]
The coordinate expression of $\mathcal F_{a}$ is continuous and
degenerate elliptic in the spatial Hessian. Evaluating the
a.e. inequality at $z_{n}$ and passing to the limit therefore gives
the viscosity subsolution inequality at $z_{0}$. The supersolution
statement follows analogously from a test function touching from
below. This is the standard almost-everywhere characterization of
viscosity inequalities for semiconvex and semiconcave functions, see,
for example, \cite[Section~3]{CrandallIshiiLions1992}.
\end{proof}

\begin{proof}[Proof of \cref{thm:VV}]
We divide the proof into five steps.

\medskip
\noindent
\emph{Step 1. Time reversal, Hamiltonian moduli, and well-posedness.}
Set
\[
  u(g,\tau):=V(g,T-\tau),
  \quad
  u^{\eps}(g,\tau):=V^{\eps}(g,T-\tau),
  \quad 0\le\tau\le T.
\]
Then $u$ and $u^{\eps}$ have the common initial datum $\phi$ and solve,
in the viscosity sense,
\begin{align}
  \partial_{\tau}u
  +\Ham\bigl(g,\nabla_{G}^{\mathrm{body}}u\bigr)&=0,
  \label{eq:VV_forward_inviscid}\\
  \partial_{\tau}u^{\eps}
  +\Ham\bigl(g,\nabla_{G}^{\mathrm{body}}u^{\eps}\bigr)
  -\eps\Lap_{G}u^{\eps}&=0.
  \label{eq:VV_forward_viscous}
\end{align}
The definition of $\Ham$ and \cref{ass:standing} imply
\begin{align}
  |\Ham(g,p)-\Ham(g,q)|
  &\le \|\xi\|_{L^{\infty}(G\times U)}\|p-q\|_{\Lie},
  \label{eq:Hamiltonian_p_modulus}\\
  |\Ham(g,p)-\Ham(h,p)|
  &\le (L_{\xi}\|p\|_{\Lie}+L_{\ell})\dist_{G}(g,h).
  \label{eq:Hamiltonian_g_modulus}
\end{align}
For each fixed $\eps>0$, \eqref{eq:VV_forward_viscous} is uniformly
parabolic on the compact manifold $G$. The second-order viscosity
comparison principle gives uniqueness. Existence follows from
Perron's method, using the constant barriers
\[
  -\|\phi\|_{\infty}-\tau\|\ell\|_{\infty}
  \quad\text{and}\quad
  \|\phi\|_{\infty}+\tau\|\ell\|_{\infty},
\]
or from the stochastic-control representation used below.

\medskip
\noindent
\emph{Step 2. A spatial Lipschitz estimate uniform in $\eps$.}
Let $\{X_{1},\ldots,X_{d}\}$ be an orthonormal basis of $\Lie$ and let
$\widetilde X_{j}(g):=(TL_{g})X_{j}$. The controlled Stratonovich
system
\begin{equation}
  dX_{s}
  =(TL_{X_{s}})\xi(X_{s},a_{s})\,ds
  +\sqrt{2\eps}\sum_{j=1}^{d}\widetilde X_{j}(X_{s})\circ dW_{s}^{j}
  \label{eq:VV_controlled_SDE}
\end{equation}
has generator
\[
  f\mapsto
  \langle\nabla_{G}^{\mathrm{body}}f,\xi(\cdot,a)\rangle_{\Lie}
  +\eps\Lap_{G}f.
\]
The standard stochastic dynamic-programming principle therefore
identifies $u^{\eps}$ with
\begin{equation}
  \inf_{a}\mathbb E\left[
    \int_{0}^{\tau}\ell(X_{s},a_{s})\,ds+\phi(X_{\tau})
  \right].
  \label{eq:VV_stochastic_value}
\end{equation}
For $\eps=0$ this reduces to the deterministic value $u$.

Couple two systems starting at $g$ and $h$ by using the same admissible
control and Brownian motion. The noise vector
\[
  \bigl((TL_{g})X_{j},(TL_{h})X_{j}\bigr)
\]
is tangent to the simultaneous right translation
$(g,h)\mapsto(g\exp(sX_{j}),h\exp(sX_{j}))$. Since the metric is
bi-invariant, this diagonal action preserves $\dist_{G}$. Decompose
the drift as
\begin{align*}
 &\bigl((TL_{g})\xi(g,a),(TL_{h})\xi(h,a)\bigr)\\
 &\qquad=
 \bigl((TL_{g})\xi(g,a),(TL_{h})\xi(g,a)\bigr)
 +\bigl(0,(TL_{h})(\xi(h,a)-\xi(g,a))\bigr).
\end{align*}
The first pair is again tangent to an isometric diagonal orbit. The
Stratonovich chain rule away from the cut locus, together with the
standard Calabi barrier argument at the cut locus (see, e.g., \cite{Kendall1986} for this device in the coupling context), consequently gives
\[
  D^{+}\dist_{G}(X_{s}^{g},X_{s}^{h})
  \le L_{\xi}\dist_{G}(X_{s}^{g},X_{s}^{h}).
\]
Hence, pathwise,
\begin{equation}
  \dist_{G}(X_{s}^{g},X_{s}^{h})
  \le e^{L_{\xi}s}\dist_{G}(g,h).
  \label{eq:VV_synchronous_distance}
\end{equation}
Using the same control in the two cost functionals, then taking the
infimum and interchanging $g$ and $h$, yields
\begin{equation}
  |u^{\eps}(g,\tau)-u^{\eps}(h,\tau)|
  \le \Lambda(\tau)\dist_{G}(g,h),
  \quad
  \Lambda(\tau)
  :=L_{\phi}e^{L_{\xi}\tau}
  +L_{\ell}\int_{0}^{\tau}e^{L_{\xi}s}\,ds.
  \label{eq:VV_uniform_Lip_eps}
\end{equation}
The same estimate holds for $u$ by taking $\eps=0$. Since $\Lambda(T)\le e^{L_{\xi}T}(L_{\phi}+L_{\ell}T)=L$, this proves the uniform Lipschitz bound \eqref{eq:VV_global_Lip_constant} asserted in the theorem.

\medskip
\noindent
\emph{Step 3. The squared-distance penalization.}
Let
\[
  \Theta(g,h):=\frac12\dist_{G}(g,h)^{2}.
\]
Fix $r_{0}$ strictly smaller than the injectivity radius of $G$. On
$\mathcal U_{r_{0}}:=\{(g,h):\dist_{G}(g,h)<r_{0}\}$, the function
$\Theta$ is smooth and
\begin{equation}
  \|D^{2}_{g,h}\Theta\|\le C_{\Theta}
  \quad\text{on }\mathcal U_{r_{0}}
  \label{eq:VV_hessian_rho_bound}
\end{equation}
for a geometric constant $C_{\Theta}$.

Bi-invariance gives
\[
  \Theta(g\exp(sY),h\exp(sY))=\Theta(g,h)
  \qquad(Y\in\Lie).
\]
Differentiating at $s=0$ shows that the body covectors generated by
$\Theta/\alpha$,
\[
  p_{g}:=\frac1\alpha(TL_{g})^{*}D_{g}\Theta,
  \quad
  p_{h}:=-\frac1\alpha(TL_{h})^{*}D_{h}\Theta,
\]
coincide in $\Lie^{*}\simeq\Lie$:
\begin{equation}
  p_{g}=p_{h}=:p,
  \quad
  \|p\|_{\Lie}=\frac{\dist_{G}(g,h)}{\alpha}.
  \label{eq:VV_body_covectors_equal}
\end{equation}
This identity is the group-specific ingredient that permits the use of
\eqref{eq:Hamiltonian_g_modulus} with one common body covector.

We shall also use a direct consequence of the parabolic theorem of
sums (see \cite{CrandallIshiiLions1992} and, for its Riemannian form, \cite{AzagraFerreraSanz2008}). At a maximum point involving the penalization $\Theta/\alpha$, it
provides symmetric forms $X$ on $T_{g}G$ and $Y$ on $T_{h}G$ such that
\[
  \begin{pmatrix}X&0\\0&-Y\end{pmatrix}
  \le A+\mu A^{2},
  \quad
  A:=D^{2}_{g,h}(\Theta/\alpha).
\]
Taking $\mu=\alpha$ and using \eqref{eq:VV_hessian_rho_bound}, we obtain
\begin{equation}
  \Tr X\le\frac{C_{D}}{\alpha},
  \quad
  -\Tr Y\le\frac{C_{D}}{\alpha},
  \label{eq:VV_trace_bounds}
\end{equation}
where $C_{D}$ depends only on $G$ and $r_{0}$.

\medskip
\noindent
\emph{Step 4. Comparison of $u$ and $u^{\eps}$.}
For $\alpha,\lambda,\sigma>0$, consider
\[
  \Psi(g,h,\tau)
  :=u(g,\tau)-u^{\eps}(h,\tau)
  -\frac{\Theta(g,h)}{\alpha}-\lambda\tau-\frac{\sigma}{T-\tau}.
\]
Let $(\bar g,\bar h,\bar\tau)$ be a maximum point, and note that the last penalization forces $\bar\tau<T$. Comparing it with
$(\bar g,\bar g,\bar\tau)$ and using \eqref{eq:VV_uniform_Lip_eps}
gives
\begin{equation}
  \bar d:=\dist_{G}(\bar g,\bar h)\le2L\alpha.
  \label{eq:VV_distance_at_max}
\end{equation}
Choose
\[
  \alpha_{0}:=\frac{r_{0}}{4\max\{L,1\}}.
\]
Then $0<\alpha\le\alpha_{0}$ ensures that the maximum lies in
$\mathcal U_{r_{0}}$.

If $\bar\tau>0$, the parabolic theorem of sums gives $a-b=\lambda+\sigma(T-\bar\tau)^{-2}\ge\lambda$,
the common body covector $p$ in \eqref{eq:VV_body_covectors_equal},
with $\|p\|\le2L$, and
\[
  a+\Ham(\bar g,p)\le0,
  \quad
  b+\Ham(\bar h,p)-\eps\Tr Y\ge0.
\]
Therefore, by \eqref{eq:Hamiltonian_g_modulus},
\eqref{eq:VV_distance_at_max}, and \eqref{eq:VV_trace_bounds},
\begin{align}
  \lambda
  &\le \Ham(\bar h,p)-\Ham(\bar g,p)-\eps\Tr Y\notag\le C_{H}\alpha+C_{D}\frac{\eps}{\alpha},
  \label{eq:VV_interior_inequality}
\end{align}
where one may take
\[
  C_{H}:=2L(2LL_{\xi}+L_{\ell}).
\]
Choosing
$\lambda=C_{H}\alpha+C_{D}\eps/\alpha+\eta$ with $\eta>0$
rules out an interior maximum. At $\tau=0$,
\[
  \Psi(g,h,0)
  \le L_{\phi}\dist_{G}(g,h)
       -\frac{\dist_{G}(g,h)^{2}}{2\alpha}
  \le\frac12L_{\phi}^{2}\alpha.
\]
Taking $h=g$, then the supremum over $(g,\tau)$, and letting first $\sigma\downarrow0$, by continuity of $u-u^{\eps}$ up to $\tau=T$, and then $\eta\downarrow0$, gives
\begin{equation}
  \sup_{G\times[0,T]}(u-u^{\eps})
  \le
  \left(\frac12L_{\phi}^{2}+C_{H}T\right)\alpha
  +C_{D}\frac{\eps T}{\alpha}.
  \label{eq:VV_one_sided_bound}
\end{equation}
For the reverse inequality, maximize
\[
  u^{\eps}(g,\tau)-u(h,\tau)
  -\frac{\Theta(g,h)}{\alpha}-\lambda\tau-\frac{\sigma}{T-\tau}.
\]
The same argument uses the estimate $\Tr X\le C_{D}/\alpha$ in
\eqref{eq:VV_trace_bounds} and yields the same right-hand side.
Consequently,
\begin{equation}
  \|u-u^{\eps}\|_{L^{\infty}(G\times[0,T])}
  \le A_{T}\alpha+C_{D}\frac{\eps T}{\alpha},
  \quad
  A_{T}:=\frac12L_{\phi}^{2}+C_{H}T,
  \label{eq:VV_two_sided_alpha}
\end{equation}
which proves \eqref{eq:VV_rate_alpha} with $B_{G}:=C_{D}$.

Whenever
$\alpha_{*}:=(C_{D}\eps T/A_{T})^{1/2}\le\alpha_{0}$, optimizing at
$\alpha=\alpha_{*}$ gives
\[
  \|u-u^{\eps}\|_{\infty}
  \le2\sqrt{A_{T}C_{D}\eps T}.
\]
For the remaining $\eps\in(0,1]$, the same form of bound follows after
enlarging the constant, using the uniform boundedness of the two value
functions. Reversing time proves \eqref{eq:VV_rate}.

\medskip
\noindent
\emph{Step 5. The $O(\eps)$ estimate under $W^{2,\infty}$ regularity.}
Assume
$V\in W^{2,\infty}(G\times(0,T))\cap C(G\times[0,T])$ and retain the
time-reversed function $u$. Its canonical representative is
$C^{1,1}_{\mathrm{loc}}$ in space-time. Since $u$ is a viscosity
solution of the first-order equation, the equation
\[
  \partial_{\tau}u
  +\Ham\bigl(g,\nabla_{G}^{\mathrm{body}}u\bigr)=0
\]
holds at every differentiability point, hence almost everywhere.
Set
\[
  M:=\|\Lap_{G}u\|_{L^{\infty}(G\times(0,T))}
  =\|\Lap_{G}V\|_{L^{\infty}(G\times(0,T))}
\]
and define
\[
  u_{-}(g,\tau):=u(g,\tau)-\eps M\tau,
  \quad
  u_{+}(g,\tau):=u(g,\tau)+\eps M\tau.
\]
Almost everywhere,
\begin{align*}
  \partial_{\tau}u_{-}
  +\Ham\bigl(g,\nabla_{G}^{\mathrm{body}}u_{-}\bigr)
  -\eps\Lap_{G}u_{-}
  &=-\eps(M+\Lap_{G}u)\le0,\\
  \partial_{\tau}u_{+}
  +\Ham\bigl(g,\nabla_{G}^{\mathrm{body}}u_{+}\bigr)
  -\eps\Lap_{G}u_{+}
  &=\eps(M-\Lap_{G}u)\ge0.
\end{align*}
By \cref{lem:ae_viscosity}, $u_{-}$ is a viscosity subsolution and
$u_{+}$ is a viscosity supersolution of
\eqref{eq:VV_forward_viscous}. Since all three functions have initial
datum $\phi$, comparison gives
\[
  u(g,\tau)-\eps M\tau
  \le u^{\eps}(g,\tau)
  \le u(g,\tau)+\eps M\tau.
\]
Taking the supremum over $0\le\tau\le T$ and reversing time gives the
first inequality in \eqref{eq:VV_smooth_rate}. The second follows from
the local coordinate expression of $\Lap_{G}$ and compactness of $G$.
\end{proof}

\section{Additional experiments}
\label{app:experiments}

This appendix collects two diagnostics of Example~1, the refinement study under the admissible couplings and the smooth-regime measurements of the viscous bias and of the rate in the time step.

\subsection{Accuracy under admissible couplings}
\label{sssec:ex1_highres}

As noted in \cref{sssec:ex1_canonical}, the canonical-coupling errors are dominated by the accumulated filter bias, while fewer and longer steps at the same resolution cost no accuracy on this problem. We refine along the member $\nu_{\delta}=1/2$, taking
\[
\delta=c_{0}N^{-1/2},\qquad K=\max\bigl\{3,\ \mathrm{round}(T\sqrt N/c_{0})\bigr\},\qquad c_{0}=2,\quad\eps=0,
\]
which keeps $\delta\to0$ and $\delta N^{2}\to\infty$ as $N$ grows, so that \cref{thm:finite_rank_lipschitz_rate} still guarantees convergence. Over the measured range the rounding leaves $K=3$ for the Fej\'er-Markov runs, which therefore isolate the decay of the filter bias at essentially fixed time step. The control set is refined to $400$ Fibonacci directions, and the directional contribution $u_{\max}T(1-\cos r_{\mathrm{cov}})$ drops to about $0.014$.

\Cref{tab:ex1_highres} reports the resulting errors for the finite-rank Fej\'er-Markov filter, whose kernel scale is $N_{f}\approx L/4$, and for the heat filter at its natural scale $N=L$. Both iterations preserve the bound $\min_{g}U_{k}\ge0$ at every band and at zero viscosity. The sup error decays with fitted order $0.86$ in $N$ for the heat filter and $0.99$ in $N_{f}$ for the Fej\'er-Markov filter, in line with the bias scale $\sqrt{K+1}/N$, and reaches $5.7\cdot10^{-2}$ in the sup norm and $2.2\cdot10^{-2}$ in $L^{2}$ at $L=64$ with six Bellman steps. When plotted against the kernel scale, the two filter families collapse onto a single curve (\cref{fig:ex1_highres}), so the accuracy is governed by the kernel scale and the number of filter applications alone. The heat filter is not finite-rank and its sampled realization is in principle subject to aliasing. Recomputing the $L=16$ row with the band-$2L$ quadrature of \cref{prop:discrete_markov_filter} changes the sup error by $2\cdot10^{-3}$, which is negligible at this accuracy.

\begin{table}[t]
\centering
\caption{Example 1. Sup and $L^{2}$ errors over the full grid at $t=0$ against the exact solution \eqref{eq:ex1_eikonal_V}, under the admissible coupling $\delta=c_{0}N^{-1/2}$ of \eqref{eq:admissible_scaling} with $c_{0}=2$ and $\eps=0$, using $401$ control candidates. The kernel scale is $N_{f}\approx L/4$ for the Fej\'er-Markov filter and $N=L$ for the heat filter.}
\label{tab:ex1_highres}
\small
\begin{tabular}{rrrcc@{\qquad}rcc}
\toprule
& \multicolumn{4}{c}{Fej\'er-Markov ($N=N_{f}$)} & \multicolumn{3}{c}{heat filter ($N=L$)}\\
\cmidrule(lr){2-5}\cmidrule(lr){6-8}
$L$ & $N_{f}$ & $K$ & sup & $L^{2}$ & $K$ & sup & $L^{2}$\\
\midrule
16 &  4.24 & 3 & 0.512 & 0.268 & 3 & 0.190 & 0.085\\
24 &  6.24 & 3 & 0.343 & 0.166 & 4 & 0.134 & 0.056\\
32 &  8.25 & 3 & 0.260 & 0.121 & 4 & 0.103 & 0.042\\
48 & 12.25 & 3 & 0.176 & 0.078 & 5 & 0.073 & 0.028\\
64 & 16.25 & 3 & 0.135 & 0.059 & 6 & 0.057 & 0.022\\
\bottomrule
\end{tabular}
\end{table}

\begin{figure}[t]
\centering
\includegraphics[width=0.66\linewidth]{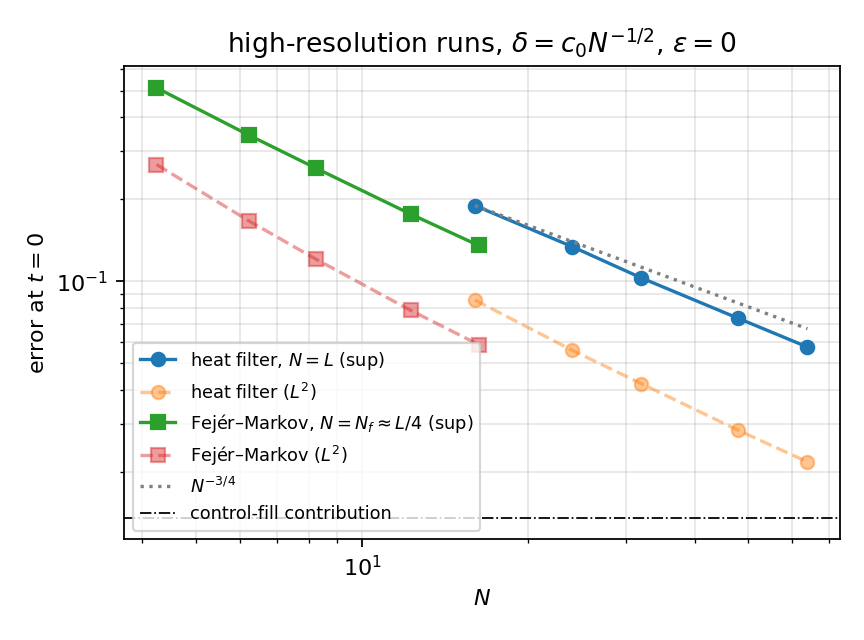}
\caption{Example 1. The errors of \cref{tab:ex1_highres} plotted against the kernel scale, equal to $N_{f}$ for the Fej\'er-Markov filter and to $L$ for the heat filter. The two families collapse onto a single curve. The dash-dotted line marks the directional contribution $u_{\max}T(1-\cos r_{\mathrm{cov}})$.}
\label{fig:ex1_highres}
\end{figure}

\subsection{Viscous bias and rate in the time step}
\label{app:smooth_rates}

On the commuting problem of \cref{sssec:ex1_saturation}, the viscous bias computed from \eqref{eq:ex1_exact} is
\begin{gather*}
\norm{V^{\eps}-V}_{\infty}=0.964,\ 0.603,\ 0.345,\ 0.186,\ 0.097\\
\text{for}\qquad
\eps=0.2,\ 0.1,\ 0.05,\ 0.025,\ 0.0125,
\end{gather*}
with local orders rising to $0.94$, the first-order smooth-regime rate of \cref{thm:VV} rather than the generic $\sqrt{\eps}$. On the minimum-effort regulation problem $\dot g=g\widehat u$ with $|u|\le2$, $\ell(u)=\frac14|u|^{2}$ and $\phi=1-\cos\theta$, whose terminal cost is band-limited at $\ell=1$, refining the time step at matched band $L=16$ and matched total filter exposure, with $L_{N}\propto\sqrt K$, gives sup errors
\[
2.5\cdot10^{-2},\ 1.1\cdot10^{-2},\ 4.8\cdot10^{-3},\ 1.6\cdot10^{-3}
\quad\text{for}\quad K=10,20,40,80.
\]
This is first order in $\delta$. It is the expected rate on this smooth problem, since one IMEX step is a Lie-Trotter splitting of the viscous flow with local error $O(\delta^{2})$, and the non-expansiveness of the step makes the local errors accumulate at most additively over the $K=T/\delta$ steps.

\end{document}